\documentclass[12pt]{article}
\usepackage{extarrows}
\usepackage{indentfirst}
\usepackage[francais]{babel}
\usepackage[all]{xy}
\usepackage{amssymb}
\usepackage{amsthm}
\usepackage{amsmath}
\usepackage{mathrsfs}
\usepackage[top = 1 in, bottom = 1 in, left = 0.1 in, right = 0.1 in]{geometry}
\usepackage[enableskew]{youngtab}
\usepackage{textcomp}
\usepackage[colorlinks, linkcolor=black,anchorcolor=blue,citecolor=red]{hyperref}
\usepackage{titletoc}

\makeatletter 
\@addtoreset{equation}{section}
\makeatother  
\renewcommand\theequation{\oldstylenums{\thesection}%
                   .\oldstylenums{\arabic{equation}}}

\theoremstyle{plain}
\newtheorem{theorem}{Theorem}
\newtheorem{prop}[theorem]{Proposition}
\newtheorem{cor}[theorem]{Corollary}
\newtheorem{lem}[theorem]{Lemma}

\theoremstyle{definition}
\newtheorem{definition}[theorem]{Definition}
\newtheorem{notation}[theorem]{Notation}

\newtheorem{example}[theorem]{Example}

\newtheorem{remark}[theorem]{Remark}
\include{header}

\numberwithin{theorem}{section}        

\def\super{\mathbb{Z}_2}
\def\even{\overline{0}}
\def\odd{\overline{1}}
\def\stensor{\underline{\otimes}}

\def\category{\mathcal{O}_{\mathrm{int}}}
\def\weight{\mathrm{wt}}
\def\End{\mathrm{End}}
\def\Alg{\mathrm{Alg}}
\def\Kac{\textbf{K}}
\def\Verma{\textbf{M}}
\def\Weyl{\textbf{W}}
\def\simple{\textbf{S}}
\def\str{\mathrm{str}}

\def\BZ{\mathbb{Z}}  
\def\BC{\mathbb{C}}

\def\Ud{U_q(\mathcal{L}\mathfrak{sl}(M,N))}
\def\Uc{U_q'(\widehat{\mathfrak{sl}(M,N)})}

\title{Representations of Quantum Affine Superalgebras} 
\author{
{\normalsize Huafeng Zhang }
}
\date{}
\allowdisplaybreaks
\begin{document}
\newcommand{\titlefont}{10 pt}
\maketitle

\renewcommand\proofname{Proof}
\renewcommand\contentsname{Contents}
\renewcommand\refname{References}
\renewcommand\abstractname{Abstract}
\abstract{\begin{footnotesize}
We study the quantum affine superalgebra $\Ud$ and its finite-dimensional representations. We prove a triangular decomposition and  establish a system of Poincar\'{e}-Birkhoff-Witt generators for this superalgebra, both in terms of Drinfel'd currents. We define the Weyl modules in the spirit of Chari-Pressley and  prove that these Weyl modules are always finite-dimensional and non-zero. In consequence, we obtain a highest weight classification  of finite-dimensional simple representations when $M \neq N$. Some concrete simple representations are constructed via evaluation morphisms. 
\end{footnotesize} }

\titlecontents{section}[1em]{\normalsize}{\contentslabel{1em}}{}{%
                \titlerule*[0.9pc]{$\cdot$}\contentspage}
\titlecontents{subsection}[4em]{\normalsize}{\contentslabel{2em}}{}{%
                \titlerule*[0.5pc]{$\cdot$}\contentspage}
\tableofcontents                
\section{Introduction}
In this paper $q \in \mathbb{C} \setminus \{0\}$ is not a root of unity and our ground field is always  $\mathbb{C}$. We study a quantized version of the enveloping algebra of the affine Lie superalgebra $\mathcal{L}\mathfrak{sl}(M,N)$, which we denote by $U_q(\mathcal{L}\mathfrak{sl}(M,N))$.
\paragraph{Some properties of $\Ud$.}
For $M,N \in \mathbb{Z}_{\geq 1}$, the quantum affine superalgebra $\Ud$ is defined in terms of Drinfel'd currents. It is the superalgebra with 
\begin{itemize}
\item[(1)] Drinfel'd generators $X_{i,n}^{\pm}, h_{i,s}, K_i^{\pm 1}$ for $1 \leq i \leq M+N-1, n \in \BZ, s \in \BZ_{\neq 0}$;
\item[(2)] $\super$-grading $|X_{M,n}^{\pm}| = \odd$, and $|X_{i,n}^{\pm}| = |K_i^{\pm 1}| = |h_{i,s}| = |K_M^{\pm 1}| = |h_{M,s}| = \even$ for $i \neq M$;
\item[(3)] defining relations \eqref{rel: Cartan}-\eqref{rel: ocillation4} (see \S \ref{sec: Drinfeld presentation} for details).
\end{itemize}
Informally, when $q=1$,  $\Ud$ can be thought of as the universal enveloping algebra of the Lie superalgebra $\mathcal{L}\mathfrak{sl}(M,N) := \mathfrak{sl}(M,N) \otimes \mathbb{C}[t,t^{-1}]$  with the convention that
\begin{align*}
&X_{i,n}^+ = E_{i,i+1} t^n,\ \ X_{i,n}^- = E_{i+1,i} t^n,\ h_{i,s} = (E_{i,i} - (-1)^{\delta_{i,M}} E_{i+1,i+1}) t^s. 
\end{align*}
Let $U_q^{\pm}(\mathcal{L}\mathfrak{sl}(M,N))$ (resp.  $U_q^0(\mathcal{L}\mathfrak{sl}(M,N))$) be the subalgebra of $\Ud$ generated by the $X_{i,n}^{\pm}$ (resp. the  $K_i^{\pm 1}, h_{i,s}$). Then the Chevalley relations imply that 
\begin{displaymath}
\Ud = U_q^-(\mathcal{L}\mathfrak{sl}(M,N)) U_q^0(\mathcal{L}\mathfrak{sl}(M,N))U_q^+(\mathcal{L}\mathfrak{sl}(M,N))
\end{displaymath}
and that $U_q^0(\mathcal{L}\mathfrak{sl}(M,N))$ is a commutative algebra. 

When $M \neq N$, it is shown in \cite[Theorem 6.8.2]{Yam2} that $\Ud$ has a Chevalley presentation and is equipped with a Hopf superalgebra structure. Using the coproduct, we can form the tensor product of two representations of $\Ud$.  Note however that the coproduct formulae for $X_{i,n}^{\pm}, h_{i,s}$ are highly non-trivial.  
 
\paragraph{Backgrounds.}
In analogy with the applications of quantum affine algebras in solvable lattice models \cite{Jimbo}, quantum affine superalgebras also appear as the algebraic supersymmetries of some solvable models.  In  \cite{BT}, the quantum affine superalgebra $U_q(\mathcal{L}\mathfrak{sl}(2,1))$, together with its universal $R$-matrix, which exists in the framework of Khoroshkin-Tolstoy \cite{KT1,KT},   was used to define the $\textbf{Q}$-operators and to deduce their functional relations. These $\textbf{Q}$-operators were then applied in integrable models of statistic mechanics (3-state $\mathfrak{gl}(2,1)$-Perk-Schultz model) and the associated quantum field theory. Here the functional relations come essentially from the tensor product decompositions of representations of $U_q(\mathcal{L}\mathfrak{sl}(2,1))$ and its Borel subalgebras.

When $M \neq N$, $U_q(\widehat{\mathfrak{sl}(M,N)})$ (extended $\Ud$ with derivation) is the quantum supersymmetry analogue of the supersymmetric $t-J$ model (with or without a boundary). A key problem is to diagonalize the commuting transfer matrices. In \cite{Kojima} for example, Kojima proposed a construction of the boundary state using the machinery of algebraic analysis method. There to obtain the bosonization of the vertex operators \cite{boson}, one needs to work in some highest weight Fock representations of $U_q(\widehat{\mathfrak{sl}(M,N)})$.  

In the case $M = N = 2$, the Lie superalgebra $\mathfrak{sl}(2,2)$ admits a two-fold non-trivial central extension.  In \cite{Beisert1}, using the quantum deformation of this centrally extended algebra and its fundamental representations (which exist infinitely), Beisert-Koroteev deduced Shastry's spectral $R$-matrix $R(u,v)$.  Also it is found \cite{Bei} that the $S$-matrix of AdS/CFT enjoys a symmetry algebra: the {\it conventional} Yangian associated to the centrally extended algebra.  Later, \cite{Beisert2} derived a quantum affine superalgebra $\widehat{\mathcal{Q}}$ depending essentially on two parameters, together with a Hopf superalgebra structure and its fundamental representations of dimension 4. This algebra is interesting itself as it is explained there to have two conventional \lq\lq limits\rq\rq: one is $U_q(\widehat{\mathfrak{sl}(2,2)})$ ; the other is the Yangian limit. The two limiting processes carry over to the fundamental representations. Higher representations of this algebra are however still missing. 

It is therefore worthwhile to study quantum (affine) superalgebras, Yangians, and their representations.  For a symmetrizable quantum affine superalgebra $U_q(\mathfrak{g})$, early in 1997, Ruibin Zhang has classified integrable irreducible highest weight representations \cite{Zh4}(here being {\it symmetrizable}  excludes the existence of simple isotopic odd roots). Recently, in  \cite{Wang, Kashiwara}, the authors obtained a (super)categorification of some quantum symmetrizable Kac-Moody superalgebras and their integrable highest weight modules from quiver Hecke superalgebras. However, the affine Lie superalgebras $\mathcal{L}\mathfrak{sl}(M,N)$ are not symmetrizable, as they contain simple isotopic odd  roots. It is desirable to study $\Ud$ and their representations. 

In the paper \cite{Zh3}, Zhang considered the $\mathfrak{gl}(M,N)$ super Yangian and its finite-dimensional representations. The super Yangian $Y(\mathfrak{gl}(M,N))$ can be viewed as a deformation of the universal enveloping superalgebra  $U(\mathfrak{gl}(M,N)\otimes \mathbb{C}[t])$. Zhang equipped the super Yangian with a Hopf superalgebra structure and wrote explicitly a Poincar\'{e}-Birkhoff-Witt (PBW for short) basis. From this PBW basis one reads a triangular decomposition. Zhang proved that all finite-dimensional representations of $Y(\mathfrak{gl}(M,N))$ are of highest weight with respect to this triangular decomposition, and parametrised these highest weights by polynomials (see \S \ref{sec: final section} below). The aim of this paper is to develop a similar highest weight representation theory for some quantum affine superalgebras. 

We remark that Zhang's proof of the classification result relied on the coproduct structure $\Delta$ and on some superalgebra automorphisms $\phi_s$ of the super Yangian. For the quantum affine superalgebra $U_q(\mathcal{L}\mathfrak{sl}(M,N))$ defined in terms of Drinfel'd generators, the coproduct structure is highly non-trivial (its existence is not clear a priori), and we do not have the analogue of the automorphisms $\phi_s$. To overcome such difficulties we propose the PBW argument in this paper, which is independent of coproduct structures.

\paragraph{Main results.} In this paper, we study  finite-dimensional representations of the quantum affine superalgebras $U_q(\mathcal{L}\mathfrak{sl}(M,N))$ for $M, N \in \BZ_{>0}$ (possibly $M = N$).
First, we prove the Drinfel'd type triangular decomposition.

\noindent \textbf{\underline{Theorem \ref{thm: triangular decomposition}}}. {\it The following multiplication map is an isomorphism of vector superspaces:
\[ U_q^-(\mathcal{L}\mathfrak{sl}(M,N)) \stensor U_q^0(\mathcal{L}\mathfrak{sl}(M,N)) \stensor U_q^+(\mathcal{L}\mathfrak{sl}(M,N)) \longrightarrow \Ud,\ \ a \stensor b \stensor c \mapsto abc. \] Furthermore, the three subalgebras above admit presentations as superalgebras. }

With respect to this triangular decomposition, we can define the  Verma modules $\Verma(\Lambda)$, which are parametrised by the  linear characters $\Lambda$ on $U_q^0(\mathcal{L}\mathfrak{sl}(M,N))$, and are isomorphic to $U_q^-(\mathcal{L}\mathfrak{sl}(M,N))$ as vector superspaces. These Verma modules are important as it is shown that when $M \neq N$, all finite-dimensional simple $\Ud$-modules are their quotients up to modification by one-dimensional modules. We are led to consider the existence of finite-dimensional non-zero quotients of $\Verma(\Lambda)$, the so-called {\it modules of highest weight $\Lambda$}.


Let $V$ be a finite-dimensional quotient of $\Verma(\Lambda)$, with a non-zero even highest weight vector $v_{\Lambda}$. When $1 \leq i \leq M+N-1$ and $i \neq M$, the subalgebra $\widehat{U}_i$ generated by $X_{i,n}^{\pm}, K_i, h_{i,s}$ for $n \in \BZ, s \in \BZ_{\neq 0}$ is isomorphic to $U_{q_i}(\mathcal{L}\mathfrak{sl}_2)$. As $\widehat{U}_i v_{\Lambda}$ is finite-dimensional, from the highest weight representation theory of $U_q(\mathcal{L}\mathfrak{sl}_2)$ we conclude that there exists a Drinfel'd polynomial $P_i \in 1 + z\mathbb{C}[z]$ such that (see \S \ref{sec: 2.1} below for the $\phi_{i,n}^{\pm}$) 
\begin{align}
\sum_{n \in \BZ}\phi_{i,n}^{\pm} z^n v_{\Lambda} = q_i^{\deg P_i} \frac{P_i(z q_i^{-1})}{P_i(z q_i)}v_{\Lambda} \in V[[z^{\pm 1}]]\ \ \mathrm{and}\ (X_{i,0}^-)^{1 + \deg P_i} v_{\Lambda} = 0. \label{equ:Verma-Weyl 1}
\end{align}
On the other hand, the subalgebra $\widehat{U}_M$ is no longer $U_q(\mathcal{L}\mathfrak{sl}_2)$, but a superalgebra with simpler structure. As $\widehat{U}_M v_{\Lambda}$ is finite-dimensional, we can eventually find another Drinfel'd polynomial 
\begin{align} 
Q(z) = \sum_{s=0}^d a_s z^s \in 1  + z \mathbb{C}[z]\ \ \mathrm{such\ that}\ \sum_{s=0}^d a_s X_{M,d-s}^- v_{\Lambda} = 0.    \label{equ:Verma-Weyl 2}
\end{align}
Now let us set
\begin{align}
\Lambda (K_M) = c,\ \Lambda(\sum_{n \in \BZ} \frac{\phi_{M,n}^+ - \phi_{M,n}^-}{q-q^{-1}} z^n) = \sum_{n \in \BZ} f_n z^n = f(z) \in \mathbb{C}[[z,z^{-1}]], \label{equ:Verma-Weyl 3}
\end{align}
then we see that
\begin{align}
f_0 = \frac{c-c^{-1}}{q-q^{-1}},\ \ \ Q(z)f(z) = 0. \label{equ:Verma-Weyl 4}
\end{align}
The linear character $\Lambda$ is completely determined by  $\underline{P} = (P_i)$ and $(f,c)$ in view of Equations \eqref{equ:Verma-Weyl 1}-\eqref{equ:Verma-Weyl 4}.  We come out with the set $\mathcal{R}_{M,N}$ of highest weights consisting of $\Lambda = (\underline{P},f,c)$ such that there exists $Q(z)$ satisfying Relations \eqref{equ:Verma-Weyl 1}-\eqref{equ:Verma-Weyl 4}. For such $(\Lambda, Q)$, motivated by the theory of Weyl modules for quantum affine algebras \cite{CP3},  we define the Weyl module $\Weyl(\Lambda;Q)$ as the quotient of $\Verma(\Lambda)$ by Relations \eqref{equ:Verma-Weyl 1}-\eqref{equ:Verma-Weyl 2}. Hence {\it all finite-dimensional non-zero quotients of $\Verma(\Lambda)$, if exist, should be quotients of $\Weyl(\Lambda;Q)$ for some $Q$.} The sufficiency of restrictions \eqref{equ:Verma-Weyl 1}-\eqref{equ:Verma-Weyl 4} on the linear characters is guaranteed by

\noindent \textbf{\underline{Theorem \ref{thm: main}}}. {\it For all $\Lambda = (\underline{P},f,c) \in \mathcal{R}_{M,N}$ and $Q \in 1 + z \mathbb{C}[z]$ such that $Qf = 0$, $\deg Q < \dim \Weyl(\Lambda;Q) < \infty$.}

In consequence, when $M \neq N$, as remarked above, {\it up to  modification by some one-dimensional modules, finite-dimensional simple $\Ud$-modules are parametrised by their highest weights $\Lambda \in \mathcal{R}_{M,N}$.} 

The first inequality $\deg Q < \dim \Weyl(\Lambda;Q)$ comes from a detailed analysis of some weight subspaces of $\Weyl(\Lambda;Q)$, using firmly the triangular decomposition Theorem \ref{thm: triangular decomposition}. Indeed, we shall see that the $\Ud$-module structure on $\Weyl(\Lambda;Q)$ determines the parameter $(\Lambda;Q)$ uniquely, which justifies the definition of a highest weight.  For the proof of $\Weyl(\Lambda;Q)$ being finite-dimensional, we argue by induction on $(M,N)$ (this explains the reason for considering also $M = N$). We use a system of linear generators for the vector superspace $\Ud$, the so-called PBW generators,  to control the size of the Weyl modules. To be more precise, let
\begin{displaymath}
\Delta := \{ \alpha_i + \alpha_{i+1} + \cdots + \alpha_j\ |\ 1 \leq i \leq j \leq M+N-1 \}
\end{displaymath} 
be the set of positive roots of the Lie superalgebra $\mathfrak{sl}(M,N)$ with the  ordering: $\alpha_i + \cdots + \alpha_j < \alpha_{i'} + \cdots + \alpha_{j'}$ if either $i < i'$ or $i = i', j < j'$.   For $(\beta, n) \in \Delta \times \BZ$, we define the root vector $X_{\beta}(n) \in U_q^+(\mathcal{L}\mathfrak{sl}(M,N))$ as quantum brackets in such a way (Definition \ref{def: positive root vectors}) that finally  

\noindent \textbf{\underline{Theorem \ref{thm: pbw}}}. {\it The vector superspace $U_q^+(\mathcal{L}\mathfrak{sl}(M,N))$ is spanned by $\prod\limits_{\beta \in \Delta}^{\rightarrow} (\prod\limits_{i=1}^{c_{\beta}} X_{\beta}(n_{i,\beta}) )$ where $c_{\beta} \in \mathbb{Z}_{\geq 0}$ for $\beta \in \Delta$ and $n_{i,\beta} \in \mathbb{Z}$ for $1 \leq i \leq c_{\beta}$.} 

The proof of the PBW theorem above is a combinatorial argument by inductions on $(M,N)$ and on the length of weights. We have not considered the problem of linear independence, which is beyond the scope of this paper.

We remark that Equation \eqref{equ:Verma-Weyl 2} is by no means superficial. Indeed, for $\Lambda \in \mathcal{R}_{M,N}$,  the quotient of $\Verma(\Lambda)$ by Relation \eqref{equ:Verma-Weyl 1}, denoted by $\Weyl(\Lambda)$,  is infinite-dimensional. We call $\Weyl(\Lambda)$ the {\it universal Weyl module} in the sense that all integrable quotients of $\Verma(\Lambda)$ remain quotients of $\Weyl(\Lambda)$. In particular, contrary to the case of quantum affine algebras, {\it integrable highest weight $\nRightarrow$ finite-dimensional highest weight}.

\paragraph*{}\ \ The paper is organised as follows. In \S \ref{sec: 2}, we remind the notion of a Weyl module for the quantum affine algebra $U_q(\mathcal{L}\mathfrak{sl}_N)$, and that of a Kac module for the quantum superalgebra $U_q(\mathfrak{gl}(M,N))$.
In \S \ref{sec: 3}, we define the quantum affine superalgebra $U_q(\mathcal{L}\mathfrak{sl}(M,N))$ and its enlargement $U_q(\mathcal{L}'\mathfrak{sl}(M,N))$ in terms of Drinfel'd currents, following Yamane \cite{Yam2}. Here, the enlargement is needed to avoid the problem of linear dependence among the simple roots of $\mathfrak{sl}(M,N)$. We prove a triangular decomposition (Theorem \ref{thm: triangular decomposition}) in terms of Drinfel'd currents, following the argument of \cite{Jantzen,He}. Then we define the root vectors (Definition \ref{def: positive root vectors}) and prove  Theorem \ref{thm: pbw}. In \S \ref{sec: 4}, the notion of a highest weight, the Verma modules $\Verma(\Lambda)$, the Weyl modules $\Weyl(\Lambda;Q)$, and  the relative simple modules $\simple'(\Lambda)$ are defined. We prove that the Weyl modules are always finite-dimensional and non-zero (Theorem \ref{thm: main}) by using the triangular decomposition and the PBW theorem. When $M \neq N$,  we conclude the highest weight classification of  finite-dimensional simple $U_q(\mathcal{L}\mathfrak{sl}(M,N))$-modules (Proposition \ref{prop: highest weight representations}). The universal Weyl modules are introduced to study integrability property.

In \S \ref{sec: 5}, we recall Yamane's isomorphism (Theorem \ref{thm: two presentations}) between Drinfel'd and Chevalley presentations for $\Ud$ in the case $M \neq N$. From this isomorphism, we deduce  a formula for the highest weight of the tensor product of two highest weight vectors (Corollary \ref{cor: tensor product highest weight}) and henceforth a commutative monoid structure on the set $\mathcal{R}_{M,N}$ of highest weight. From Zhang's evaluation morphisms (Proposition \ref{prop: evaluation morphisms Chevalley}) we construct explicitly some simple $\Ud$-modules (Proposition \ref{prop: evaluation modules}). 

\S \ref{sec: final section} is left to further discussions. We include in the two appendixes the related calculations that are needed in the triangular decomposition and the coproduct formulae for some Drinfel'd currents.  

\noindent
\textbf{Acknowledgement:} the author is grateful to his supervisor David Hernandez for the discussions, to Marc Rosso for kindly giving his preprint \cite{Rosso2}, and to Xin Fang, Jyun-Ao Lin and Mathieu Mansuy for the discussions and for pointing out useful references.

\section{Preliminaries}  \label{sec: 2}
We recall the highest weight representation theories for the quantum affine algebra $U_q(\mathcal{L}\mathfrak{sl}_N)$ and the quantum superalgebra $U_q(\mathfrak{gl}(M,N))$. Here we use $\mathfrak{gl}$ instead of $\mathfrak{sl}$ to avoid the problem of linear dependence among simple roots when $M = N$ (see Notation \ref{notation: weight}). 
\subsection{Weyl modules for the quantum affine algebra $U_q(\mathcal{L}\mathfrak{sl}_N)$}  \label{sec: 2.1}
Fix $N \in \mathbb{Z}_{\geq 2}$. Let $(a_{i,j}) \in \mathrm{Mat}(N-1,\BZ)$ be a Cartan matrix for the simple Lie algebra $\mathfrak{sl}_N$ with 
\[ a_{i,j} = 2 \delta_{i,j} - \delta_{i,j-1} - \delta_{i,j+1}. \]
Following Drinfel'd, the quantum affine algebra $U_q(\mathcal{L}\mathfrak{sl}_N)$ is an algebra with \cite[Theorem 4.7]{Beck1}:
\begin{itemize}
\item[(a)] generators $X_{i,n}^{\pm}, h_{i,s}, K_i^{\pm 1}$ with $n \in \BZ, s \in \BZ_{\neq 0}, 1 \leq i \leq N-1$; 
\item[(b)] relations for $1 \leq i,j \leq N-1, m,n,k \in \BZ, s,t \in \BZ_{\neq 0}$
\begin{align*}
& K_iK_i^{-1} = K_i^{-1}K_i = 1,\ [K_i,K_j] = [K_i, h_{j,s}] = [h_{i,s},h_{j,t}] = 0, \\
& K_i X_{j,n}^{\pm} K_i^{-1} = q^{\pm a_{i,j}} X_{j,n}^{\pm},\ [h_{i,s}, X_{j,n}^{\pm}] = \pm \frac{[s a_{i,j}]_q}{s} X_{j,n+s}^{\pm}, \\
& [X_{i,m}^+, X_{j,n}^-] = \delta_{i,j} \frac{\phi_{i,m+n}^+ - \phi_{i,m+n}^-}{q-q^{-1}}, \\
& [X_{i,m}^{\pm}, X_{j,n}^{\pm}] = 0 \ \ \ \ \mathrm{if}\ |i-j| > 1, \\
& X_{i,m+1}^{\pm} X_{j,n}^{\pm} - q^{\pm a_{i,j}} X_{j,n}^{\pm} X_{i,m+1}^{\pm} = q^{\pm a_{i,j}} X_{i,m}^{\pm} X_{j,n+1}^{\pm} - X_{j,n+1}^{\pm} X_{i,m}^{\pm} \ \ \ \ \mathrm{if}\ |i-j| \leq 1, \\
&  [X_{i,m}^{\pm}, [X_{i,n}^{\pm}, X_{j,k}^{\pm}]_{q^{-1}} ]_q +  [X_{i,n}^{\pm}, [X_{i,m}^{\pm}, X_{j,k}^{\pm}]_{q^{-1}} ]_q = 0 \ \ \ \mathrm{if}\ |i-j| = 1.
\end{align*}
\end{itemize}
Here $[n]_q = \frac{q^n-q^{-n}}{q-q^{-1}}, [a,b]_u = ab - u ba$ and $[,] = [,]_1$. The $\phi_{i,m}^{\pm}$ are defined by the generating series 
\[ \sum_{n \in \BZ} \phi_{i,m}^{\pm} z^m = K_i^{\pm 1} \exp (\pm(q-q^{-1})\sum_{s \in \BZ_{>0}} h_{i,\pm s} z^{\pm s} ) \in U_q(\mathcal{L}\mathfrak{sl}_N)[[z^{\pm 1}]]. \]
Note that $U_q(\mathcal{L}\mathfrak{sl}_N)$ has a structure of Hopf algebra from its Chevalley presentation.

Let $\Lambda = (P_i(z): 1 \leq i \leq N-1) \in (1 + z \mathbb{C}[z])^{N-1}$. The {\it Weyl module}, $\Weyl(\Lambda)$, is the $U_q(\mathcal{L}\mathfrak{sl}_N)$-module generated by $v_{\Lambda}$ with relations (see \cite[\S 4]{CP3}, or the review \cite[\S 3.4]{CH} where we borrow the notations):
\begin{align}  
& X_{i,n}^+ v_{\Lambda} = 0 \ \  \ \ \mathrm{for}\ n \in \BZ,\ 1 \leq  i  \leq N - 1, \label{rel: Weyl 0} \\
& \sum_{n \in \BZ} \phi_{i,n}^{\pm} v_{\Lambda} z^n = q^{\deg P_i} \frac{P_i(zq^{-1})}{P_i(z q)}v_{\Lambda} \in \mathbb{C}v_{\Lambda}[[z^{\pm 1}]] \ \ \ \ \mathrm{for}\ 1 \leq i \leq N-1, \label{rel: Weyl 1} \\
& (X_{i,0}^-)^{1 + \deg P_i} v_{\Lambda} = 0 \ \ \ \mathrm{for}\ 1 \leq i \leq N-1. \label{rel: Weyl 2}
\end{align} 
Let $V$ be an $U_q(\mathcal{L}\mathfrak{sl}_N)$-module. We say that $V$ is {\it integrable} if the actions of $X_{i,0}^{\pm}$ for $1 \leq i \leq N-1$ are locally nilpotent. We say that $V$ is of {\it highest weight } $\Lambda$ if $V$ is generated by a vector $v$ satisfying Relations \eqref{rel: Weyl 0}-\eqref{rel: Weyl 1}. 

We reformulate \cite[Theorem 3.3]{CP2} and \cite[Proposition 4.6]{CP3} in the case of $\mathfrak{sl}_N$ as follows.
\begin{theorem}  \label{thm: Weyl modules for quantum affine algebras}
(a) For  all $\Lambda \in (1+z\mathbb{C}[z])^{N-1}$, we have $0 < \dim \Weyl(\Lambda) < \infty$, and $\Weyl(\Lambda)$ has a unique quotient which is a simple $U_q(\mathcal{L}\mathfrak{sl}_N)$-module, denoted by $\simple(\Lambda)$.

(b) All finite-dimensional simple $U_q(\mathcal{L}\mathfrak{sl}_N)$-modules are of the form $\simple(\Lambda) \otimes \mathbb{C}_{\theta}$ where $\Lambda \in (1 + z \mathbb{C}[z])^{N-1}$ and $\mathbb{C}_{\theta}$ is a one-dimensional $U_q(\mathcal{L}\mathfrak{sl}_N)$-module.

(c) All integrable modules of highest weight $\Lambda$ are quotients of $\Weyl(\Lambda)$, in particular, they are finite-dimensional. 
\end{theorem}
The Weyl modules $\Weyl(\Lambda)$ are generally non-simple, due to the non-semi-simplicity of the category of finite-dimensional  $U_q(\mathcal{L}\mathfrak{sl}_N)$-modules, a phenomenon that appears also in the classical case $U(\mathcal{L}\mathfrak{sl}_N)$.
\subsection{Kac modules for  the quantum superalgebra $U_q(\mathfrak{gl}(M,N))$}   \label{sec: 2.2}
From this section on, we consider superalgebras. By definition, a {\it superalgebra} is an (associative and unitary) algebra $A$ with a compatible $\super$-grading $A = A_{\even} \oplus A_{\odd}$, i.e. $A_{i}A_j \subseteq A_{i+j}$ for $i,j \in \super$. We remark that a superalgebra can be defined by a presentation: generators, their $\super$-degrees, and their defining relations. 

Let $V = V_{\even} \oplus V_{\odd}$ be a vector superspace. Write $|v| = i$ for $i \in \super$ and $v \in V_i$.   Endow the algebra of endomorphisms $\End(V)$ with the following canonical superalgebra structure:
\begin{displaymath}
\End(V)_{i} := \{ f \in \End(V)\ |\ f(V_j) \subseteq V_{i+j}\ \mathrm{for}\ j \in \super  \}. 
\end{displaymath} 
By a representation of a superalgebra $A$, we mean a couple $(\rho, V)$ where $V$ is a vector superspace and $\rho: A \longrightarrow \End(V)$ is a homomorphism of superalgebras. Call $V$ an $A$-module in this case. When $A$ is a Hopf superalgebra, given two representations  $(\rho_i,V_i)_{i=1,2}$, we can form another representation $((\rho_1 \stensor \rho_2) \Delta, V_1 \stensor V_2 )$. Here $\stensor$ means the {\it super tensor product} and $\Delta: A \longrightarrow A \stensor A$ is the coproduct. 
  
On the other hand, a {\it Lie superalgebra} is by definition a vector superspace $V = V_{\even} \oplus V_{\odd}$ with a Lie bracket $[,]: V \times V \longrightarrow V$ such that $[V_i,V_j] \subseteq V_{i+j}$ and for $a \in V_i, b \in V_j, c \in V_k$ (with $i,j,k \in \super$)
\begin{align*}
& [a,b] = - (-1)^{ij} [b,a]  \\
& [a,[b,c]] = [[a,b],c] + (-1)^{ij} [b,[a,c]].
\end{align*}
When $A$ is a superalgebra, $[a,b] = ab - (-1)^{ij} ba$ for $a \in A_i, b \in A_j$ makes $A$ into a Lie superalgebra. In particular, when $V$ is the vector superspace with $V_{\even} = \mathbb{C}^M$ and $V_{\odd} = \mathbb{C}^N$, we write $\End(V)$ as $\mathfrak{gl}(M,N)$ to emphasis its Lie superalgebra structure. There is a super-trace on $\mathfrak{gl}(M,N)$ given by
\begin{displaymath}
\str: \mathfrak{gl}(M,N) \longrightarrow \mathbb{C},\ \ f + g \mapsto \mathrm{tr}_{V_{\even}}(f|_{V_{\even}}) - \mathrm{tr}_{V_{\odd}}(f|_{V_{\odd}}) \ \ \mathrm{for}\ f \in \mathfrak{gl}(M,N)_{\even},\ g \in \mathfrak{gl}(M,N)_{\odd}.
\end{displaymath} 
And $\mathfrak{sl}(M,N) := \ker (\str)$ is a sub-Lie-superalgebra of $\mathfrak{gl}(M,N)$. We refer to \cite{Kac1,Sch} for the classification of finite-dimensional simple Lie superalgebras in terms of Dynkin diagrams and Cartan matrices. 

Fix $M,N \in \mathbb{Z}_{\geq 1}$. Equip the free $\BZ$-module $\bigoplus\limits_{i=1}^{M+N} \mathbb{Z}\epsilon_i$ with the following bilinear from 
\begin{align}
& (\epsilon_i, \epsilon_j) = l_i \delta_{i,j},\ \ l_i = \begin{cases}
1 & \mathrm{if}\ 1 \leq i \leq M, \\
-1 & \mathrm{if}\ M+1 \leq i \leq M+N.
\end{cases}   \label{equ: bilinear form canonnical}
\end{align} 
For $1 \leq i \leq M+N-1$, set $q_i = q^{l_i}$.
The quantum superalgebra $U_q(\mathfrak{gl}(M,N))$ is a superalgebra with:
\begin{itemize}
\item[(a)] generators $t_i^{\pm 1}, e_j^{\pm}$ where $1 \leq i \leq M+N, 1 \leq j \leq M+N-1$;
\item[(b)] $\super$-grading $|e_M^{\pm}| = \odd$ and $|t_M^{\pm 1}| = |t_i^{\pm 1}| = |e_i^{\pm}| = \even$ for $1 \leq i \leq M+N-1, i \neq M$;
\item[(c)] relations   for  $1 \leq i \leq M+N, 1 \leq j,k \leq M+N-1$ \cite[Proposition 10.4.1]{Yam1}
\begin{align*}
&t_i t_i^{-1} = 1 = t_i^{-1} t_i,\ t_i e_j^{\pm} t_i^{-1} = q^{\pm l_i (\epsilon_i, \epsilon_j - \epsilon_{j+1}) } e_j^{\pm}, \\
&[e_j^+, e_k^{-}] = \delta_{j,k} \frac{t_j^{l_j} t_{j+1}^{-l_{j+1}} - t_j^{-l_j} t_{j+1}^{l_{j+1}} }{q_j - q_j^{-1}}, \\
& [e_j^{\pm}, [e_j^{\pm}, e_k^{\pm}]_{q^{-1}}]_q = 0 \ \ \ \mathrm{if}\ c_{j,k} = \pm 1, j \neq M, \\
& [[[e_{M-1}^{\pm}, e_M^{\pm}]_q, e_{M+1}^{\pm}]_{q^{-1}}, e_M^{\pm}] = 0 \ \ \ \ \mathrm{when}\ M,N > 1,
\end{align*}
\end{itemize}
where the super-brackets are: $[,] = [,]_1, [a,b]_u = a b - (-1)^{|a||b|} u ba$ for $a,b$ homogeneous. $U_q(\mathfrak{gl}(M,N))$ is endowed with a Hopf superalgebra structure as follows \cite[Eq. (2.5)]{BKK}
\begin{equation}  \label{equ: coproduct BKK}
\Delta(t_i) = t_i \stensor t_i,\ \Delta (e_j^+) = 1 \stensor e_j^+ + e_j^+ \stensor t_j^{-l_j} t_{j+1}^{l_{j+1}},\ \Delta (e_j^-) = t_j^{l_j} t_{j+1}^{-l_{j+1}} \stensor e_j^- + e_j^- \stensor 1. 
\end{equation} 
We remark that the subalgebra $U_q(\mathfrak{sl}(M,N))$, generated by $e_i^{\pm}, t_i^{l_i}t_{i+1}^{-l_{i+1}}$ for $1 \leq i \leq M+N-1$ is a sub-Hopf-superalgebra.
Let  $e^+ := [\cdots[[e_1^+, e_2^+]_{q_2}, e_3^+]_{q_3}],\cdots,e_{M+N-1}^+]_{q_{M+N-1}}$. The following lemma is needed later. 
\begin{lem} \label{lem: Yamane commutation relations}
(see \cite[Lemma 5.2.1]{Yam1}) For $2 \leq j \leq M+N-1$, $[e^+, e_j^+]_{q^{(\epsilon_1 - \epsilon_{M+N}, \epsilon_{j+1} - \epsilon_j)}} = 0$.
\end{lem}

Let $\mathcal{S}_{M,N}$ be the set of $\Lambda = (\Lambda_i: 1 \leq i \leq M+N ) \in \mathbb{C}$ such that \cite[Eq. (13)]{Zh2}
\[ \Delta_i := \Lambda_i - \Lambda_{i+1} \in \mathbb{Z}_{\geq 0}\ \mathrm{for}\ 1 \leq i \leq M+N-1, i \neq M. \]
Let $\Lambda \in \mathcal{S}_{M,N}$. The {\it Kac module}, $\Kac(\Lambda)$, is the $U_q(\mathfrak{gl}(M,N))$-module generated by $v_{\Lambda}$, with $\super$-grading $\even$, and relations \cite[\S 3]{Zh2}
\begin{align}
& e_j^{+} v_{\Lambda} = 0\ \ \ \ \ \mathrm{for}\ 1 \leq j \leq M+N-1, \\
& t_i v_{\Lambda} = q^{\Lambda_i} v_{\Lambda} \ \ \ \ \mathrm{for}\ 1 \leq i \leq M+N, \\
& (e_j^-)^{1 + \Delta_j} v_{\Lambda} = 0 \ \ \ \ \mathrm{for}\ 1 \leq j \leq M+N-1,j \neq M.
\end{align}
We call it Kac module as it is a generalisation of Kac's induction module construction for Lie superalgebras \cite[Proposition 2.1]{Kac}. Note that we also have the notion of integrable modules (actions of the $e_j^{\pm}$ being locally nilpotent) and highest weight modules. The Kac modules in the category of finite-dimensional $U_q(\mathfrak{gl}(M,N))$-modules play the same role as the Weyl modules in that of finite-dimensional $U_q(\mathcal{L}\mathfrak{sl}_N)$-modules.
\begin{theorem} \label{thm: Kac modules for quantum superalgebra}
(a) For $\Lambda \in \mathcal{S}_{M,N}$, $0 < \dim \Kac(\Lambda) < \infty$, and $\Kac(\Lambda)$ has a unique quotient which is a simple $U_q(\mathfrak{gl}(M,N))$-module, denoted by $L(\Lambda)$.

(b) All finite-dimensional simple $U_q(\mathfrak{gl}(M,N))$-modules are of the form $L(\Lambda) \stensor \mathbb{C}_{\theta}$ where $\Lambda \in \mathcal{S}_{M,N}$ and $\mathbb{C}_{\theta}$ is a one-dimensional $U_q(\mathfrak{gl}(M,N))$-module.

(c) All integrable modules of highest weight $\Lambda$ are quotients of $\Kac(\Lambda)$, in particular, they are finite-dimensional.  
\end{theorem} 

\section{The quantum affine superalgebra $U_q(\mathcal{L}\mathfrak{sl}(M,N))$}  \label{sec: 3}
In this section, we recall the Drinfel'd realization of the quantum affine superalgebra $U_q(\mathcal{L}\mathfrak{sl}(M,N))$ following Yamane \cite[Theorem 8.5.1]{Yam2}.  We prove a triangular decomposition for this superalgebra. Then we give a system of linear generators of PBW type in terms of Drinfel'd currents. These turn out to be crucial for the development of finite-dimensional representations in the next section.
\subsection{Drinfel'd presentation of $U_q(\mathcal{L}\mathfrak{sl}(M,N))$} \label{sec: Drinfeld presentation}
Following Equation \eqref{equ: bilinear form canonnical}, define 
\begin{displaymath}
c_{i,j} := (\epsilon_i - \epsilon_{i+1}, \epsilon_j - \epsilon_{j+1})\ \ \  \mathrm{for}\ 1 \leq i, j \leq M+N-1.
\end{displaymath}
 Hence $(l_i c_{i,j})$ can be viewed as a Cartan matrix for the Lie superalgebra $\mathfrak{sl}(M,N)$.
\begin{definition}  \label{def: Drinfeld presentation}
 \cite[Theorem 8.5.1]{Yam2} $U_q(\mathcal{L}\mathfrak{sl}(M,N))$ is the superalgebra generated by $X_{i,n}^{\pm}, K_i^{\pm 1}, h_{i,s}$ for $1 \leq i \leq M+N-1, n \in \BZ, s \in \BZ_{\neq 0}$, with the $\super$-grading $|X_{M,n}^{\pm}| = \odd$ for $n \in \BZ$ and $\even$ for other generators, and with the following relations: $1 \leq i,j \leq M+N-1, m,n,k,u \in \mathbb{Z}, s,t \in \BZ_{\neq 0}$
\begin{align}
& K_iK_i^{-1} = 1 = K_i^{-1}K_i,\ [K_i, K_j] = [K_i, h_{j,s}] = [h_{i,s}, h_{j,t}] = 0,  \label{rel: Cartan}  \\
& K_i X_{j,n}^{\pm} K_i^{-1} = q^{\pm c_{i,j}} X_{j,n}^{\pm},\ [h_{i,s}, X_{j,n}^{\pm}] = \pm \frac{[sl_i c_{i,j}]_{q_i}}{s} X_{j,n+s}^{\pm},\label{rel: triangular1}  \\
& [X_{i,m}^{+}, X_{j,n}^{-}] = \delta_{i,j} \frac{\phi_{i,m+n}^+ - \phi_{i,m+n}^-}{q_i-q_i^{-1}}, \label{rel: triangular2}  \\
& [X_{i,m}^{\pm}, X_{j,n}^{\pm}] = 0 \ \ \ \mathrm{for}\ c_{i,j} = 0, \label{rel: Drinfel1}  \\
& X_{i,m+1}^{\pm} X_{j,n}^{\pm} - q^{\pm c_{i,j}} X_{j,n}^{\pm}X_{i,m+1}^{\pm} = q^{\pm c_{i,j}} X_{i,m}^{\pm}X_{j,n+1}^{\pm} - X_{j,n+1}^{\pm} X_{i,m}^{\pm} \ \ \ \mathrm{for}\ c_{i,j} \neq 0, \label{rel: Drinfeld2}  \\
& [X_{i,m}^{\pm}, [X_{i,n}^{\pm}, X_{j,k}^{\pm}]_{q^{-1}} ]_q +  [X_{i,n}^{\pm}, [X_{i,m}^{\pm}, X_{j,k}^{\pm}]_{q^{-1}} ]_q = 0 \ \ \ \mathrm{for}\ c_{i,j} = \pm 1, i \neq M, \label{rel: Serre3}  \\
& [[[X_{M-1,m}^{\pm}, X_{M,n}^{\pm}]_{q^{-1}}, X_{M+1,k}^{\pm}]_q, X_{M,u}^{\pm} ] + [[[X_{M-1,m}^{\pm}, X_{M,u}^{\pm}]_{q^{-1}}, X_{M+1,k}^{\pm}]_q, X_{M,n}^{\pm} ] = 0 \ \ \mathrm{when}\ M,N > 1.  \label{rel: ocillation4}
\end{align}
where the $\phi_{i,n}^{\pm}$ are given by the generating series
\begin{align}
\sum_{n \in \BZ} \phi_{i,n}^{\pm} z^n = K_i^{\pm 1} \exp (\pm(q_i-q_i^{-1})\sum_{s\in \BZ_{>0}} h_{i,\pm s} z^{\pm s} ) \in \Ud[[z^{\pm 1}]]
\end{align}
\end{definition} 
We understand that $U_q(\mathcal{L}\mathfrak{sl}(M,0)) = U_q(\mathcal{L}\mathfrak{sl}_M)$ and $U_q(\mathcal{L}\mathfrak{sl}(0,N)) = U_{q^{-1}}(\mathcal{L}\mathfrak{sl}_N)$.
We also need an extension of the superalgebra $U_q(\mathcal{L}\mathfrak{sl}(M,N))$. For this, note that there is an action of the group algebra $\mathbb{C}[K_0,K_0^{-1}]$ on it: for  $i \in \{ 1,\cdots,M+N-1 \}, s \in \mathbb{Z}_{\neq 0}, n \in \mathbb{Z}$
\begin{align}
K_0 K_i^{\pm 1} K_0^{-1} = K_i^{\pm 1},\ K_0 h_{i,s} K_0^{-1} = h_{i,s},\ K_0 X_{i,n}^{\pm} K_0^{-1} = q^{\pm \delta_{i,1}} X_{i,n}^{\pm}  \label{rel: triangular1'}.
\end{align}
Let $U_q(\mathcal{L}'\mathfrak{sl}(M,N)) := U_q(\mathcal{L}\mathfrak{sl}(M,N)) \rtimes \mathbb{C}[K_0,K_0^{-1}]$.

One can see informally $U_q(\mathcal{L}'\mathfrak{sl}(M,N))$ as a deformation of the universal enveloping algebra of the Lie superalgebra $\mathcal{L}'\mathfrak{sl}(M,N)$
 where 
\begin{align*}
 \mathcal{L}'\mathfrak{sl}(M,N) = \mathcal{L}\mathfrak{sl}(M,N) \oplus \mathbb{C}(E_{11} \otimes 1) \subset \mathfrak{gl}(M,N) \otimes \mathbb{C}[t,t^{-1}]. 
\end{align*}
When $M = N$, the Lie superalgebra $\mathcal{L}'\mathfrak{sl}(M,N)$ is nothing but $(\mathfrak{sl}(M,N)^{(1)})^{\mathscr{H}}$ in Yamane's notation \cite[\S 1.5]{Yam2}.

\subsection{Triangular decomposition}   \label{sec: 3.2}
There is an injection of superalgebras $U_q(\mathcal{L}\mathfrak{sl}(M,N)) \hookrightarrow U_q(\mathcal{L}'\mathfrak{sl}(M,N))$ given by $x \mapsto x \rtimes 1$. Identify $U_q(\mathcal{L}\mathfrak{sl}(M,N))$ with a subalgebra of $U_q(\mathcal{L}'\mathfrak{sl}(M,N))$ from now on.
\begin{notation}
Let $U_q^{\pm}(\mathcal{L}\mathfrak{sl}(M,N))$ (resp.  $U_q^0(\mathcal{L}\mathfrak{sl}(M,N))$) be the subalgebra of $U_q(\mathcal{L}\mathfrak{sl}(M,N))$ generated by the $X_{i,n}^{\pm}$ (resp. the $K_i^{\pm 1}, h_{i,s}$) for all $i \in \{1,\cdots,M+N-1\}, n \in \mathbb{Z}, s \in \mathbb{Z}_{\neq 0}$. Let $U_q^0(\mathcal{L}'\mathfrak{sl}(M,N))$ be  the subalgebra of $U_q(\mathcal{L}'\mathfrak{sl}(M,N))$ generated by $U_q^0(\mathcal{L}\mathfrak{sl}(M,N))$ and $K_0^{\pm 1}$. These subalgebras are clearly $\super$-homogeneous.
\end{notation}
\begin{theorem}  \label{thm: triangular decomposition}
 We have the following triangular decomposition for $U_q(\mathcal{L}\mathfrak{sl}(M,N))$:
\begin{itemize}
\item[(a)] the multiplication  $m: U_q^-(\mathcal{L}\mathfrak{sl}(M,N)) \stensor U_q^0(\mathcal{L}\mathfrak{sl}(M,N)) \stensor U_q^+(\mathcal{L}\mathfrak{sl}(M,N)) \longrightarrow U_q(\mathcal{L}\mathfrak{sl}(M,N))$ is an isomorphism of vector superspaces;

\item[(b)] $U_q^{+}(\mathcal{L}\mathfrak{sl}(M,N))$ (resp. $U_q^-(\mathcal{L}\mathfrak{sl}(M,N))$) is isomorphic to the algebra with generators $X_{i,n}^+$ (resp. $X_{i,n}^-$) and Relations \eqref{rel: Drinfel1}-\eqref{rel: ocillation4} with $+$ (resp. Relations \eqref{rel: Drinfel1}-\eqref{rel: ocillation4} with $-$); 

\item[(c)] $U_q^0(\mathcal{L}\mathfrak{sl}(M,N))$ is an algebra of Laurent polynomials 
\begin{align}
U_q^0(\mathcal{L}\mathfrak{sl}(M,N)) \cong \mathbb{C}[h_{i,s}: s \in \mathbb{Z}_{\neq 0}, 1 \leq i \leq M+N-1][K_i,K_i^{-1}: 1 \leq i \leq M+N-1].  \label{iso: diagonal part}
\end{align}  
\end{itemize}
\end{theorem}
As an immediate consequence, we obtain also a triangular decomposition for $U_q(\mathcal{L}'\mathfrak{sl}(M,N))$. 
\begin{cor} \label{cor: triangular enlargement}
The multiplication below 
\begin{displaymath}
m: U_q^-(\mathcal{L}\mathfrak{sl}(M,N)) \stensor U_q^0(\mathcal{L}'\mathfrak{sl}(M,N)) \stensor U_q^+(\mathcal{L}\mathfrak{sl}(M,N)) \longrightarrow U_q(\mathcal{L}'\mathfrak{sl}(M,N)),\ \ a\stensor b \stensor c \mapsto abc 
\end{displaymath} 
 is an isomorphism of vector superspaces.  $U_q^0(\mathcal{L}'\mathfrak{sl}(M,N))$ is an algebra of Laurent polynomials 
\begin{align}
U_q^0(\mathcal{L}'\mathfrak{sl}(M,N)) \cong \mathbb{C}[h_{i,s}: s \in \mathbb{Z}_{\neq 0}, 1 \leq i \leq M+N-1][K_i, K_i^{- 1}: 0 \leq i \leq M+N-1]. 
\end{align}
\end{cor}
Another consequence is  the existence of (anti-)isomorphisms of superalgebras.
\begin{cor} \label{cor: canonical isomorphisms}
 (1) There is an isomorphism of superalgebras $\tau_1: U_q^+(\mathcal{L}\mathfrak{sl}(M,N)) \longrightarrow U_q^-(\mathcal{L}\mathfrak{sl}(M,N))$ defined by $\tau_1(X_{i,n}^+) = X_{i,-n}^-$ for all $n \in \mathbb{Z}$ and $1 \leq i \leq M+N-1$.

(2) There is an anti-automorphism of superalgebras $\tau_2: U_q^+(\mathcal{L}\mathfrak{sl}(M,N)) \longrightarrow U_q^+(\mathcal{L}\mathfrak{sl}(M,N))$ defined by $\tau_2(X_{i,n}^+) = X_{i,-n}^+$ for all $n \in \mathbb{Z}$ and $1 \leq i \leq M+N-1$. 
\end{cor}
\begin{proof}
In view of   Theorem \ref{thm: triangular decomposition} about presentations of algebras, it is enough to prove that $\tau_1,\tau_2$ respect Relations \eqref{rel: Drinfel1}-\eqref{rel: ocillation4}.
\end{proof}
\begin{remark}  \label{rmk: two types of triangular decomposition}
(1) The triangular decomposition will be used to construct the Verma modules and to argue that the Weyl modules are non-zero. See \S \ref{sec: 4.2}.

(2) There are two types of triangular decomposition: one is in terms of Chevalley generators, the other Drinfel'd currents. For the Chevalley type, the triangular decomposition for  quantum Kac-Moody algebras was proved in \cite[Theorem 4.21]{Jantzen}. For the Drinfel'd type, Hernandez proved the triangular decomposition for  general quantum affinizations \cite[Theorem 3.2]{He}. Their ideas of proof are essentially the same, which we shall follow below.

(3) For  $\mathfrak{g}$ a simple finite-dimensional Lie algebra, as demonstrated by Gross\'{e} \cite[Proposition 8]{double}, one can realize the quantum affine algebra $U_q(\widehat{\mathfrak{g}})$ as a quantum double by introducing topological coproducts on the Borel subalgebras with respect to Drinfel'd currents.  In this way, the Drinfel'd type triangular decomposition follows automatically and a topological Hopf algebra structure is deduced on $U_q(\widehat{\mathfrak{g}})$. We believe that analogous results hold for $\Ud$. In particular, $\Ud$ could be endowed with a topological Hopf superalgebra structure (with coproduct being Drinfel'd new coproduct).
\end{remark}

 We proceed to proving Theorem \ref{thm: triangular decomposition}.
Let $\tilde{V}$ be the superalgebra defined by: generators $X_{i,n}^{\pm 1}, h_{i,s}, K_i^{\pm 1}$ ($1 \leq i \leq M+N-1, n \in \mathbb{Z}, s \in \BZ_{\neq 0}$); $\super$-grading $|X_{M,n}^{\pm}| = \odd$ and $\even$ otherwise; Relations \eqref{rel: Cartan}-\eqref{rel: triangular2}. Define its three subalgebras $\tilde{V}^{\pm}$ and $\tilde{V}^0$ analogously. Then $(\tilde{V}^-, \tilde{V}^0, \tilde{V}^+)$ forms a triangular decomposition for $\tilde{V}$. Moreover,  $\tilde{V}^{+}$ (resp. $\tilde{V}^-$) is freely generated by the $X_{i,n}^+$ (resp. the $X_{i,n}^-$), and  $\tilde{V}^0$ is the RHS of $\eqref{iso: diagonal part}$.

For $1 \leq i \leq j \leq M+N-1$, let $I_{i,j}^{\pm}$ be the vector subspace of $\tilde{V}^{\pm}$ generated by the vectors $[X_{i,m}^{\pm}, X_{j,n}^{\pm}]$ if $c_{i,j} = 0$ and $X_{i,m+1}^{\pm} X_{j,n}^{\pm} - q^{\pm c_{i,j}} X_{j,n}^{\pm}X_{i,m+1}^{\pm} - q^{\pm c_{i,j}} X_{i,m}^{\pm}X_{j,n+1}^{\pm} + X_{j,n+1}^{\pm} X_{i,m}^{\pm}$ if $c_{i,j} \neq 0$ for all $m,n \in \BZ$.
Let $I^+$ (resp. $I^-$) be the sum of the $I_{i,j}^+$ (resp. the $I_{i,j}^-$).
\begin{lem}  \label{lem: technical}
Let $1 \leq i, k, u \leq M+N-1$ with $k \leq u$. Then $[I_{k,u}^{\pm}, X_{i,0}^{\mp}] = 0$ in $\tilde{V}$. Furthermore, the vector subspaces $\tilde{V}^-\tilde{V}^0\tilde{V}^+I^+\tilde{V}^+$ and $\tilde{V}^-I^-\tilde{V}^-\tilde{V}^0\tilde{V}^+$ are two-sided ideals of $\tilde{V}$.
\end{lem}
\begin{proof}
We argue that this follows essentially from \cite[Theorem 3.2]{He}. If $i \notin \{k,u\}$, then it is clear that $[I_{k,u}^{\pm}, X_{i,0}^{\mp}] = 0$ from Relation \eqref{rel: triangular2}. Without loss of generality, suppose $i=k$.

If $c_{k,u} = 0$ and $k \neq u$, then $[[X_{k,m}^+, X_{u,n}^+], X_{k,0}^-] = [\frac{\phi_{k,m}^+ - \phi_{k,m}^-}{q_k - q_{k}^{-1}}, X_{u,n}^+]$. Writing $\phi_{k,m}^{\pm}$ as a product of $K_k^{\pm 1}, h_{k,s}$ and using the relations $K_k X_{u,n}^+ = X_{u,n}^+ K_k, h_{k,s} X_{u,n}^+ = X_{u,n}^+ h_{k,s}$, we see that  $[\frac{\phi_{k,m}^+ - \phi_{k,m}^-}{q_k - q_{k}^{-1}}, X_{u,n}^+] = 0$. This says that $[I_{k,u}^+, X_{k,0}^-] = 0$. Similarly, $[I_{k,u}^-, X_{k,0}^+] = 0$. 

If $c_{k,u} = 0$ and $k = u$, then $k = u = M$. We have 
\[[[X_{M,m}^+, X_{M,n}^+], X_{M,0}^-] = [X_{M,m}^+, \frac{\phi_{M,n}^+ - \phi_{M,n}^-}{q-q^{-1}} ] - [\frac{\phi_{M,m}^+ - \phi_{M,m}^-}{q-q^{-1}}, X_{M,n}^+].\]
 Again the relations $K_M X_{M,n}^+ = X_{M,n}^+ K_M, h_{M,s}X_{M,n}^+ = X_{M,n}^+ h_{M,s}$ imply that 
 $[I_{M,M}^+,X_{M,0}^-] = 0$. Similarly, $[I_{M,M}^-, X_{M,0}^+] = 0$.

If $c_{k,u} \neq 0$ and $k \neq u$, then $c_{k,u} = \pm 1$ and $[[X_{k,m}^+,X_{u,n}^+],X_{k,0}^-] = [\frac{\phi_{k,m}^+-\phi_{k,m}^-}{q_k -q_k^{-1}}, X_{u,n}^+] \in \tilde{V}$. We want to write this vector as a product of the form $\tilde{V}^-\tilde{V}^0\tilde{V}^+$ by using \textbf{only} the following relations
\begin{align*}
& K_k X_{u,n}^+ K_k^{-1} = q^{c_{k,u}} X_{u,n}^+,  \\
& [h_{k,s}, X_{u,n}^+] = \frac{[s l_k c_{k,u}]_{q_k}}{s} X_{u,n+s}^+, \\
& K_k^{\pm 1} \exp(\pm(q_k - q_k^{-1})\sum_{s \in \BZ_{>0}} h_{k,\pm s} z^{\pm s} ) = \sum_{n \in \BZ} \phi_{k,n}^{\pm} z^n. 
\end{align*} 
We are in the same situation as $U_{q_k}(\widehat{\mathfrak{sl}_3})$, when showing that the Drinfel'd relations of degree 2 respect the triangular decomposition. It follows from Theorem 3.2 and the technical lemmas in \S 3.3.1 of  \cite{He} that $[\frac{\phi_{k,m}^+-\phi_{k,m}^-}{q_k -q_k^{-1}}, X_{u,n}^+] = 0$. As a result, $[I_{k,u}^{\pm}, X_{k,0}^{\mp}] = 0$.

Similar considerations lead to $[I_{k,u}^{\pm}, X_{k,0}^{\mp}] = 0$ when $c_{k,u} \neq 0$ and $k=u$.

For the second part, note that the $I_{i,j}^{\pm}$ are stable by the $[h_{u,s},]$. Relation \eqref{rel: triangular2} applies.
\end{proof}

This means that the Drinfel'd relations of degree 2 respect the triangular decomposition.   Let $V$ be the quotient of $\tilde{V}$ by the two-sided ideal generated by the $I^+ + I^-$. Then 
\begin{displaymath}
V = \frac{\tilde{V}}{\tilde{V}^-I^-\tilde{V}^-\tilde{V}^0\tilde{V}^+ +\tilde{V}^-\tilde{V}^0\tilde{V}^+I^+\tilde{V}^+ } \cong \frac{\tilde{V}^-}{\tilde{V}^-I^-\tilde{V}^-} \stensor \tilde{V}^0 \stensor \frac{\tilde{V}^+}{\tilde{V}^+I^+\tilde{V}^+}
\end{displaymath}
where the isomorphism is induced by the triangular decomposition for $\tilde{V}$. Let $\pi_1: \tilde{V} \longrightarrow V$ be the canonical projection. By abuse of notation, we identify $X_{i,n}^{\pm}, K_i^{\pm 1}$ and $h_{i,s}$  with $\pi_1(X_{i,n}^{\pm}), \pi_1(K_i^{\pm 1})$, and $\pi_1(h_{i,s})$ respectively. Let $V^{\pm} = \pi_1(\tilde{V}^{\pm})$ and $V^0 = \pi_1(\tilde{V}^0)$. The above identifications say that $(V^-,V^0,V^+)$ forms a triangular decomposition for $V$. Moreover, the projection $\pi_1$ induces isomorphisms
\begin{displaymath}
V^{+} \cong \frac{\tilde{V}^+}{\tilde{V}^+ I^+ \tilde{V}^+},\ V^- \cong \frac{\tilde{V}^-}{\tilde{V}^-I^-\tilde{V}^-},\ V^0 \cong \tilde{V}^0.  
\end{displaymath}  

When $c_{i,j} = \pm 1$ and $i \neq M$, let  $J_{i,j}^{\pm}$ be subspace of $V^{\pm}$ generated by the LHS of  Relation \eqref{rel: Serre3} with $\pm$ for all $m,n,k \in \BZ$. Let $J^+$ (resp. $J^-$) be the sum of the $J_{i,j}^+$ (resp. the $J_{i,j}^-$). Using Theorem 3.2 and the technical lemmas in \S 3.3.1 of \cite{He} we deduce that (the same argument as Lemma \ref{lem: technical} above)
\begin{lem}
For all $1 \leq i,j,k \leq M+N-1$ such that $c_{i,j} = \pm 1$ and $i \neq M$, we have $[J_{i,j}^{\pm}, X_{k,0}^{\mp}] = 0$ in $V$. Therefore, the vector subspaces  $V^-V^0V^+J^+V^+$ and $V^-J^-V^-V^0V^+$ are two-sided ideals of $V$.
\end{lem}
In other words, the Serre relations of degree 3 respect the triangular decomposition. Suppose now  $M,N > 1$. Let $O^+$ (resp. $O^-$) be the subspace of $V^+$ (resp. $V^-$) generated by the LHS of Relation \eqref{rel: ocillation4} with $+$ (resp. with $-$) for all $m,n,k,u \in \mathbb{Z}$.

\begin{lem}  \label{lem: oscillation relations vs triangular decomposition}
In the superalgebra $V$, $[O^{\pm}, X_{i,0}^{\mp}] = 0$ for all $1 \leq i \leq M+N-1$. Therefore, $V^-O^-V^-V^0V^+$ and $V^-V^0V^+O^+V^+$ are two-sided ideals of $V$.
\end{lem}
\noindent {\it Sketch of proof.} When $i \notin \{M-1, M, M+1\}$, this is clear from Relation \eqref{rel: triangular2}. We are reduced to the case $M = N = 2$.  The related calculations are carried out in  Appendix A. \hfill $\Box$

By definition the superalgebra $\Ud$ is the quotient of $V$ by the two-sided ideal $N$ generated by $J^+ + J^- + O^+ + O^-$. Now from the two lemmas above we get 
\begin{displaymath}
N = V^-(J^-+O^-)V^-V^0V^+ + V^-V^0V^+(J^++O^+)V^+,
\end{displaymath}
from which Theorem \ref{thm: triangular decomposition} follows.
  
\subsection{Linear generators of PBW type}
We shall find a system of linear generators for the vector superspace $U_q^+(\mathcal{L}\mathfrak{sl}(M,N))$.  In view of Corollary \ref{cor: canonical isomorphisms}, this will produce one for $U_q^-(\mathcal{L}\mathfrak{sl}(M,N))$.

\begin{notation} \label{notation: weight}
(1) For simplicity, in this section, let $U_{M,N} := U_q^+(\mathcal{L}\mathfrak{sl}(M,N))$.

(2) Let $A'_{M,N}$ be the subalgebra of $U_q(\mathcal{L}'\mathfrak{sl}(M,N))$ generated by $K_i^{\pm 1}$ with $0 \leq i \leq M+N-1$. Then $A'_{M,N} = \mathbb{C}[K_i,K_i^{-1}: 0 \leq i \leq M+N-1]$ is an algebra of Laurent polynomials (Corollary \ref{cor: triangular enlargement}). Let $P_{M,N}$ be the set of algebra homomorphisms $A'_{M,N} \longrightarrow \mathbb{C}$. Then $P_{M,N}$ has an abelian group structure: for $\alpha, \beta \in P_{M,N}$
\begin{displaymath}
(\alpha + \beta)(K_i^{\pm 1}) = \alpha(K_i^{\pm 1}) \beta(K_i^{\pm 1}),\ \ 0(K_i^{\pm 1}) = 1,\ \ (-\alpha)(K_i^{\pm 1}) = \alpha(K_i^{\mp 1}).
\end{displaymath}
From  Relations \eqref{rel: triangular1} and \eqref{rel: triangular1'}, we see that $U_{M,N} = \bigoplus\limits_{\alpha \in P_{M,N}} (U_{M,N})_{\alpha}$ where
\begin{displaymath}
(U_{M,N})_{\alpha} = \{ x \in U_{M,N}\ |\ K_i x K_i^{-1} = \alpha(K_i) x\ \mathrm{for}\ 0 \leq i \leq M+N-1  \}.
\end{displaymath}
Moreover, $(U_{M,N})_{\alpha} (U_{M,N})_{\beta} = (U_{M,N})_{\alpha + \beta}$. Let $\weight(U_{M,N}) := \{\alpha \in P_{M,N}\ |\ (U_{M,N})_{\alpha} \neq 0 \}$.

(3) For $1 \leq i \leq M+N-1$, define $\alpha_i \in P_{M,N}$ by: $\alpha_i(K_j) = \begin{cases} q^{\delta_{i,1}} & \mathrm{if}\ j = 0, \\
q^{c_{i,j}} & \mathrm{if}\ j > 0. \end{cases}$\ These $\alpha_i$ are $\mathbb{Z}$-linearly independent in $P_{M,N}$. Let $Q_{M,N} := \bigoplus\limits_{i=1}^{M+N-1}\mathbb{Z}\alpha_i$ and $Q_{M,N}^+ := \bigoplus\limits_{i=1}^{M+N-1}\mathbb{Z}_{\geq 0}\alpha_i$. We have $X_{i,n}^+ \in (U_{M,N})_{\alpha_i}$ and $\weight(U_{M,N}) \subseteq Q_{M,N}^+$. (It is for the reason of linear independence among $\alpha_i$ that we introduce $U_q(\mathcal{L}'\mathfrak{sl}(M,N))$.)

(4) Set $\Delta_{M,N} := \{ \alpha_i + \alpha_{i+1} + \cdots + \alpha_j \in Q_{M,N}^+\ |\ 1 \leq i \leq j \leq M+N-1 \}$ with the following total ordering: $\alpha_i + \cdots + \alpha_j \leq \alpha_{i'}+ \cdots + \alpha_j$ if $i < i'$ or $i=i', j \leq j'$. 
\end{notation}

Following \cite[Definition 3.9]{Rosso}, we can now define the root vectors
\begin{definition}  \label{def: positive root vectors}
For $\beta = \alpha_i + \cdots + \alpha_j \in \Delta_{M,N}$ and $n \in \mathbb{Z}$, define $X_{\beta}(n) \in (U_{M,N})_{\beta}$ by 
\begin{align}
X_{\beta}(n) := [\cdots [[X_{i,n}^+, X_{i+1,0}^+]_{q_{i+1}}, X_{i+2,0}^+]_{q_{i+2}}, \cdots, X_{j,0}^+]_{q_j}
\end{align}
with the convention that $X_{\alpha_i}(n) = X_{i,n}^+ = X_i(n)$.
\end{definition}
Similar to the  quantum affine algebra $U_{r,s}(\widehat{\mathfrak{sl}_n})$ \cite[Theorem 3.11]{Rosso}, we have
\begin{theorem}  \label{thm: pbw}
The vector space $U_{M,N}$ is spanned by vectors of the form $\prod\limits_{\beta \in \Delta_{M,N}}^{\rightarrow} (\prod\limits_{i=1}^{c_{\beta}} X_{\beta}(n_{i,\beta}) )$ where $c_{\beta} \in \mathbb{Z}_{\geq 0}$ for $\beta \in \Delta_{M,N}$ and $n_{i,\beta} \in \mathbb{Z}$ for $1 \leq i \leq c_{\beta}$.
\end{theorem}  
\begin{remark}  \label{rmk: parity}
(1) The above generators are called of Poincar\'{e}-Birkhoff-Witt type because on specialisation $q = 1$ they degenerate to PBW generators for universal enveloping algebra of Lie superalgebras \cite[Theorem 6.1.1]{PBW}. This PBW theorem will be used to argue that the set of weights of a Weyl module is always finite.

(2) We believe that the vectors in Theorem \ref{thm: pbw} with the following conditions form a basis of $U_{M,N}$:  for $1 \leq i < j \leq c_{\beta}$, $n_{i,\beta} \leq n_{j,\beta}$ if $p(\beta) = \even$ and $n_{i,\beta} < n_{j,\beta}$ if $p(\beta) = \odd$. Here $p \in \hom_{\BZ}(Q_{M,N}, \super)$ is the parity map given by: $p(\alpha_i) = \begin{cases}
\odd & \mathrm{if}\ i = M, \\
\even & \mathrm{otherwise}.
\end{cases}$  Indeed, in the paper \cite{Rosso}, the PBW basis Theorem 3.11 was obtained for the two-parameter quantum affine algebra $U_{r,s}(\widehat{\mathfrak{sl}}_n)$, with the linear independence among the PBW generators following from a general argument of Lyndon words \cite{Rosso2}. Hu-Rosso-Zhang called this PBW basis the {\it quantum affine Lyndon basis}. 

(3) For $\mathfrak{g}$ a simple finite-dimensional Lie algebra, Beck has found a convex PBW-type basis for the quantum affine algebra $U_q(\widehat{\mathfrak{g}})$ in terms of Chevalley generators (see \cite[Proposition 6.1]{Beck1} and \cite[Proposition 3]{Beck2}). When $\mathfrak{g} = \mathfrak{sl}_2$, the Drinfel'd type Borel subalgebra of $U_q(\mathcal{L}\mathfrak{sl}_2)$ can be realized as the Hall algebra of the category of coherent sheaves on the projective line $\mathbb{P}^1(\mathbb{F}_q)$. In this way, the Drinfel'd type PBW basis follows easily (\cite[Proposition 25]{Hall}).
\end{remark}     
\begin{lem}  \label{lem: commutation relations}
For $2 \leq i \leq M+N-1$ and $n \in \mathbb{Z}$, we have $[X_{\alpha_1+\cdots+\alpha_{M+N-1}}(n), X_i(0)]_{\lambda_i} = 0$ where $\lambda_i = q^{-(\epsilon_1-\epsilon_{M+N},\epsilon_i - \epsilon_{i+1})}$ (see Equation \eqref{equ: bilinear form canonnical} for the definition of the involved bilinear form).
\end{lem}
\begin{proof}
Note that the association $t_1 \mapsto K_0, e_1^{\pm} \mapsto X_{1,\pm n}^{\pm}, e_i^{\pm} \mapsto X_{i,0}^{\pm}$ for $2 \leq i \leq M+N-1$ extends to a homomorphism of superalgebras $U_q(\mathfrak{gl}(M,N)) \longrightarrow U_q(\mathcal{L}'\mathfrak{sl}(M,N))$. Lemma \ref{lem: Yamane commutation relations} applies.
\end{proof}
Let $U_{M,N}'$ be the vector subspace of $U_{M,N}$ spanned by the vectors in Theorem \ref{thm: pbw}. As these vectors are all $Q_{M,N}$-homogeneous, $U_{M,N}'$ is $Q_{M,N}$-graded. Our aim is to prove that $U_{M,N} = U_{M,N}'$, or equivalently, $(U_{M,N})_{\beta} \subseteq (U_{M,N}')_{\beta}$ for all $\beta \in Q_{M,N}^+$. Remark that $U_{M,0} = U_{M,0}', U_{0,N} = U_{0,N}'$ and $U_{1,1} = U_{1,1}'$.

\begin{prop}  \label{prop: root system}
For $\beta \in \Delta_{M,N}$, $(U_{M,N})_{\beta} = (U_{M,N}')_{\beta}$.
\end{prop}
\begin{proof}
This comes essentially from Proposition 3.10 of \cite{Rosso}, whose proof relied only on the Drinfel'd relations of degree 2.
\end{proof}

\noindent \textbf{Proof of Theorem \ref{thm: pbw}.} This is divided into three steps.

{\it Step 1: \underline{ induction hypotheses}.}
We shall prove $U_{M,N} = U_{M,N}'$ by induction on $(M,N)$. This is true when $MN = 0$ or $M+N \leq 2$. Fix $M,N \in \mathbb{Z}_{>0}$ with $M+N \geq 3$. Suppose 

\noindent \textbf{Hypothesis A.}\ {\it If $M', N' \in \mathbb{Z}_{\geq 0}$ verify $M' \leq M, N' \leq N, M'+N' < M+N$, then $U_{M',N'} = U_{M',N'}'$. }

We want to show that $(U_{M,N})_{\gamma} = (U_{M,N}')_{\gamma}$ for all $\gamma \in Q_{M,N}^+$. Define the height function $h\in \hom_{\BZ}( Q_{M,N}, \BZ)$ by $h(\alpha_i) = 1$  for all $1 \leq i \leq M+N-1$. We proceed by induction on $h(\gamma)$. From Proposition \ref{prop: root system} and   Relations \eqref{rel: Drinfel1}-\eqref{rel: Drinfeld2}, it is clear that $(U_{M,N})_{\gamma} = (U_{M,N}')_{\gamma}$ when $h(\gamma) \leq 2$. Fix $k \in \mathbb{Z}_{>2}$. Suppose

\noindent \textbf{Hypothesis B.}\ {\it If $\gamma \in Q_{M,N}^+$ and $h(\gamma) < k$, then  $(U_{M,N})_{\gamma} = (U_{M,N}')_{\gamma}$. } 

Fix $\beta \in Q_{M,N}^+$ with $h(\beta) = k$. We need to ensure $(U_{M,N})_{\beta} \subseteq (U_{M,N}')_{\beta}$.

{\it Step 2: \underline{ consequences of these hypotheses}.} To simplify notations, let  $X_{\gamma} := \sum\limits_{n \in \BZ} \mathbb{C} X_{\gamma}(n)$ for $\gamma \in \Delta_{M,N}$ and $X_i := X_{\alpha_i}$ for $1 \leq i \leq M+N-1$.

\noindent \textbf{Claim 1.} We have $X_1 (U_{M,N})_{\beta - \alpha_1} \subseteq (U_{M,N}')_{\beta}$.
\begin{proof}
This is a direct application of Hypothesis B, as $h(\beta - \alpha_1) = k-1$ and $X_1 U_{M,N}' \subseteq U_{M,N}'$.
\end{proof}
\noindent \textbf{Claim 2.} For $2 \leq i \leq M+N-1$, we have $(U_{M,N})_{\beta - \alpha_i} X_i \subseteq (U_{M,N}')_{\beta}$.
\begin{proof}
According to Hypothesis B, it suffices to verify that $U_{M,N}' X_i \subseteq U_{M,N}'$ for $2 \leq i \leq M+N-1$. From the definition of $U_{M,N}'$, we have to ensure that
\begin{displaymath}
\prod_{\gamma \in \Delta_{M,N}^1}^{\rightarrow} (\prod_{j=1}^{c_{\gamma}} X_{\gamma}(n_{j,\gamma})) X_i \subseteq U_{M,N}'
\end{displaymath}
where $\Delta_{M,N}^1 = \{\alpha_s + \cdots + \alpha_t\ |\ 2 \leq s \leq t \leq M+N-1  \}$ is an ordered subset of $\Delta_{M,N}$. We are reduced to consider the subalgebra of $U_{M,N}$ generated by the $X_s(n)$ with $2 \leq s \leq M+N-1$ and $n \in \mathbb{Z}$, which is canonically isomorphic to $U_{M-1,N}$ (Theorem \ref{thm: triangular decomposition}). Hypothesis A applies. 
\end{proof}
\noindent \textbf{Claim 3.} For $2 \leq s \leq M+N-1$, we have $X_{\alpha_1+\cdots+\alpha_s} (U_{M,N})_{\beta - \alpha_1-\cdots-\alpha_s} \subseteq (U_{M,N}')_{\beta}$.
\begin{proof}
If $\beta - \alpha_1 - \cdots - \alpha_s \in Q_{M,N}\setminus Q_{M,N}^+$, then $(U_{M,N})_{\beta - \alpha_1 - \cdots - \alpha_s} = 0$ and the LHS is 0.  If $\beta = \alpha_1+\cdots+\alpha_s$, then Proposition \ref{prop: root system} applies. We suppose from now on that $\beta - \alpha_1-\cdots-\alpha_s \in Q_{M,N}^+\setminus \{0\}$. In view of the definition of $U_{M,N}'$ and Hypothesis B, it is enough to prove that: for $1 \leq t < s \leq M+N-1$
\begin{equation}  \label{equ: little case}
X_{\alpha_1+\cdots+\alpha_s} X_{\alpha_1+\cdots+\alpha_t} (U_{M,N}')_{\beta - (\alpha_1+\cdots+\alpha_s)-(\alpha_1+\cdots+\alpha_t)} \subseteq U_{M,N}'.
\end{equation}

\textbf{Case $s < M+N-1$.}  Suppose  $\beta - (\alpha_1+\cdots+\alpha_s)-(\alpha_1+\cdots+\alpha_t) \in Q_{M,N}^+$. As $t < s < M+N-1$, $X_{\alpha_1+\cdots+\alpha_s}$ and $X_{\alpha_1+\cdots+\alpha_t}$ are both in the subalgebra of $U_{M,N}$ generated by the $X_i(n)$ with $1 \leq i \leq M+N-2$ and $n \in \BZ$, which is isomorphic to $U_{M,N-1}$. From Hypothesis A, we get
\begin{align*}
& X_{\alpha_1+\cdots+\alpha_s} X_{\alpha_1+\cdots+\alpha_t} \subseteq \sum_{i=1}^{s-1} X_{\alpha_1+\cdots+\alpha_i} (U_{M,N}')_{(\alpha_1+\cdots+\alpha_t)+(\alpha_{i+1}+\cdots+\alpha_s)},\\
& X_{\alpha_1+\cdots+\alpha_s} X_{\alpha_1+\cdots+\alpha_t} (U_{M,N}')_{\beta - (\alpha_1+\cdots+\alpha_s)-(\alpha_1+\cdots+\alpha_t)} \subseteq \sum_{i=1}^{s-1} X_{\alpha_1+\cdots+\alpha_i} (U_{M,N}')_{\beta-(\alpha_1+\cdots+\alpha_i)}.
\end{align*}
By induction on $2 \leq s \leq M+N-2$ and Claim 1, we get \eqref{equ: little case}.

\textbf{Case $s = M+N-1$.} The idea is to write the LHS of \eqref{equ: little case} as a sum of the form $\sum\limits_{r=1}^{M+N-2} X_{\alpha_1+\cdots+\alpha_r} U_{M,N}$ in order to reduce to the first case.
If $\beta - (\alpha_1 + \cdots + \alpha_{M+N-1}) - (\alpha_1+\cdots+\alpha_t) \in Q_{M,N}^+\setminus \{0\}$, then we can apply Hypothesis B to $(\alpha_1+\cdots+\alpha_{M+N-1}) + (\alpha_1+\cdots+\alpha_t)$: 
\begin{align*}
& X_{\alpha_1+\cdots+\alpha_{M+N-1}} X_{\alpha_1+\cdots+\alpha_{t}} \subseteq \sum_{i=1}^{M+N-2} X_{\alpha_1+\cdots+\alpha_i} (U_{M,N}')_{(\alpha_1+\cdots+\alpha_t) + (\alpha_{i+1} + \cdots+\alpha_{M+N-1})}  \\
& X_{\alpha_1+\cdots+\alpha_{M+N-1}} X_{\alpha_1+\cdots+\alpha_{t}} (U_{M,N}')_{\beta - (\alpha_1+\cdots+\alpha_{M+N-1}) -(\alpha_1+\cdots+\alpha_t)} \subseteq \sum_{i=1}^{M+N-2} X_{\alpha_1+\cdots + \alpha_i} (U_{M,N}')_{\beta - (\alpha_1+\cdots+\alpha_i)}.
\end{align*}  
We return henceforth to the first case $s < M+N-1$ and conclude. It remains to consider the situation $\beta = (\alpha_1+\cdots+\alpha_{M+N-1}) + (\alpha_1+\cdots+\alpha_t)$ and we are left to ensure that
\begin{displaymath}  
X_{\alpha_1+\cdots+\alpha_{M+N-1}} X_{\alpha_1+\cdots+\alpha_{t}} \subseteq U_{M,N}'
\end{displaymath}
for $1 \leq t \leq M+N-2$. When $M+N \leq 3$, this can be checked by hand. Assume from now on $M+N > 3$. 

Suppose first that $t = 1$ and $\beta = \alpha_1 + (\alpha_1+\cdots+\alpha_{M+N-1})$. From  Definition \ref{def: positive root vectors},
\begin{displaymath}
X_{\alpha_1+\cdots+\alpha_{M+N-1}}(n) = X_{\alpha_1+\cdots+\alpha_{M+N-2}}(n) X_{M+N-1}(0) \pm q_{M+N-1} X_{M+N-1}(0)X_{\alpha_1+\cdots+\alpha_{M+N-2}}(n).  
\end{displaymath} 
Noting $X_{M+N-1}(0) X_1(a) = X_1(a) X_{M+N-1}(0)$ (since $M+N-1 \geq 3$), we get 
\begin{displaymath}
X_{\alpha_1+\cdots+\alpha_{M+N-1}}(n) X_1(a) \in \pm q_{M+N-1} X_{M+N-1}(0) X_{\alpha_1+\cdots+\alpha_{M+N-2}}(n) X_1(a) + (U_{M,N})_{\beta - \alpha_{M+N-1}} X_{M+N-1}. 
\end{displaymath}
The second term of the RHS is contained in $(U_{M,N}')_{\beta}$ thanks to Claim 2. Using Hypothesis B for $\beta - \alpha_{M+N-1}$ (or Hypothesis A for the subalgebra of $U_{M,N}$ generated by the $X_{i}(n)$ with $1 \leq i \leq M+N-2$), we have
\begin{displaymath}
X_{\alpha_1+\cdots + \alpha_{M+N-2}} X_1 \subseteq \sum_{i=1}^{M+N-3} X_{\alpha_1+\cdots+\alpha_i} (U_{M,N}')_{\alpha_1+\alpha_{i+1}+\cdots+\alpha_{M+N-2}}
\end{displaymath}
and we get $X_{\alpha_1+\cdots+\alpha_{M+N-1}}(n) X_1(a)  \in X_{M+N-1}(0) \sum\limits_{i=1}^{M+N-3} X_{\alpha_1+\cdots+\alpha_i} (U_{M,N}')_{\alpha_1 + \alpha_{i+1}+\cdots+\alpha_{M+N-2}} + U_{M,N}'$. Since $X_{M+N-1} X_{\alpha_1+\cdots+\alpha_{i}} = X_{\alpha_1+\cdots+\alpha_i} X_{M+N-1}$ for $1 \leq i \leq M+N-3$, we arrive at the first case $s < M+N-1$, and the first term the RHS is contained in $U_{M,N}'$.

Suppose next $2 \leq t \leq M+N-2$, so that $\beta = (\alpha_1+\cdots+\alpha_t) + (\alpha_1+\cdots+\alpha_{M+N-1})$ and $k = M+N-1 + t$. Note that
\begin{equation}  \label{equ: root vectors approx}
X_{\alpha_1+\cdots+\alpha_t}(a) \in X_t(0)X_{t-1}(0)\cdots X_2(0) X_1(a) + \sum_{i=2}^{M+N-1} (U_{M,N})_{\alpha_1+\cdots+\alpha_t - \alpha_i} X_i
\end{equation}
in view of Claim 2. It suffices to prove that for all $a,b \in \BZ$
\begin{displaymath}
v_{a,b} := X_{\alpha_1+\cdots+\alpha_{M+N-1}}(a) X_t(0)X_{t-1}(0) \cdots X_2(0) X_1(b) \in (U_{M,N}')_{\beta}.
\end{displaymath}
From Lemma \ref{lem: commutation relations}, we deduce that
\begin{displaymath}
v_{a,b} = \pm X_t(0) X_{t-1}(0) \cdots X_2(0) X_{\alpha_1+\cdots+\alpha_{M+N-1}}(a)X_1(b).
\end{displaymath}
Applying Hypothesis B to $\alpha_1+\cdots+\alpha_{M+N-1} + \alpha_1$ (which is of height $k - (t-1)$), we get
\begin{displaymath}
X_{\alpha_1+\cdots+\alpha_{M+N-1}}(a)X_1(b) \in \sum_{i=1}^{M+N-2}X_{\alpha_1+\cdots+\alpha_i} (U_{M,N}')_{\alpha_1+\alpha_{i+1}+\cdots+\alpha_{M+N-1}}.
\end{displaymath}
Applying Hypothesis B once more to $(\alpha_2+\cdots+\alpha_t)+(\alpha_1+\cdots+\alpha_i)$ (of height $< k-1$), we get
\begin{displaymath}
v_{a,b} \in \sum_{i=1}^{M+N-2} X_{\alpha_1+\cdots+\alpha_i} (U_{M,N}')_{\beta - (\alpha_1+\cdots+\alpha_i)}.
\end{displaymath}
As desired, we return to the first case $s < M+N-1$. 
\end{proof} 
{\it Step 3: \underline{ demonstration of Theorem \ref{thm: pbw}}.} Now we are ready to show that $(U_{M,N})_{\beta} \subseteq (U_{M,N}')_{\beta}$. Remark  that
\begin{displaymath}
(U_{M,N})_{\beta} = \sum_{i=1}^{M+N-1} X_i (U_{M,N})_{\beta - \alpha_i} = \sum_{i=1}^{M+N-1} X_i(U_{M,N}')_{\beta - \alpha_i}
\end{displaymath}
where the second equality comes from Hypothesis B applied to $\beta - \alpha_i$. We are led to verify that
\begin{displaymath}
X_i X_{\alpha_1+\cdots+\alpha_s} (U_{M,N})_{\beta - \alpha_i - (\alpha_1+\cdots+\alpha_s)} \subseteq (U_{M,N}')_{\beta}
\end{displaymath}
for $2 \leq i \leq M+N-1$ and $1 \leq s \leq M+N-1$. Assume furthermore $\beta = \alpha_i + (\alpha_1+\cdots+\alpha_s)$ (using the same argument as one in the proof of the first case of Claim 3), so that $k = s+1$. When $i \geq s+1$, thanks to Proposition \ref{prop: root system} and Relation \eqref{rel: Drinfel1}, it is clear that $X_{i}X_{\alpha_1+\cdots+\alpha_s} \subseteq (U_{M,N}')_{\beta}$. Suppose $i \leq s$. If $s < M+N-1$, then we are working in the subalgebra of $U_{M,N}$ generated by the $X_i(n)$ with $1 \leq i \leq M+N-2$ and $n \in \BZ$, Hypotheses A applied. Thus assume $s = M+N-1$ and we are to show
\begin{displaymath}
X_i X_{\alpha_1+\cdots+\alpha_{M+N-1}} \subseteq (U_{M,N}')_{\beta}
\end{displaymath} 
for all $1 \leq i \leq M+N-1$. Here $\beta = \alpha_i + (\alpha_1+\cdots+\alpha_{M+N-1})$ is of height $k = M+N$.

(a) Suppose that $i=1$. In view of Relation \eqref{equ: root vectors approx},
\begin{displaymath}
X_1(a)X_{\alpha_1+\cdots+\alpha_{M+N-1}}(b) \in X_1(a)X_{M+N-1}(0)X_{M+N-2}(0)\cdots X_2(0) X_1(b) + \sum_{j=2}^{M+N-1}(U_{M,N})_{\beta - \alpha_j} X_j.
\end{displaymath}
The second term of the RHS is contained in $(U_{M,N}')_{\beta}$ thanks to Claim 2. For the first term, from Proposition \ref{prop: root system} (or Hypotheses B applied  to $\beta - \alpha_1$), we get
\begin{displaymath}
X_1(a) X_{M+N-1}(0)X_{M+N-2}(0)\cdots X_2(0) \in \sum_{j=1}^{M+N-1} X_{\alpha_1+\cdots+\alpha_j} (U_{M,N})_{\alpha_{j+1}+\cdots + \alpha_{M+N-1}}
\end{displaymath}
and $X_1(a)X_{M+N-1}(0)X_{M+N-2}(0)\cdots X_2(0) X_1(b) \in \sum\limits_{j=1}^{M+N-1} X_{\alpha_1+\cdots+\alpha_j} (U_{M,N})_{\beta -(\alpha_1+\cdots+ \alpha_j)} \subseteq (U_{M,N}')_{\beta}$ thanks to Claims 1,3.

(b) Suppose  $i=M+N-1$. Following the proof of case (a), it is enough to verify that
\begin{displaymath}
w_{a,b} := X_{M+N-1}(a) X_{M+N-1}(0)X_{M+N-2}(0)\cdots X_2(0)X_1(b) \in (U_{M,N}')_{\beta}
\end{displaymath}
for $a,b \in \BZ$. From Relations \eqref{rel: Drinfel1}-\eqref{rel: Drinfeld2}, we get 
\begin{eqnarray*}
X_{M+N-1}(0)\cdots X_2(0)X_1(b) &\in & X_{M+N-1}(a)X_{M+N-2}(-a)X_{M+N-3}(0)\cdots X_2(0) X_1 \\
&\ & + \sum_{j=2}^{M+N-1} (U_{M,N})_{\alpha_1+\cdots +\alpha_{M+N-1} -\alpha_j} X_j
\end{eqnarray*} 
and $w_{a,b} \in X_{M+N-1}(a)^2X_{M+N-2}(-a)X_{M+N-3}(0)\cdots X_2(0)X_1 + \sum\limits_{j=2}^{M+N-1}(U_{M,N})_{\beta - \alpha_j} X_j$, the second term of the RHS being contained in $(U_{M,N}')_{\beta}$ thanks to Claim 2. For the first term: either $N = 1$ and we have $X_{M+N-1}(a)^2 = 0$; or $N > 1$ and the Serre relation of degree 3 between $X_{M+N-1}(a)$ and $X_{M+N-2}(-a)$ together with Relation \eqref{rel: Drinfel1} implies that 
\begin{displaymath}
X_{M+N-1}(a)^2X_{M+N-2}(-a)X_{M+N-3}(0)\cdots X_2(0)X_1 \in (U_{M,N})_{\beta -\alpha_{M+N-1}} X_{M+N-1}.
\end{displaymath}
Thus $w_{a,b} \in \sum\limits_{j=2}^{M+N-1} (U_{M,N})_{\beta - \alpha_j} X_j \subseteq (U_{M,N}')_{\beta}$ thanks to Claim 2.

(c) Suppose at last $1 < i < M+N-1$. As in the cases (a) and (b), it suffices to verify that 
\begin{displaymath}
u_{a,b} := X_i(a)X_{M+N-1}(0)X_{M+N-2}(0) \cdots X_2(0)X_1(b) \in (U_{M,N}')_{\beta}.
\end{displaymath}
for $a,b \in \mathbb{Z}$. An argument of Relation \eqref{rel: Drinfel1} shows that
\begin{displaymath}
u_{a,b} \in X_{M+N-1}(0)\cdots X_{i+2}(0)X_i(a)X_{i+1}(0)X_{i}(0)X_{i-1}X_{i-2}\cdots X_1.
\end{displaymath}
Next, using Relation \eqref{rel: Drinfeld2} between $X_i$ and $X_{i+1}$, together with Serre relations around $X_i$ of degree 3 and Relation \eqref{rel: ocillation4}  of degree 4  when $i=M$ (which guarantees $M,N > 1$), we get
\begin{equation}  \label{rel: serre and oscillation}
X_i(a)X_{i+1}(0)X_{i}(0)X_{i-1} \subseteq (U_{M,N})_{\alpha_{i-1}+\alpha_i+\alpha_{i+1}} X_i + (U_{M,N})_{\alpha_{i-1}+2\alpha_i} X_{i+1}.
\end{equation}
Now an argument of Relation \eqref{rel: Drinfel1} and Claim 2 ensure that
\begin{displaymath}
u_{a,b} \in (U_{M,N})_{\beta - \alpha_i} X_i + (U_{M,N})_{\beta - \alpha_{i+1}} X_{i+1} \subseteq (U_{M,N}')_{\beta}.
\end{displaymath}
This completes the  proof of Theorem \ref{thm: pbw}.  \hfill  $\Box$

\begin{remark}
In the proof of Theorem \ref{thm: pbw},  Relations \eqref{rel: Drinfel1}-\eqref{rel: Drinfeld2} were used repeatedly. The Serre relations of degree 3 appeared first in case (b) of Step 3. Then, to prove  \eqref{rel: serre and oscillation}, we find the Serre relations of degree 3 and  the oscillation relation of degree 4 (when $M,N > 1$) indispensable.
\end{remark}

\section{Representations of $U_q(\mathcal{L}\mathfrak{sl}(M,N))$}  \label{sec: 4}
In this section, we consider the analogy of Theorem \ref{thm: Weyl modules for quantum affine algebras} for quantum affine superalgebras. As we shall see, in the super case, corresponding to the odd isotopic root $\alpha_M$, Drinfel'd polynomials have to be replaced by formal series with torsion.
\subsection{Highest weight representations}


From Theorem \ref{thm: Weyl modules for quantum affine algebras}, we see that the highest weights of finite-dimensional simple $U_q(\mathcal{L}\mathfrak{sl}_N)$-modules are essentially elements of $(1 + z \mathbb{C}[z])^{N-1}$. For quantum affine superalgebras, this set is replaced by $\mathcal{R}_{M,N}$. 
\begin{definition}  \label{def: highest weight}
Define $\mathcal{R}_{M,N}$ to be the set of triples $(\underline{P}, f, c)$ such that:
\begin{itemize}
\item[(a)] $f(z) = \sum\limits_{n \in \mathbb{Z}} f_n z^n \in \mathbb{C}[[z,z^{-1}]]$ is a formal series annihilated by a non-zero polynomial;
\item[(b)] $c \in \mathbb{C} \setminus \{0\}$ with $\frac{c - c^{-1}}{q-q^{-1}} = f_0$;
\item[(c)] $\underline{P} = (P_i)_{1 \leq i \leq M+N-1, i \neq M}$ with $P_i(z) \in 1 + z \mathbb{C}[z]$ for all $1 \leq i \leq M+N-1$ and $i \neq M$. 
\end{itemize}
Define also $\tilde{\mathcal{R}}_{M,N}$ to be the set of $(\underline{P}, f, c; Q)$ such that: $(\underline{P},f,c) \in \mathcal{R}_{M,N}$ and 
\begin{itemize}
\item[(d)] $Q(z) \in 1 + z\mathbb{C}[z]$ with $Q(z) f(z) = 0$.
\end{itemize}
\end{definition}
For convention, we admit that $\mathcal{R}_{N,0} = \mathcal{R}_{0,N} = \tilde{\mathcal{R}}_{0,N} = \tilde{\mathcal{R}}_{N,0} = (1 + z\mathbb{C}[z])^{N-1}$.
\paragraph{Verma modules.} Let $(\underline{P},f,c) \in \mathcal{R}_{M,N}$. The  Verma module, denoted by $\Verma(\underline{P},f,c)$, is the $\Ud$-module generated by $v_{(\underline{P},f,c)}$ of $\super$-degree $\even$ subject to the relations
\begin{align}
& X_{i,n}^+ v_{(\underline{P},f,c)} = 0 \ \ \mathrm{for}\ 1 \leq i \leq M+N-1,\ n \in \mathbb{Z}, \label{rel: highest weight vector 1}  \\
&\sum_{n \in \BZ} \phi_{i,n}^{\pm}z^n v_{(\underline{P},f,c)} = q_i^{\deg P_i} \frac{P_i(z q_i^{-1})}{P_i(z q_i)} v_{(\underline{P},f,c)} \in \mathbb{C}v_{(\underline{P},f,c)} [[z^{\pm 1}]] \ \ \ \mathrm{for}\ 1 \leq i \leq M+N-1, i \neq M,  \label{rel: hwv2} \\
& K_M v_{(\underline{P},f,c)} =  c v_{(\underline{P},f,c)},\ \sum_{n \in \BZ} \frac{\phi_{M,n}^+ - \phi_{M,n}^-}{q-q^{-1}}z^n v_{(\underline{P},f,c)} = f(z) v_{(\underline{P},f,c)} \in \mathbb{C}v_{(\underline{P},f,c)} [[z,z^{-1}]].  \label{rel: hwv3}
\end{align} 
Note that $\Verma(\underline{P},f,c)$ has a natural $U_q(\mathcal{L}'\mathfrak{sl}(M,N))$-module structure by demanding
\begin{align}  
K_0 v_{(\underline{P},f,c)} = v_{(\underline{P},f,c)}.  \label{rel: auxillary}
\end{align}
From the triangular decomposition of $U_q(\mathcal{L}\mathfrak{sl}(M,N))$, we have an isomorphism of vector superspaces 
\begin{equation}  \label{equ: isom Verma}
U_q^-(\mathcal{L}\mathfrak{sl}(M,N)) \longrightarrow \Verma(\underline{P},f,c), \ \ x \mapsto x v_{(\underline{P},f,c)}.
\end{equation}
Later in \S \ref{sec: final section}, we will write Relation \eqref{rel: hwv3} in a form similar to Relation \eqref{rel: hwv2}. See Equation \eqref{equ: highest weight vector odd}. 
\paragraph{Weyl modules.} Let $(\underline{P},f,c;Q) \in \tilde{\mathcal{R}}_{M,N}$. The Weyl module, $\Weyl(\underline{P},f,c;Q)$, is the $\Ud$-module generated by $v_{(\underline{P},f,c)}$ of $\super$-degree $\even$ subject to Relations \eqref{rel: highest weight vector 1}-\eqref{rel: hwv3} and 
\begin{align}  
&(X_{i,0}^-)^{1+\deg P_i} v_{(\underline{P},f,c)} = 0 \ \ \ \mathrm{for}\ 1 \leq i \leq M+N-1, i \neq M, \label{rel: weyl even} \\
&\sum_{s=0}^d a_{d-s} X_{M,s}^- v_{(\underline{P},f,c)} = 0 \ \ \ \mathrm{where\ we\ understand}\ Q(z) = \sum_{s=0}^d a_s z^s \in 1 + z \BC[z].  \label{rel: Weyl odd}   
\end{align}
For the convention, when $(M,N) = (1,1)$, we shall replace Relation \eqref{rel: Weyl odd} by the following:
\begin{align}
&\sum_{s=0}^d a_{d-s} X_{M,s+n}^- v_{(\underline{P},f,c)} = 0 \ \ \ \mathrm{for\ all}\ n \in \BZ.  \label{rel: Weyl odd for M=N=1}
\end{align}
Note that $\Weyl(\underline{P},f,c;Q)$ is endowed with an $U_q(\mathcal{L}'\mathfrak{sl}(M,N))$-module structure through Relation \eqref{rel: auxillary}.
\paragraph{Simple modules.} Let $(\underline{P},f,c) \in \mathcal{R}_{M,N}$. From the isomorphism \eqref{equ: isom Verma} we see that the  $U_q(\mathcal{L}'\mathfrak{sl}(M,N))$-module $\Verma(\underline{P},f,c)$ has a weight space decomposition (see Notation \ref{notation: weight})
\begin{eqnarray*}
&&\Verma(\underline{P},f,c) = \bigoplus_{\mu \in \lambda_{(\underline{P},f,c)} - Q_{M,N}^+} (\Verma(\underline{P},f,c))_{\mu}\ \ \mathrm{with}\\
&& (\Verma(\underline{P},f,c))_{\mu} = \{ x \in \Verma(\underline{P},f,c)\ |\ K_i x = \mu(K_i) x \ \mathrm{for}\ 0 \leq i \leq M+N-1 \}
\end{eqnarray*}
where $\lambda_{(\underline{P},f,c)} \in P_{M,N}$ is given by: $K_0 \mapsto 1$; $K_M \mapsto c$; $K_i \mapsto q_i^{\deg P_i}$ for $1 \leq i \leq M+N-1$ and $i \neq M$. In particular, $(\Verma(\underline{P},f,c))_{\lambda_{(\underline{P},f,c)}} = \mathbb{C} v_{(\underline{P},f,c)}$ is  one-dimensional. In consequence, there is a unique quotient of $\Verma(\underline{P},f,c)$ which is  simple as a $U_q(\mathcal{L}'\mathfrak{sl}(M,N))$-module. This leads to the following
\begin{definition}   \label{def: simple modules}
For $(\underline{P},f,c) \in \mathcal{R}_{M,N}$, let $\simple'(\underline{P},f,c)$ be the simple quotient of $\Verma(\underline{P},f,c)$ as $U_q(\mathcal{L}'\mathfrak{sl}(M,N))$-module.  
\end{definition} 
\begin{remark}  \label{remark: weyl simple}
(1) $\simple'(\underline{P},f,c)$ is  not necessarily simple as a $\Ud$-module.

(2) By definition, we have natural epimorphisms of $U_q(\mathcal{L}'\mathfrak{sl}(M,N))$-modules: 
\begin{displaymath}
\Verma(\underline{P},f,c) \twoheadrightarrow \simple'(\underline{P},f,c),\ \ \Verma(\underline{P},f,c) \twoheadrightarrow \Weyl(\underline{P},f,c;Q),  v_{(\underline{P},f,c)} \mapsto v_{(\underline{P},f,c)}
\end{displaymath}
for all $(\underline{P},f,c;Q) \in \tilde{\mathcal{R}}_{M,N}$. If $\Weyl(\underline{P},f,c;Q) \neq 0$, then the first epimorphism factorises through the second. We shall see in the next section that this is indeed always the case.

\end{remark}
\begin{lem}  \label{lem: extension}
If $M \neq N$, then $\simple'(\underline{P},f,c)$ is a simple $\Ud$-module for all $(\underline{P},f,c) \in \mathcal{R}_{M,N}$.
\end{lem}
\begin{proof}
Let $A_{M,N}$ be the subalgebra of $\Ud$ generated by $K_i^{\pm 1}$ for $1 \leq i \leq M+N-1$. As in Notation \ref{notation: weight}, there is a unique abelian group structure over $\Alg(A_{M,N}, \mathbb{C})$ so that the inclusion $A_{M,N} \hookrightarrow A_{M,N}'$ induces a group homomorphism $\iota: P_{M,N} \longrightarrow \Alg(A_{M,N},\mathbb{C}), \alpha \mapsto \alpha|_{A_{M,N}}$.

If $M \neq N$, then the restriction $\iota|_{Q_{M,N}}: Q_{M,N} \longrightarrow \Alg(A_{M,N}, \mathbb{C})$ is  injective. It follows that the decomposition in weight spaces of $\Verma(\underline{P},f,c)$ with respect to $\Alg(A_{M,N}',\mathbb{C})$ is exactly one with respect to $\Alg(A_{M,N},\mathbb{C})$. Thus, all sub-$\Ud$-modules of $\Verma(\underline{P},f,c)$ are sub-$U_q(\mathcal{L}'\mathfrak{sl}(M,N))$-modules.
\end{proof}

\subsection{Main result}   \label{sec: 4.2}
In this section, we shall see that the Weyl modules we defined before are always finite-dimensional and non-zero, a generalisation of Theorem \ref{thm: Weyl modules for quantum affine algebras} (a). More precisely, we have
\begin{theorem}  \label{thm: main}
For all $(\underline{P},f,c;Q) \in \tilde{\mathcal{R}}_{M,N}$, we have $\deg Q < \dim \Weyl(\underline{P},f,c;Q) < \infty$.
\end{theorem}
As an immediate consequence (Remark \ref{remark: weyl simple})
\begin{cor}
For all $(\underline{P},f,c;Q) \in \tilde{\mathcal{R}}_{M,N}$, there are epimorphisms of $U_q(\mathcal{L}'\mathfrak{sl}(M,N))$-modules
\begin{displaymath}
\Verma(\underline{P},f,c) \twoheadrightarrow \Weyl(\underline{P},f,c;Q) \twoheadrightarrow \simple'(\underline{P},f,c).
\end{displaymath}
In particular, $\simple'(\underline{P},f,c)$ is  finite-dimensional.
\end{cor}

\begin{remark}
 When $MN = 0$, we understand $Q(z) = 1$, and Theorem \ref{thm: main} above becomes Theorem \ref{thm: Weyl modules for quantum affine algebras} (a). When $M=N=1$,  $U_q^-(\mathcal{L}\mathfrak{sl}(1,1))$ is an exterior algebra,  and Relation \eqref{rel: Weyl odd for M=N=1} guarantees that $\Weyl(f,c;Q)$ be finite-dimensional. 
\end{remark}
To prove Theorem \ref{thm: main}, one can assume  $M > 1, N \geq 1$ due to the following:

\begin{lem} \label{lem: automorphism of Dynkin diagrams}
Suppose $MN > 0$. The following defines a superalgebra isomorphism: 
\begin{displaymath}
\begin{cases}
\pi_{M,N}: U_q(\mathcal{L}\mathfrak{sl}(M,N)) \longrightarrow U_q(\mathcal{L}\mathfrak{sl}(N,M)) \\
K_i \mapsto K_{M+N-i}^{-1},\ X_{i,n}^{+} \mapsto X_{M+N-i, -n}^{+},\ X_{i,n}^- \mapsto (-1)^{p(\alpha_i)} X_{M+N-i, -n}^-,\ h_{i,s} \mapsto (-1)^{p(\alpha_i)} h_{M+N-i,-s}
\end{cases}
\end{displaymath}
for $1 \leq i \leq M+N-1, n \in \BZ, s \in \BZ_{\neq 0}$. Here $p \in \hom_{\BZ}(Q_{M,N},\super)$ is the parity map in Remark \ref{rmk: parity}.
\end{lem}
\begin{proof}
This comes directly from  Definition \ref{def: Drinfeld presentation} of $U_q(\mathcal{L}\mathfrak{sl}(M,N))$.
\end{proof}
We remark that the isomorphism $\pi_{M,N}$ respects the corresponding triangular decompositions of $U_q(\mathcal{L}\mathfrak{sl}(M,N))$ and $U_q(\mathcal{L}\mathfrak{sl}(N,M))$. Hence, $\pi_{M,N}^{*}$ of a Verma/Weyl module over $U_q(\mathcal{L}\mathfrak{sl}(N,M))$ is again a Verma/Weyl module over $U_q(\mathcal{L}\mathfrak{sl}(M,N))$. 

\noindent \textbf{Proof of Theorem \ref{thm: main}.} This is divided into two parts. We fix notations first. Let $(\underline{P},f,c;Q) \in \tilde{\mathcal{R}}_{M,N}$ with $f(z) = \sum\limits_{n \in \BZ} f_n z^n$ and $Q(z) = \sum\limits_{s=0}^d a_s z^s$ of degree $d$. Let $\Verma := \Verma(\underline{P},f,c)$ be the Verma module over $U_q(\mathcal{L}'\mathfrak{sl}(M,N))$. Let $v := v_{(\underline{P},f,c)} \in \Verma$. Let $\lambda := \lambda_{(\underline{P},f,c)} \in P_{M,N}$ be given by: $K_0 \mapsto 1$; $K_M \mapsto c$; $K_i \mapsto q_i^{\deg P_i}$ for $1 \leq i \leq M+N-1$ and $i \neq M$. Let $\Weyl := \Weyl(\underline{P},f,c;Q)$. 

\noindent {\it Part I. \underline{Non-triviality of Weyl modules}.} As noted in the preceding section, $\Verma = \bigoplus\limits_{\mu \in \lambda - Q_{M,N}^+}(\Verma)_{\mu}$ where 
\begin{displaymath}
(\Verma)_{\mu} = \{w \in \Verma \ |\  K_i w = \mu(K_i) w \ \textrm{for}\ 0 \leq i \leq M+N-1 \},\ (\Verma)_{\lambda} = \mathbb{C} v .
\end{displaymath} 
Furthermore, $(U_q(\mathcal{L}'\mathfrak{sl}(M,N)))_{\alpha} (\Verma)_{\mu} \subseteq (\Verma)_{\mu + \alpha}$ for $\alpha \in Q_{M,N}$ (see Notation \ref{notation: weight}). By definition, $\Weyl$ is the quotient of $\Verma$ by the sub-$U_q(\mathcal{L}'\mathfrak{sl}(M,N))$-module $J$ generated by $v_M := \sum\limits_{s=0}^d a_{d-s} X_{M,s}^- v$  and $v_i \triangleq (X_{i,0}^-)^{1 + \deg P_i} v$ where $1 \leq i \leq M+N-1$ and $i \neq M$. (Here we use the assumption that $M > 1, N \geq 1$.) Since $v_M \in (\Verma)_{\lambda - \alpha_M}$ and $v_i \in (\Verma)_{\lambda - (1+\deg P_i)\alpha_i}$, $J$ is $P_{M,N}$-graded,  and so is $\Weyl$:
\begin{displaymath}
\Weyl = \Verma/J = \bigoplus_{\mu \in \lambda - Q_{M,N}^+} (\Weyl)_{\mu}.
\end{displaymath}
Let $J_1 := U_q^+(\mathcal{L}\mathfrak{sl}(M,N)) (\sum\limits_{i=1}^{M+N-1} \mathbb{C}v_i) \subseteq J$. We want to find $(J_1)_{\lambda}$ and $(J_1)_{\lambda - \alpha_M}$. Indeed
\begin{displaymath}
(J_1)_{\lambda} = X_M^+ v_{M} + \sum_{i \neq M} (X_i^+)^{1+\deg P_i} v_i,\ (J_1)_{\lambda - \alpha_M} = \mathbb{C} v_M
\end{displaymath}
where $X_i^+ := \sum\limits_{n \in \BZ} \mathbb{C}X_{i,n}^+$ (we have used these $X_i^+$ in the proof of Theorem \ref{thm: pbw}).

\noindent \textbf{Claim.} $(J_1)_{\lambda} = 0$.
\begin{proof}
We have $ X_{M,n}^+ v_M  = \sum\limits_{s=0}^d a_{d-s} [X_{M,n}^+, X_{M,s}^-]v = \sum\limits_{s=0}^d a_{d-s} f_{n+s}v = 0$ as $(\sum\limits_{s=0}^d a_s z^s)(\sum\limits_m f_m z^m) = 0$.
For $i \neq M$, let $\widehat{U_i}$ be the subalgebra of $U_q(\mathcal{L}'\mathfrak{sl}(M,N))$ generated by the $X_{i,m}^{\pm}, K_i^{\pm 1}, h_{i,s}$ with $m \in \BZ$ and $s \in \BZ_{\neq 0}$. The subspace $\widehat{U_i} v$ of $\textbf{M}$ is a quotient of the Verma module $\textbf{M}(P_i)$ over $U_{q_i}(\mathcal{L}\mathfrak{sl}_2)$ of highest weight $P_i$.  Theorem \ref{thm: Weyl modules for quantum affine algebras} (a) forces that $(X_i^+)^{1+\deg P_i} v_i = (X_i^{+})^{1+\deg P_i} (X_{i,0}^-)^{1+\deg P_i} v = 0$, as it must be in the Verma module $\textbf{M}(P_i)$. 
\end{proof}

From the triangular decomposition of $U_q(\mathcal{L}'\mathfrak{sl}(M,N))$, we see that $J = U_q^-(\mathcal{L}\mathfrak{sl}(M,N)) U_q^0(\mathcal{L}'\mathfrak{sl}(M,N)) J_1$ and
$(J)_{\lambda} = 0,\ (J)_{\lambda - \alpha_M} = U_q^0(\mathcal{L}'\mathfrak{sl}(M,N))v_M$. 
As $M > 1$,  $(J)_{\lambda - \alpha_M} = \sum\limits_{n \in \mathbb{Z}} \mathbb{C} v_M(n)$ where $v_M(n) := \sum\limits_{s=0}^d a_{d-s} X_{M,s+n}^- v$.
Using the isomorphism \eqref{equ: isom Verma} and the defining relations of $U_q^-(\mathcal{L}\mathfrak{sl}(M,N))$, we conclude that $(\Verma)_{\lambda - \alpha_M} = \bigoplus\limits_{n \in \BZ} \mathbb{C} X_{M,n}^- v$ and $(\Weyl)_{\lambda - \alpha_M}$ has a presentation as vector space
\begin{equation}
(\Weyl)_{\lambda - \alpha_M} = (\Verma)_{\lambda - \alpha_M}/(J)_{\lambda - \alpha_M} = \mathrm{Vect}\langle X_{M,n}^-v, n \in \BZ\ |\  \sum_{s=0}^d a_{d-s} X_{M,n+s}^- v = 0\ \textrm{for}\ n \in \BZ  \rangle.
 \end{equation}
Remark that $h_{M-1,1} X_{M,n}^- v = X_{M,n+1}^-v + \theta_{M-1} X_{M,n}^- v$ with $\theta_{M-1} = - \mathrm{Res}(z^{-2}P_{M-1}(z)) dz$.  Conclude

\begin{prop}   \label{prop: uniqueness of Weyl modules}
Suppose $M > 1, N \geq 1$. Using the notations above, we have $\dim (\Weyl)_{\lambda} = 1$ and $\dim (\Weyl)_{\lambda - \alpha_M} = d = \deg Q$. Moreover, $Q(z-\theta_{M-1})$ is the characteristic polynomial of $h_{M-1,1} \in \mathrm{End}((\Weyl)_{\lambda - \alpha_M})$.
\end{prop} 

\begin{remark}  \label{rmk: properties of Weyl modules}
(1) This proves the first part of Theorem \ref{thm: main}: $\deg Q < \dim \Weyl$. 

(2) The proposition above also says that the polynomial $Q(z)$ can be reconstructed from the $U_q(\mathcal{L}'\mathfrak{sl}(M,N))$-module structure on $\Weyl$. (Similarly, when $M \neq N$, $Q(z)$ can be deduced from the $U_q(\mathcal{L}\mathfrak{sl}(M,N))$-module structure on $\Weyl$. See the proof of Lemma \ref{lem: extension}.) The same goes for $\underline{P}$ by using the theory of Weyl modules over $U_q(\mathcal{L}\mathfrak{sl}_2)$, and for $(f,c)$ in view of Relation \eqref{rel: hwv3}. In conclusion: 

\noindent {\it the parameter $(\underline{P},f,c;Q)$ 
 is uniquely determined by the $U_q(\mathcal{L}'\mathfrak{sl}(M,N))$-module structure (or $U_q(\mathcal{L}\mathfrak{sl}(M,N))$-module structure when $M \neq N$) on $\Weyl(\underline{P},f,c;Q)$.}
\end{remark}

\noindent {\it Part II. \underline{Dimension of Weyl modules}.} As in the proof of Theorem \ref{thm: pbw}, we use induction on $(M,N)$. Suppose that the theorem is true for $(M',N')$ such that $M' \leq M, N' \leq N$ and $M'+N' < M+N$.  We adapt the notations above and assume  $M > 1, N \geq 1$. It remains to show that $\Weyl$ is finite-dimensional. Let $P$ be the set of weights:
\begin{displaymath}
P := \{ \mu \in P_{M,N}\ |\ (\Weyl)_{\mu} \neq 0 \}.
\end{displaymath}
Then $P \subseteq \lambda - Q_{M,N}^+$. As $\Weyl$ is $P_{M,N}$-graded, it suffices to prove the following:
\begin{itemize}
\item[(1)] for all $\mu \in P$, $(\Weyl)_{\mu}$ is finite-dimensional;
\item[(2)] $P$ is a finite subset of $P_{M,N}$.
\end{itemize}
For (1),  $(\Weyl)_{\lambda - \alpha_M}$ is of dimension $\deg Q$. For $i \neq M$, $(\Weyl)_{\lambda - \alpha_i}$ being  finite-dimensional comes from Relation \eqref{rel: weyl even} and the theory of Weyl modules over $U_q(\mathcal{L}\mathfrak{sl}_2)$. One can copy (word by word) the proof of  \cite[\S 5, (b)]{CP2}, or that of  \cite[Proposition 4.4]{CP3} where only the Drinfel'd relations of degree 2 were involved.

We proceed to verifying (2). First, by using the isomorphism $\tau_1: U_q^+(\mathcal{L}\mathfrak{sl}(M,N)) \longrightarrow U_q^-(\mathcal{L}\mathfrak{sl}(M,N))$ in Corollary \ref{cor: canonical isomorphisms} and the root vectors in Definition \ref{def: positive root vectors}, we define: for $\beta \in \Delta_{M,N}$ and $n \in \BZ$
\begin{displaymath}
X_{\beta}^-(n) := \tau_1 (X_{\beta}(-n)),\ \ X_{\beta}^- := \sum_{n \in \BZ} \mathbb{C} X_{\beta}^-(n),\ \ X_i^- :=  X_{\alpha_i}^-.
\end{displaymath} 
Theorem \ref{thm: pbw} says that 
\begin{displaymath}
U_q^-(\mathcal{L}\mathfrak{sl}(M,N)) = \sum_{d_{\beta} \geq 0} \prod_{\beta \in \Delta_{M,N}}^{\rightarrow} (X_{\beta}^-)^{d_{\beta}} = \sum_{d_i} (X_1^-)^{d_1}(X_{\alpha_1+\alpha_2}^-)^{d_2} \cdots (X_{\alpha_1+\cdots+\alpha_{M+N-1}}^-)^{d_{M+N-1}} U_{M-1,N}^-
\end{displaymath}
where $U_{M-1,N}^-$ is the subalgebra of $U_q^-(\mathcal{L}\mathfrak{sl}(M,N))$ generated by the $X_{i,n}^-$ with $2 \leq i \leq M+N-1$ and $n \in \BZ$. According to Theorem \ref{thm: triangular decomposition}, $U_{M-1,N}^-$ is isomorphic to $U_q^-(\mathcal{L}\mathfrak{sl}(M-1,N))$ as superalgebras.

\noindent \textbf{Claim.} $U_{M-1,N}^- v$ is finite-dimensional.
\begin{proof}
Let $U_{M-1,N}$ be the subalgebra of $U_q(\mathcal{L}\mathfrak{sl}(M,N))$ generated by the $X_{i,n}^{\pm}, K_i^{\pm 1}, h_{i,s}$ with $2 \leq i \leq M+N-1, n \in \BZ, s \in \BZ_{\neq 0}$. Then $U_{M-1,N} \cong U_q^-(\mathcal{L}\mathfrak{sl}(M-1,N))$ as superalgebras. Moreover, $U_{M-1,N} v$ can be realised as a quotient of the the Weyl module $\Weyl((P_i)_{i \geq 2},f,c;Q)$ over $U_q(\mathcal{L}\mathfrak{sl}(M-1,N))$. From the induction hypothesis, $U_{M-1,N} v$ is finite-dimensional. Note that $U_{M-1,N}^-v = U_{M-1,N}v$.  
\end{proof}  
Let $C_1 := \dim U_{M-1,N}^-v$. Then as a subspace of $\Weyl$, $(U_{M-1,N}^-v)_{\lambda - \sum\limits_{i=2}^{M+N-1}c_i \alpha_i } = 0$ if $c_i \geq C_1$ for some $2 \leq i \leq M+N-1$. In consequence, if $\mu = \lambda - \sum\limits_{i=1}^{M+N-1} u_i \alpha_i \in P$, then 
\begin{displaymath}
\mu = \lambda - \sum_{i=1}^{M+N-1} e_i (\alpha_1+\cdots+\alpha_i) - \sum_{i=2}^{M+N-1} f_i \alpha_i 
\end{displaymath}
with $e_i \geq 0$ and $0 \leq  f_i < C_1$. It follows that 
\begin{equation}  \label{estimation 1}
u_i - u_j > -C_1 \ \ \ \mathrm{for}\ 1 \leq i < j \leq M+N-1.
\end{equation}

On the other hand, using the anti-automorphism $\tau_2$ of Corollary \ref{cor: canonical isomorphisms}, we can also write 
\begin{displaymath}
U_q^-(\mathcal{L}\mathfrak{sl}(M,N)) = U_{M-1,N}^- (\sum_{d \geq 0} (X_{\alpha_1+\cdots+\alpha_{M+N-1}}^-)^d ) U_{M,N-1}^-
\end{displaymath}
where $U_{M,N-1}^-$ is the subalgebra of $U_q(\mathcal{L}\mathfrak{sl}(M,N))$ generated by the $X_{i,n}^-$ with $1 \leq i \leq M+N-2$ and $n \in \BZ$. Similar argument as in the proof of the claim above shows that $U_{M,N-1}^-v$ is a finite-dimensional subspace of $\Weyl$. Let $C_2 = \dim U_{M,N-1}^-v$. Then
\begin{displaymath}
\mu = \lambda - \sum_{i=2}^{M+N-1}e'_i \alpha_i - e'_1 (\alpha_1+\cdots+\alpha_{M+N-1}) - \sum_{j=1}^{M+N-2} f_j' \alpha_j
\end{displaymath}
for some $e_i' \geq 0$ and $0 \leq f_j' < C_2$. It follows that
\begin{equation}  \label{estimation 2}
u_2 - u_1 > -C_2.
\end{equation}
Now  inequalities \eqref{estimation 1} and \eqref{estimation 2} imply that $|u_1 - u_2| < \mathrm{max}\{C_1,C_2\}$ for all $\mu = \lambda - \sum\limits_{i=1}^{M+N-1}u_i \alpha_i \in P$. In particular, $\mu + s \alpha_1 \notin P$ when $|s|\gg 0$.  Hence, $X_{1,0}^{+}$ and $X_{1,0}^{-}$ are locally nilpotent operators on $\Weyl$. Let $U_0$ be the subalgebra of $U_q(\mathcal{L}'\mathfrak{sl}(M,N))$ generated by the $X_{1,0}^{\pm}, K_i^{\pm 1}$ with $0 \leq i \leq M+N-1$. Then $U_0$ is an enlargement of $U_q(\mathfrak{sl}_2)$. From the theory of integrable modules over $U_q(\mathfrak{sl}_2)$ we see that $\mu \in P$ implies $s_1(\mu) \in P$. Here, for $\mu = \lambda - \sum\limits_{i=1}^{M+N-1} u_i \alpha_i$, we have $s_1(\mu) = \lambda - (\deg P_1 -u_1+u_2) \alpha_1 - \sum\limits_{i=2}^{M+N-1} u_i \alpha_i$. In view of \eqref{estimation 1}, 
\begin{equation}  \label{estimation 3}
(\deg P_1 - u_1 + u_2) - u_2 > -C_1.
\end{equation}
Now the three inequalities \eqref{estimation 1}-\eqref{estimation 3} say that all the $u_i$ are bounded by a constant. In other words, $P$ is finite. This completes the proof of Theorem \ref{thm: main}. \hfill  $\Box$

\begin{remark}
(1) Our proof relied heavily on the theory of Weyl modules over $U_q(\mathcal{L}\mathfrak{sl}_2)$. Using PBW generators, we deduced the integrability property of Weyl modules: the actions of $X_{i,0}^{\pm}$ for $1 \leq i \leq M+N-1$  are locally nilpotent.  Even in the non-graded case of quantum affine algebras considered in \cite{CP3}, the integrability property (Theorem \ref{thm: Weyl modules for quantum affine algebras}) is  non-trivial (see the references therein). 

(2) From integrability, we get an action of Weyl group on the set $P$ of weights \cite[\S 41.2]{Lusztig}. In the non-graded case, the action of Weyl group already forces that $P$ be finite (argument of Weyl chambers). In our case, the Weyl group, being $\mathfrak{S}_{M} \times \mathfrak{S}_N$, is not enough to ensure the finiteness of $P$. And once again, we used  PBW generators to obtain further information on $P$.
\end{remark}
\subsection{Classification of finite-dimensional simple representations}
In this section, we show that all finite-dimensional simple modules of $U_q(\mathcal{L}'\mathfrak{sl}(M,N))$ (or $\Ud$ when $M \neq N$) are almost of the form $\simple'(\underline{P},f,c)$ with $(\underline{P},f,c) \in \mathcal{R}_{M,N}$, a super-version of Theorem \ref{thm: Weyl modules for quantum affine algebras} (b).
\begin{lem}   \label{lem: finite-dimensional modules}
Suppose $MN > 0$ and $(M,N) \neq (1,1)$. Let $V$ be a finite-dimensional non-zero $U_q(\mathcal{L}'\mathfrak{sl}(M,N))$-module. Then there exist a $\super$-homogeneous vector $v \in V \setminus \{0\}$, $\varepsilon_i \in \{\pm 1\}$ for $1 \leq i \leq M+N-1, i \neq M$, $t \in \mathbb{C}\setminus \{0\}$ and $(\underline{P},f,c;Q) \in \tilde{\mathcal{R}}_{M,N}$ satisfying Relations \eqref{rel: highest weight vector 1},\eqref{rel: hwv3},\eqref{rel: weyl even}, \eqref{rel: Weyl odd}, $K_0 v = t v$  and
\begin{displaymath}
\sum_{n\in \mathbb{Z}} \phi_{i,n}^{\pm} z^n v = \varepsilon_i q_i^{\deg P_i} \frac{P_i(z q_i^{-1})}{P_i(z q_i)} v \in V[[z^{\pm 1}]]
\end{displaymath}
for $1 \leq i \leq M+N-1, i \neq M$.
\end{lem}
\begin{proof}
We follow the proof of \cite[Proposition 3.2]{CP1} step by step. From the finite-dimensional representation of the commutative algebra $A_{M,N}'$ on $V$, one finds $\lambda \in P_{M,N}$ such that 
\begin{displaymath}
(V)_{\lambda} := \{ w \in V\ |\  K_i w = \lambda(K_i) w \ \mathrm{for}\ 0 \leq i \leq M+N-1 \} \neq 0.
\end{displaymath}
By replacing $V$ with $U_q(\mathcal{L}'\mathfrak{sl}(M,N)) (V)_{\lambda}$, one can suppose that $V$ is $P_{M,N}$-graded:
$V = \bigoplus\limits_{\mu \in P_{M,N}} (V)_{\mu}$.
Let $P := \{\mu \in P_{M,N}\ |\ (V)_{\mu} \neq 0 \}$. Then $P$ is a finite set. There exists $\lambda_0 \in P$ such that $\lambda_0 + \alpha_i \notin P$ for all $1 \leq i \leq M+N-1$. (Here we really need the fact that these $\alpha_i$ are linearly independent.) Note that $(V)_{\lambda_0}$ is also $\super$-graded. Moreover, $(V)_{\lambda_0}$ is stable by the commutative subalgebra $U_q^0(\mathcal{L}'\mathfrak{sl}(M,N))$. One can therefore choose a non-zero $\super$-homogeneous vector $v \in (V)_{\lambda_0}$ which is a common eigenvector of $U_q^0(\mathcal{L}'\mathfrak{sl}(M,N))$. In particular, $X_{i,n}^+ v = 0$ for all $1 \leq i \leq M+N-1$ and $n \in \BZ$, and $K_0 v = \lambda_0(K_0)v$.
 
When $i \neq M$, let $\widehat{U_i}$ be the subalgebra generated by the $X_{i,n}^{\pm}, K_i^{\pm 1}, h_{i,s}$ with $n \in \BZ$ and $s \in \BZ_{\neq 0}$. Then $\widehat{U_i} \cong U_{q_i}(\mathcal{L}\mathfrak{sl}_2)$ as algebras, and $\widehat{U_i} v$ is a finite-dimensional highest weight $U_{q_i}(\mathcal{L}\mathfrak{sl}_2)$-module. One can thus find $(\varepsilon_i, P_i)$ satisfying the above relation thanks to Theorem \ref{thm: Weyl modules for quantum affine algebras} (b).

When $i = M$, by definition of $v$, there exists $f_n \in \mathbb{C}$ for all $n \in \BZ$ such that $\frac{\phi_{M,n}^+ - \phi_{M,n}^-}{q-q^{-1}} v = f_n v$. On the other hand, as $X_M^- v$ is finite-dimensional, there exist $m \in \BZ, d \in \mathbb{Z}_{\geq 0}$ and $a_0,\cdots,a_d \in \mathbb{C}$ such that 
\begin{displaymath}
 a_d \neq 0,\ a_0 = 1,\ \sum_{s=0}^d a_{d-s} X_{M,s+m}^- v = 0.
\end{displaymath} 
By applying  $h_{M-1,t}$ to the above equation and noting that $h_{M-1,t} v \in \mathbb{C} v,\ [h_{M-1,t}, X_{M,s+m}^-] = \frac{[t]_q}{t} X_{M,s+m+t}^-$ we get Relation \eqref{rel: Weyl odd} with respect to the polynomial $Q(z) = \sum\limits_{s=0}^d a_s z^s$. By applying $X_{M,0}^+$ to Relation \eqref{rel: Weyl odd}, we conclude that $Q(z)(\sum\limits_{n \in \BZ} f_n z^n) = 0$. 
\end{proof}
Analogous result holds for the superalgebra $U_q(\mathcal{L}\mathfrak{sl}(M,N))$ when $MN > 0$ and $M \neq N$, as the weights $\alpha_i|_{A_{M,N}}$ are linearly independent (see the proof of Lemma \ref{lem: extension}). Let  $\tilde{D}$ be the set of superalgebra automorphisms of $U_q(\mathcal{L}'\mathfrak{sl}(M,N))$ of the following forms: 
\begin{displaymath}
K_0 \mapsto t K_0,\ K_i \mapsto \varepsilon_i K_i,\ h_{i,s} \mapsto h_{i,s},\ X_{i,n}^{+} \mapsto \varepsilon_i X_{i,n}^{+},\ X_{i,n}^- \mapsto X_{i,n}^- 
\end{displaymath}
for $1 \leq i \leq M+N-1, n \in \BZ, s \in \BZ_{\neq 0}$, where $\varepsilon_i \in \{\pm 1 \}, t \in \mathbb{C} \setminus \{0\}$ with $\varepsilon_M = 1$. Note that such an automorphism always preserves $\Ud$. Let $D$ be the set of superalgebra automorphisms of $\Ud$ of the form $\pi|_{\Ud}$ with $\pi \in \tilde{D}$.
\begin{cor}
Suppose $MN > 0$ and $(M,N) \neq (1,1)$. All finite-dimensional simple $U_q(\mathcal{L}'\mathfrak{sl}(M,N))$-modules are of the form $\pi^*(\simple'(\underline{P},f,c))$ where $(\underline{P},f,c) \in \mathcal{R}_{M,N}$ and $\pi \in \tilde{D}$.
\end{cor}
\begin{definition}
Let $(\underline{P},f,c) \in \mathcal{R}_{M,N}$ and let $V$ be a $\Ud$-module. We say that  $V$ is of {\it highest weight} $(\underline{P},f,c)$ if there is an epimorphism of $\Ud$-modules: $\Verma(\underline{P},f,c) \twoheadrightarrow V$.  
\end{definition}
 One can now have a super-version of Theorem \ref{thm: Weyl modules for quantum affine algebras} (b). Let $\iota: \Ud \hookrightarrow U_q(\mathcal{L}'\mathfrak{sl}(M,N))$ be the canonical injection defined in \S \ref{sec: 3.2}.
\begin{prop}  \label{prop: highest weight representations}
Suppose $MN > 0$ and $M \neq N$. 
\begin{itemize}
\item[(a)] All finite-dimensional simple $\Ud$-modules are of the form $\pi^*\iota^* \simple'(\underline{P},f,c)$ where $(\underline{P},f,c) \in \mathcal{R}_{M,N}$ and $\pi \in D$.
\item[(b)] Let $V$ be a finite-dimensional $\Ud$-module of highest weight $(\underline{P},f,c) \in \mathcal{R}_{M,N}$.  Let $\theta: \Verma(\underline{P},f,c) \twoheadrightarrow V$ be an epimorphism of $\Ud$-modules. Then there exists $Q(z) \in \mathbb{C}[z]$ such that $(\underline{P},f,c;Q) \in \tilde{\mathcal{R}}_{M,N}$ and  $\theta$ factorizes through the canonical epimorphism $\Verma(\underline{P},f,c) \twoheadrightarrow \Weyl(\underline{P},f,c;Q)$.
\item[(c)] For $(\underline{P},f,c;Q_1),(\underline{P},f,c;Q_2) \in \tilde{\mathcal{R}}_{M,N}$, the canonical epimorphism $\Verma(\underline{P},f,c) \twoheadrightarrow \Weyl(\underline{P},f,c;Q_1)$ factorizes through $\Verma(\underline{P},f,c) \twoheadrightarrow \Weyl(\underline{P},f,c;Q_2)$ if and only if $Q_1(z)$ divides $Q_2(z)$ as polynomials.
\end{itemize}
\end{prop}
\begin{proof}
(a) and (b) come from Lemma \ref{lem: finite-dimensional modules} and Theorem \ref{thm: main}. For (c), the \lq\lq if\rq\rq part is clear from the definition of Weyl modules. Without loss of generality, assume $M>N$. Suppose that $\Weyl(\underline{P},f,c;Q_2) \twoheadrightarrow \Weyl(\underline{P},f,c;Q_1)$ and we have a surjection \begin{displaymath}
(\Weyl(\underline{P},f,c;Q_2))_{\lambda_{(\underline{P},f,c)} - \alpha_M} \twoheadrightarrow (\Weyl(\underline{P},f,c;Q_1))_{\lambda_{(\underline{P},f,c)} - \alpha_M}
\end{displaymath}
which respects clearly the actions of $h_{M-1,1}$. We conclude from Proposition \ref{prop: uniqueness of Weyl modules}.  
\end{proof}
\subsection{Integrable representations}
This section deals with generalisations of Theorem \ref{thm: Weyl modules for quantum affine algebras} (c). We shall see that, for all $\Lambda \in \mathcal{R}_{M,N}$, there exists a largest integrable module of  highest weight $\Lambda$. However, such modules turn out to be infinite-dimensional, contrary to the quantum affine algebra case.
\begin{definition}
Call a $\Ud$-module {\it integrable} if the actions of $X_{i,0}^{\pm}$ are locally nilpotent for $1 \leq i \leq M+N-1$.
\end{definition}
Note that the actions of $X_{M,0}^{\pm}$ are always nilpotent. From the representation theory of $U_q(\mathfrak{sl}_2)$ \cite[Chapter 2]{Jantzen}, we see that finite-dimensional $\Ud$-modules are always integrable. In particular, the Weyl modules and all their quotients are integrable.

\paragraph{Universal Weyl modules.} Let $(\underline{P},f,c) \in \mathcal{R}_{M,N}$. The associated universal Weyl module, denoted by $\Weyl(\underline{P},f,c)$, is the $\Ud$-module generated by $v_{(\underline{P},f,c)}$ of $\super$-degree $\even$ subject to Relations \eqref{rel: highest weight vector 1}-\eqref{rel: hwv3} and \eqref{rel: weyl even}. $\Weyl(\underline{P},f,c)$ becomes a $U_q(\mathcal{L}'\mathfrak{sl}(M,N))$-module by Relation \eqref{rel: auxillary}. Note that $\Weyl(\underline{P},f,c)$, being a quotient of $\Verma(\underline{P},f,c)$, is $P_{M,N}$-graded. Let $\mathrm{wt}(\Weyl(\underline{P},f,c)) \subseteq \lambda_{(\underline{P},f,c)} - Q_{M,N}^+$ be the set of weights. 
\begin{prop}    \label{prop: universal Weyl modules are integrable}
Suppose $MN > 0$. Fix $(\underline{P},f,c) \in \mathcal{R}_{M,N}$. Then there exists $k \in \BZ_{>0}$ such that: 
\begin{displaymath}
\mathrm{wt}(\Weyl(\underline{P},f,c)) \subseteq \{ \lambda_{(\underline{P},f,c)} - \sum_{i=1}^{M+N-1} u_i \alpha_i \in P_{M,N} \ |\ u_M - u_i > -k \ \textrm{for}\ 1 \leq i \leq M+N-1 \}.
\end{displaymath}
In particular, $\Weyl(\underline{P},f,c)$ is integrable.
\end{prop}
\begin{proof}
The idea is similar to that of the proof of Theorem \ref{thm: main}: to use  PBW generators to deduce restrictions on the set of weights. One also needs the isomorphism $\pi_{M,N}: \Ud \longrightarrow U_q(\mathcal{L}\mathfrak{sl}(N,M)) $ to change the forms of the PBW generators.  
\end{proof}
Thus, for $(\underline{P},f,c) \in \mathcal{R}_{M,N}$, the universal Weyl module  $\Weyl(\underline{P},f,c)$ is the largest integrable highest weight module of highest weight $(\underline{P},f,c)$.   Note however that $\Weyl(\underline{P},f,c)$ is by no means finite-dimensional when $MN > 0$. Indeed, for all $(\underline{P},f,c;Q)$, we have an epimorphism of $\Ud$-modules $\Weyl(\underline{P},f,c) \twoheadrightarrow \Weyl(\underline{P},f,c;Q)$. It follows from Proposition \ref{prop: uniqueness of Weyl modules} that $\dim \Weyl(\underline{P},f,c) > \deg Q$. As $\deg Q$ can be chosen arbitrarily large, $\Weyl(\underline{P},f,c)$ is infinite-dimensional.
\section{Evaluation morphisms}   \label{sec: 5}
Throughout this section, we assume $M > 1, N \geq 1$ and $M \neq N$. After \cite[Theorem 8.5.1]{Yam2}, there is another presentation of the quantum superalgebra $U_q(\mathcal{L}\mathfrak{sl}(M,N))$. From this new presentation, we get a structure of Hopf superalgebra on $\Ud$ (in the usual sense). Using evaluation morphisms between $\Ud$ and $U_q(\mathfrak{gl}(M,N))$, we construct certain finite-dimensional simple $\Ud$-modules.  
\subsection{Chevalley presentation of $\Ud$ for $M \neq N$}
Enlarge the Cartan matrix $C = (c_{i,j})_{1 \leq i,j \leq M+N-1}$ to the affine Cartan matrix $\widehat{C} = (c_{i,j})_{0 \leq i,j \leq M+N-1}$ with $c_{0,0} = 0, c_{0,i} = c_{i,0} = -\delta_{i,1} + \delta_{i,M+N-1}$ for $1 \leq i \leq M+N-1$. Set $q_0 = q$.
\begin{definition} \cite[Proposition 6.7.1]{Yam2}
$U_q'(\widehat{\mathfrak{sl}(M,N)})$ is the superalgebra generated by $E_i^{\pm},K_i^{\pm 1}$ for $0 \leq i \leq M+N-1$ with the $\super$-grading $|E_0^{\pm}| =  |E_M^{\pm}|  = \odd$ and $\even$ for other generators, and with the following relations: $0 \leq i,j \leq M+N-1$
\begin{align*}
& K_i K_i^{-1} = 1 = K_i^{-1} K_i,  \\
& K_i E_j^{\pm} K_i^{-1} = q^{\pm c_{i,j}} E_j^{\pm}, \\
& [E_i^+, E_j^-] = \delta_{i,j} \frac{K_i - K_i^{-1}}{q_i - q_i^{-1}}, \\
& [E_i^{\pm}, E_j^{\pm}] = 0 \ \ \ \mathrm{for}\ c_{i,j} = 0, \\
& [E_i^{\pm}, [E_i^{\pm}, E_j^{\pm}]_{q^{-1}}]_q = 0 \ \ \ \mathrm{for}\ c_{i,j} = \pm 1, i \neq 0,M, \\
& [[[E_{M-1}^{\pm}, E_M^{\pm}]_{q^{-1}}, E_{M+1}^{\pm}]_q, E_M^{\pm}] = 0 = [[[E_{1}^{\pm}, E_0^{\pm}]_{q^{-1}}, E_{M+N-1}^{\pm}]_q, E_0^{\pm}] \ \ \ \mathrm{when}\ M+N > 3, \\
& [E_0^{\pm},[E_2^{\pm},[E_0^{\pm},[E_2^{\pm},E_1^{\pm}]_q]] ]_{q^{-1}} = [E_2^{\pm},[E_0^{\pm},[E_2^{\pm},[E_0^{\pm},E_1^{\pm}]_q]] ]_{q^{-1}}\ \ \ \mathrm{when}\ (M,N) = (2,1).
\end{align*}
\end{definition}
Remark that $c := K_0 K_1 \cdots K_{M+N-1}$ is central in $\Uc$. We reformulate part of  \cite[Theorem 8.5.1]{Yam2}:
\begin{theorem} \label{thm: two presentations}
There exists a unique superalgebra homomorphism $\Phi: \Uc \longrightarrow \Ud$ such that:
\begin{align*}
&\Phi(K_0) = (K_1\cdots K_{M+N-1})^{-1},\  \Phi(K_i) = K_i,\ \Phi(E_i^{\pm}) = X_{i,0}^{\pm} \ \ \mathrm{for}\ 1 \leq i \leq M+N-1, \\
&\Phi(E_0^+) = (-1)^{M+N-1} q^{N-M}[\cdots [X_{1,1}^-,X_{2,0}^-]_{q_2},\cdots, X_{M+N-1,0}^-]_{q_{M+N-1}} (K_1\cdots K_{M+N-1})^{-1}, \\
&\Phi(E_0^-) = (K_1 \cdots K_{M+N-1}) [\cdots [X_{1,-1}^+,X_{2,0}^+]_{q_2},\cdots, X_{M+N-1,0}^+]_{q_{M+N-1}}.
\end{align*}
Furthermore, $\Phi$ is surjective with kernel $\ker \Phi = \Uc (c - 1)$.
\end{theorem}
Note that $\Phi(E_0^+) = (-1)^{M+N-1}q^{N-M} X_{\alpha_1+\cdots+\alpha_{M+N-1}}^-(1) \Phi(K_0)$ and $\Phi(E_0^-) = \Phi(K_0^{-1}) X_{\alpha_1+\cdots+\alpha_{M+N-1}}(-1)$ from Definition \ref{def: positive root vectors} and the proof of Theorem \ref{thm: main}. 

From the construction of $\Uc$ in \S 6.1 of \cite{Yam2}, we see that $\Uc$ is endowed with a Hopf superalgebra structure:  for $0 \leq i \leq M+N-1$
\begin{equation}  \label{equ: Hopf superalgebra}
\Delta (K_i) = K_i \stensor K_i,\ \Delta(E_i^+) = 1 \stensor E_i^+ + E_i^+ \stensor K_i^{-1},\ \Delta (E_i^-) = K_i \stensor E_i^- + E_i^- \stensor 1.
\end{equation}
 Here the coproduct formula is consistent with that of  \eqref{equ: coproduct BKK}. Under this coproduct, $\Delta (c) = c \otimes c$. Hence, $\ker \Phi = \Uc (c-1)$ becomes a Hopf ideal of $\Uc$, and $\Phi$ induces a Hopf superalgebra structure on $\Ud$.

To simplify notations, let $U := \Ud$ and $X^{\pm} := \sum\limits_{i=1}^{M+N-1} X_i^{\pm}$ where $X_i^{\pm} = \sum\limits_{n \in \BZ} \mathbb{C} X_{i,n}^{\pm} \subseteq U$. 

\begin{lem}  \label{lem: coproduct Drinfel'd zero generators}
 For $1 \leq i \leq M+N-2$, there exist $x_{i}^{\pm},y_i^{\pm}, z_i^{\pm} \in \mathbb{C}$ such that modulo $U (X^-)^2 \stensor U (X^+)^2$
\begin{eqnarray*}
 \Delta (h_{i,1}) &\equiv & h_{i,1} \stensor 1 + 1 \stensor h_{i,1} + x_i^+ X_{i-1,1}^-K_{i-1}^{-1} \stensor K_{i-1}X_{i-1,0}^+ + y_i^+  X_{i,1}^-K_{i}^{-1} \stensor K_{i}X_{i,0}^+ + z_i^+ X_{i+1,1}^-K_{i+1}^{-1} \stensor K_{i+1}X_{i+1,0}^+  \\
 \Delta (h_{i,-1}) &\equiv & h_{i,-1} \stensor 1 + 1 \stensor h_{i,-1} + x_i^- X_{i-1,0}^-K_{i-1}^{-1} \stensor K_{i-1} X_{i-1,-1}^+ + y_i^-X_{i,0}^-K_{i}^{-1} \stensor K_{i} X_{i,-1}^+ \\
 &\ & + z_i^- X_{i+1,0}^-K_{i+1}^{-1} \stensor K_{i+1} X_{i+1,-1}^+.
\end{eqnarray*}
Moreover, $z_i^{\pm} = \pm (q_i - q_i^{-1})$. We understand $X_{0,n}^{\pm} = 0, K_0^{\pm 1} = \Phi(K_0^{\pm 1})$.
\end{lem}
The idea is to express $h_{i,\pm 1}$ as products in the $\Phi(E_i^{\pm}), \Phi(K_i^{\pm 1})$ with $0 \leq i \leq M+N-1$ and then use the coproduct Formulae \eqref{equ: Hopf superalgebra}. Details are left to  Appendix B.

Now we can deduce a similar result of \cite[Proposition 4.4]{CP1}
\begin{prop}  \label{prop: coproduct formulas Drinfel'd}
Let $1 \leq j \leq M+N-1$ and $n \in \BZ_{>0}$. In the vector superspace $U \stensor U$, we have 
\begin{itemize}
\item[(a)] $\Delta(X_{j,0}^+) = 1 \stensor X_{j,0}^+ + X_{j,0}^+ \stensor K_j^{-1}$, and modulo $UX^- \stensor U(X^+)^2$,
\begin{align*}
& \Delta(X_{j,n}^+) \equiv 1 \stensor X_{j,n}^+ + X_{j,n}^+ \stensor K_j^{-1} + \sum_{s=1}^n K_j^{-1}\phi_{j,s}^+ \stensor X_{j,n-s}^+, \\
& \Delta(X_{j,-n}^+) \equiv K_j^{-2} \stensor X_{j,-n}^+ + X_{j,-n}^+ \stensor K_j^{-1} + \sum_{s=1}^{n-1} K_j^{-1} \phi_{j,-s}^- \stensor X_{j,-n+s}^+,
\end{align*}
\item[(b)] $\Delta(X_{j,0}^-) = K_j \stensor X_{j,0}^- + X_{j,0}^- \stensor 1$, and modulo $U(X^-)^2 \stensor U X^+$, 
\begin{align*}
& \Delta(X_{j,n}^-) \equiv K_j \stensor X_{j,n}^- + X_{j,n}^- \stensor K_j^2 + \sum_{s=1}^{n-1} X_{j,s}^- \stensor K_j \phi_{j,n-s}^+, \\
& \Delta(X_{j,-n}^-) \equiv K_j \stensor X_{j,-n}^- + X_{j,-n}^- \stensor 1 + \sum_{s=1}^n X_{j,-n+s}^- \stensor K_j \phi_{j,-s}^-,
\end{align*}
\item[(c)] $\Delta(\phi_{j,0}^{\pm}) = \phi_{j,0}^{\pm} \stensor \phi_{j,0}^{\pm}$, and modulo $U X^- \stensor UX^+ + UX^+ \stensor UX^-$, 
\begin{align}  
\Delta (\phi_{j,\pm n}^{\pm}) \equiv \sum_{s=0}^n \phi_{j, \pm s}^{\pm} \stensor \phi_{j,\pm (n-s)}^{\pm}.
\end{align}
\end{itemize}
\end{prop} 
\begin{proof}
Note that Relation \eqref{rel: triangular1} implies $[h_{i, s}, X_{i+1,n}^{\pm}] = \mp (-1)^{p(\alpha_i)} X_{i+1,n+s}^{\pm}$ for $s = \pm 1$. One can prove the first four formulae by induction on $n$, using repeatedly the above lemma. Take the first formula as an example. Assume that $M+2 \leq j \leq M+N-1$ (similar in other cases). Suppose we know the formula for $\Delta(X_{j,n}^+)$. Then $\Delta(X_{j,n+1}^+) = -[\Delta(h_{j-1},1), \Delta(X_{j,n}^+)]$. Remark that for $i \neq j$ 
\begin{displaymath}
[X_{j,n}^+ \stensor K_j^{-1}, X_{i,1}^- K_i^{-1} \stensor K_i X_{i,0}^+] = [X_{j,n}^+, X_{i,1}^-] K_i^{-1} \stensor K_iK_j^{-1} X_{i,0}^+ = 0.
\end{displaymath}  
Hence, modulo $UX^- \stensor U(X^+)^2$, (recall that $q_{j-1} = q_j = q^{-1}$)
\begin{eqnarray*}
\Delta(X_{j,n+1}^+)  &\equiv & [1 \stensor X_{j,n}^+ + X_{j,n}^+ \stensor K_j^{-1} + \sum_{s=1}^n K_j^{-1}\phi_{j,s}^+ \stensor X_{j,n-s}^+,\\
&\ & 1 \stensor h_{j-1,1} + h_{j-1,1} \stensor 1 + (q^{-1}-q) X_{j,1}^-K_j^{-1} \stensor K_j X_{j,0}^+ ] \\
&\equiv & 1 \stensor X_{j,n+1}^+ + X_{j,n+1}^+ \stensor K_j^{-1} + \sum_{s=1}^n K_j^{-1} \phi_{j,s}^+ \stensor X_{j,n+1-s}^+ + (q^{-1}-q) [X_{j,n}^+, X_{j,1}^-] K_j^{-1} \stensor X_{j,0}^+ \\
&\equiv & 1 \stensor X_{j,n+1}^+ + X_{j,n+1}^+ \stensor K_j^{-1} + \sum_{s=1}^{n+1} K_j^{-1} \phi_{j,s}^+ \stensor X_{j,n+1-s}^+,
\end{eqnarray*} 
as desired. For the last formula, we use  $\frac{ \phi_{j,n}^+}{q_j - q_j^{-1}} = [X_{j,n}^+,X_{j,0}^-], \frac{\phi_{j,-n}^-}{q_j^{-1}-q_j} = [X_{j,0}^+, X_{j,-n}^-]$ and conclude.
\end{proof}
As in the case of quantum affine algebras \cite[Proposition 4.3]{CP2}, we get
\begin{cor}  \label{cor: tensor product highest weight}
For $(\underline{P},f,c), (\underline{Q},g,d) \in \mathcal{R}_{M,N}$, there exists a morphism of $\Ud$-modules 
\begin{displaymath}
\Verma(\underline{PQ}, f*g, cd) \longrightarrow \Verma(\underline{P},f,c) \stensor \Verma(\underline{Q},g,d)
\end{displaymath}
such that $v_{(\underline{PQ},f*g,cd)} \mapsto v_{(\underline{P},f,c)} \stensor v_{(\underline{Q},g,d)}$. Here $f * g = \frac{f^+ g^+ - f^- g^-}{q-q^{-1}}$ with $f^{\pm} = c^{\pm 1} \pm (q-q^{-1}) \sum\limits_{s=1}^{\infty} f_{\pm s} z^{\pm s}$ and  $(\underline{PQ})_i = P_i Q_i$ for $1 \leq i \leq M+N-1, i \neq M$.
\end{cor} 
The corollary above endows $\mathcal{R}_{M,N}$ with a structure of monoid (valid even if $M=N$): 
\begin{displaymath}
*: \mathcal{R}_{M,N} \times \mathcal{R}_{M,N} \longrightarrow \mathcal{R}_{M,N}, \ \ (\underline{P},f,c) * (\underline{Q},g,d) = (\underline{PQ}, f * g, cd)
\end{displaymath}
where the neutral element is $(\underline{1}, 0, 1)$. From the commutativity of $(\mathcal{R}_{M,N},*)$ we  also see that if the tensor product $S_1 \stensor S_2$ of two finite-dimensional simple $\Ud$-modules remains simple, then so does $S_2 \stensor S_1$ and $S_1 \stensor S_2 \cong S_2 \stensor S_1$ as $\Ud$-modules. 
\subsection{Evaluation morphisms}  \label{sec: 5.2}
In this section, we construct some simple modules of $\Ud$ via evaluation morphisms.
  
As in the case of quantum affine algebras $U_q(\widehat{\mathfrak{sl}_n})$, we also have evaluation morphisms:  
\begin{prop}  \label{prop: evaluation morphisms Chevalley}
There exists a morphism of superalgebras $\mathrm{ev}: U_q'(\widehat{\mathfrak{sl}(M,N)}) \longrightarrow U_q(\mathfrak{gl}(M,N))$ such that $\mathrm{ev}(K_j) = t_j^{l_j} t_{j+1}^{-l_{j+1}},\mathrm{ev}( E_j^{\pm}) =  e_j^{\pm}$ for $1 \leq j \leq M+N-1$, $\mathrm{ev}(K_0) = t_1^{-1} t_{M+N}^{-1}$ and
\begin{align*}
& \mathrm{ev}(E_0^+) = - (-q)^{N-M} [\cdots [e_1^-, e_2^-]_q,\cdots, e_{M+N-1}^-]_{q_{M+N-1}} t_1 t_{M+N}^{-1}, \\
& \mathrm{ev} (E_0^-) = t_1^{-1} t_{M+N} [\cdots [e_1^+, e_2^+]_q,\cdots, e_{M+N-1}^+]_{q_{M+N-1}}.
\end{align*}
\end{prop}
\begin{remark}
$\mathrm{ev}(E_0^{\pm})$ appeared implicitly in \cite[Lemma 4]{Zh1}. Note however that Zhang did not verify the degree 5 relations in the case $(M,N) = (2,1)$. A lengthy calculation shows that this is true, and that $\mathrm{ev}$ is always well-defined. The modified element $K = t_1 t_{M+N}^{-1}$ is needed to ensure that $K e_i^{\pm} K^{-1} = e_i^{\pm}$ for $2 \leq i \leq M+N-2$ and $K e_i^{\pm} K^{-1} = q^{\pm 1} e_i^{\pm}$ when $i = 1, M+N-1$. If $0 < |M-N| \leq 2$, then $K$ can be chosen so that $K \in U_q(\mathfrak{sl}(M,N))$.
\end{remark}
It is clear that $\mathrm{ev}(K_0\cdots K_{M+N-1}) = 1$. This implies that $\mathrm{ev}: U_q'(\widehat{\mathfrak{sl}(M,N)}) \longrightarrow U_q(\mathfrak{gl}(M,N))$ factorizes through $\Phi: U_q(\widehat{\mathfrak{sl}(M,N)}) \twoheadrightarrow \Ud$. Let $\mathrm{ev}': \Ud \longrightarrow U_q(\mathfrak{gl}(M,N))$ be the superalgebra morphism thus obtained.
\begin{lem} \label{lem: evaluation on Drinfel'd generators}
For the superalgebra morphism $\mathrm{ev}': \Ud \longrightarrow U_q(\mathfrak{gl}(M,N))$, we have
\begin{itemize}
\item[(a)] $\mathrm{ev}'(X_{1,n}^+) = t_1^{2n} e_1^+,\ \mathrm{ev}'(X_{1,n}^-) = e_1^- t_1^{2n}$; 
\item[(b)] for $2 \leq j \leq M+N-1$, modulo $\sum\limits_{i=1}^{M+N-1} U_q(\mathfrak{gl}(M,N)) e_i^+$, $\mathrm{ev}'(X_{j,n}^+)\equiv 0$ and 
\begin{displaymath}
\mathrm{ev}'(X_{j,n}^-) \equiv (\prod_{s=2}^j q_s^{-1})^n e_j^- t_j^{2 l_j n}. 
\end{displaymath}
\end{itemize} 
\end{lem}
\begin{proof}
According to the Formulae \eqref{equ: h_{i,1}}-\eqref{equ: h_{i,-1}} in Appendix B, we get 
\begin{displaymath}
\mathrm{ev}'(h_{1,\pm 1}) = (1-q^{\pm 2}) e_1^- (t_1t_2)^{\pm 1} e_1^+ \pm \frac{t_1^{\pm 2} - t_2^{\pm 2}}{q-q^{-1}}. 
\end{displaymath}
The rest is clear in view of Relations \eqref{rel: triangular1}-\eqref{rel: triangular2}. 
\end{proof}
Remark that from  Definition \ref{def: Drinfeld presentation} of quantum affine superalgebras $\Ud$ admit naturally a $\BZ$-grading provided by the second index (the first index gives $Q_{M,N}$-grading). From this $\BZ$-grading, we construct for each $a \in \mathbb{C} \setminus \{0\}$, a superalgebra automorphism $\Phi_a: \Ud \longrightarrow \Ud$ defined by:  
\begin{equation*}
X_{i,n}^{\pm} \mapsto a^n X_{i,n}^{\pm},\ h_{i,s} \mapsto a^s h_{i,s}, K_i^{\pm 1} \mapsto K_i^{\pm 1} 
\end{equation*}
for $1 \leq i \leq M+N-1, n \in \BZ, s \in \BZ_{\neq 0}$. Furthermore, define the {\it evaluation morphism} $\mathrm{ev}_a$ by 
\begin{equation}
\mathrm{ev}_a := \mathrm{ev}' \circ \Phi_a : \Ud  \longrightarrow U_q(\mathfrak{gl}(M,N)). 
\end{equation} 

Given a representation $(\rho, V)$ of $U_q(\mathfrak{gl}(M,N))$, one can construct a family $(\rho \circ \mathrm{ev}_a, \mathrm{ev}_a^* V : a \in \mathbb{C} \setminus \{0\})$ of representations of $\Ud$. In particular, one obtains some finite-dimensional simple modules in this way. Take $a \in \mathbb{C}\setminus \{0\}$. Lemma \ref{lem: evaluation on Drinfel'd generators} together with Theorem \ref{thm: Kac modules for quantum superalgebra} leads to the following proposition: 
 
\begin{prop}   \label{prop: evaluation modules}
Let $\Lambda \in \mathcal{S}_{M,N}$ and $a \in \mathbb{C}\setminus \{0\}$. Then $\mathrm{ev}_a^* L(\Lambda) \cong \iota^* \simple'(\underline{P},f,c)$ as simple $\Ud$-modules, where
\begin{align}
& P_i(z) = \prod_{s=1}^{\Delta_i} (1 - q_i^{1-2 s} a \theta_i q_i^{2 \Lambda_i} z) \ \ \ \mathrm{for}\ 1 \leq i \leq M+N-1, i \neq M, \\
& c = q^{\Lambda_M + \Lambda_{M+1}},\ \ f(z) = \frac{c - c^{-1}}{q-q^{-1}}\sum_{n \in \BZ} (a \theta_M q^{2 \Lambda_M} z)^n,  \\
& \theta_1 = 1, \ \ \theta_s = \prod_{i=2}^s q_i^{-1} \ \ \ \mathrm{for}\ 2 \leq s \leq M+N-1.
\end{align}
\end{prop}
\begin{example}
 When $\Lambda_i  = \delta_{i,1}$, we get the fundamental representation of $U_q(\mathfrak{gl}(M,N))$ on the vector superspace $V = V_{\even} \oplus V_{\odd}$ with $\dim V_{\even} = M, \dim V_{\odd} = N$ (see for example \cite[\S 3.2]{BKK} for the actions of Chevalley generators). We obtain therefore  finite-dimensional  simple $\Ud$-modules corresponding to $(\underline{P}, 0, 1) \in \mathcal{R}_{M,N}$ where $P_i(z) = 1 - q \delta_{i,1} a z$ with $a \in \mathbb{C}\setminus \{0\}$.
\end{example}
\section{Further discussions}  \label{sec: final section}
\paragraph{Representations of quantum superalgebras.} 
As we have seen in \S \ref{sec: 5.2}, for $a \in \mathbb{C}\setminus \{0\}$ there exists a superalgebra homomorphism (assuming $M \neq N$)
\[ \mathrm{ev}_a: \Ud \longrightarrow U_q(\mathfrak{gl}(M,N)).  \]
One can pull back representations of $U_q(\mathfrak{gl}(M,N))$ to get those of $\Ud$. 

In 2000, Benkart-Kang-Kashiwara \cite{BKK} proposed  a subcategory $\category$ of finite-dimensional representations of $U_q(\mathfrak{gl}(M,N))$ over the field $\mathbb{Q}(q)$  to study the crystal bases. Using the notations in \S \ref{sec: 2.2}, one has also the finite-dimensional simple modules $L(\Lambda)$ for $\Lambda \in \mathbb{Z}^{M+N}$ verifying $\Lambda_i - \Lambda_{i+1}  \in \BZ_{\geq 0}$ for $1 \leq i \leq M+N-1, i \neq M$. Simple modules in $\category$ are of the form $L(\Lambda) \stensor D$ where $D \in \category$ is one-dimensional and   (Proposition 3.4)
\begin{itemize}
\item[(i)] $\Lambda_M + \Lambda_{M+N} \geq 0$;
\item[(ii)] if $\Lambda_{M+k} > \Lambda_{M+k+1}$ for some $1 \leq k \leq N-1$, then $\Lambda_M + \Lambda_{M+k+1} \geq k$.
\end{itemize}
The simple module $L(\Lambda)$ in $\category$ always admits a polarizable crystal base (Theorem 5.1), with the associate crystal  $B(\Lambda)$  being realised as the set of all semi-standard tableaux of shape $Y_{\Lambda}$ (Definition 4.1). Here $Y_{\Lambda}$ is a Young diagram constructed from $\Lambda$. In this way one gets a combinatorial description of the character and the dimension for the simple module $L(\Lambda)$. 

Explicit constructions of representations of the quantum superalgebra $U_q(\mathfrak{gl}(M,N))$ are also of importance to us. In \cite{KV}, Ky-Van constructed finite-dimensional representations of $U_q(\mathfrak{gl}(2,1))$ and studied their basis with respect to its even subalgebra $U_q(\mathfrak{gl}(2) \oplus \mathfrak{gl}(1))$. Early in \cite{Ky,KS}, certain finite-dimensional representations of $U_q(\mathfrak{gl}(2,2))$ were constructed  together with their decomposition into simple modules with respect to the subalgebra $U_q(\mathfrak{gl}(2) \oplus \mathfrak{gl}(2))$.  In 1991, Palev-Tolstoy \cite{PT} deformed the finite-dimensional Kac/simple modules of $U(\mathfrak{gl}(N,1))$ to the corresponding modules of $U_q(\mathfrak{gl}(N,1))$ and wrote down the actions of the algebra generators in terms of Gel'fand-Zetlin basis. Later, Palev-Stoilova-Van der Jeugt \cite{PSV} generalised the above constructions to the quantum superalgebra $U_q(\mathfrak{gl}(M,N))$. However, their methods applied only to a certain class of irreducible representations, the so-called {\it essentially typical representations}.   Recently,  the coherent state method was applied to construct representations of superalgebras and quantum superalgebras. In \cite{KKNV}, Kien-Ky-Nam-Van used the vector coherent state method to construct representations of $U_q(\mathfrak{gl}(2,1))$. However, for quantum superalgebras $U_q(\mathfrak{gl}(M,N))$ of higher ranks, the analogous constructions are still not explicit.

\paragraph{Relations with Yangians.} In the paper \cite{Zh3}, Zhang developed a highest weight theory for finite-dimensional representations of the super Yangian $Y(\mathfrak{gl}(M,N))$, and obtained a classification of finite-dimensional simple modules (Theorems 3, 4). Here the set $\mathcal{T}_{M,N}$ of highest weights consists of $\Lambda = (\Lambda_i: 1 \leq i \leq M+N)$ such that 
\begin{align*}
& \Lambda_i(z) \in (-1)^{[i]} + z^{-1} \mathbb{C}[z^{-1}]\ \ \ \ \mathrm{for}\ 1 \leq i \leq M+N,\\
& \frac{\Lambda_i(z)}{\Lambda_{i+1}(z)} = \frac{P_i(z + (-1)^{[i]})}{P_i(z)} \ \ \ \ \mathrm{for}\ 1 \leq i \leq M+N-1, i\neq M,\\
& \frac{\Lambda_M(z)}{\Lambda_{M+1}(z)} = - \frac{\tilde{Q}(z)}{Q(z)}
\end{align*} 
where $Q,\tilde{Q} \in 1 + z^{-1}\mathbb{C}[z^{-1}]$ are co-prime  polynomials, $P_i(z) \in \mathbb{C}[z]$ are polynomials of leading coefficient $1$, and $[i] = \begin{cases}
\even & i \leq M, \\
\odd & i > M.
\end{cases}$\ For any $\Lambda \in \mathcal{T}_{M,N}$, there exists a finite-dimensional simple $Y(\mathfrak{gl}(M,N))$-module $\simple(\Lambda)$. Up to modification by one-dimensional modules, all simple modules are of the form $\simple(\Lambda)$.  Zhang also constructed explicitly the simple modules $\simple(\Lambda)$ for $\Lambda \in \mathcal{T}_{M,N}^0$:
\begin{align*}
 \Lambda_i(z) \in (-1)^{[i]} + \mathbb{C}z^{-1} \ \ \ \ \mathrm{for}\ 1 \leq i \leq M+N.
\end{align*} 
Other simple modules $\simple(\Lambda)$ can always be realised as subquotients of $\bigotimes\limits_{s=1}^n \phi_s^* \simple(\Lambda_{(s)})$ where $\Lambda_{(s)} \in \mathcal{T}_{M,N}^0$ and the $\phi_s$ are some superalgebra automorphisms of $Y(\mathfrak{gl}(M,N))$.

In the case of quantum affine superalgebra $\Ud$, the set of highest weights is $\mathcal{R}_{M,N}$, which is a commutative monoid. It is not easy to see, in view of Definition \ref{def: highest weight} and Corollary \ref{cor: tensor product highest weight}, that as monoids
\begin{displaymath}
\mathcal{R}_{M,N} \cong (1 + z\mathbb{C}[z])^{M+N-2} \times \mathcal{R}_{1,1}. 
\end{displaymath}
In the following, we investigate the monoid structure of $\mathcal{R}_{1,1}$. As we shall see, $\mathcal{R}_{1,1}$ is almost $\mathcal{T}_{1,1}$. Thus, informally speaking, the two monoids $\mathcal{R}_{M,N}$ and $\mathcal{T}_{M,N}$ are almost equivalent, and  the finite-dimensional representation theories for $\Ud$ and for  $Y(\mathfrak{gl}(M,N))$ should have some hidden similarities.

Recall that $\mathcal{R}_{1,1}$ is the set of couples $(f,c)$ where $f$ is a formal series with torsion, with $\frac{c-c^{-1}}{q-q^{-1}}$ being the constant term. Let $\iota_{\pm}: \mathbb{C}[[z^{\pm 1}]] \longrightarrow \mathbb{C}[[z,z^{-1}]]$ be the canonical inclusions of formal series.
\begin{prop}
Let $\mathcal{T}$ be the set of triples $(c,Q,P)$ where: 
\begin{itemize}
\item[(a)] $c \in \mathbb{C} \setminus \{0\}$, $P(z),Q(z) \in 1 + z \mathbb{C}[z]$;
\item[(b)] $P(z), Q(z)$ are co-prime as polynomials, moreover, $\lim_{z\rightarrow \infty} \frac{Q(z)}{P(z)} = c^{-2}$.
\end{itemize}
Equip $\mathcal{T}$ with a structure of monoid by: $(c,Q,P) * (c',Q',P') = (cc',QQ',PP'), 1 = (1,1,1)$. Then
\begin{align}
\iota_{+,-}: \mathcal{T} \longrightarrow \mathcal{R}_{1,1}, \ \ (c,Q,P) \mapsto (\frac{1}{q-q^{-1}}\iota_+(c \frac{Q(z)}{P(z)}) - \frac{1}{q-q^{-1}} \iota_-(c \frac{Q(z)}{P(z)}), c)  \label{equ: series to polynomials}
\end{align}
defines an isomorphism of monoids.
\end{prop}
\begin{proof}
First, $\iota_{+,-}$ is well-defined, as the formal series in the RHS of \eqref{equ: series to polynomials} is killed by $P(z)$. Moreover, $\iota_{+,-}$ respects the monoid structures in view of Corollary \ref{cor: tensor product highest weight}.

Next, fix $(f,c) \in \mathcal{R}_{1,1}$. Let $P(z) \in 1 + z\mathbb{C}[z]$ be the Drinfel'd polynomial of smallest degree killing $f(z)$. Express
\begin{displaymath}
f(z) = \sum_{n} f_n z^n,\ \ P(z) = 1 + a_1 z + \cdots + a_d z^d \ \ \ \ \mathrm{with}\ a_d \neq 0. 
\end{displaymath}
Then, for all $n \in \BZ$,
\begin{displaymath}
f_{n+d} + f_{n+d-1} a_1 + \cdots + f_{n+1} a_{d-1} + f_n a_d = 0.
\end{displaymath}
Consider the formal power series 
\begin{displaymath}
Q(z) := c^{-1} P(z) (c + (q-q^{-1})\sum_{n>0}f_n z^n) = \sum_{s\geq 0} a_s' z^s \in \mathbb{C}[[z]].
\end{displaymath}
It is clear that $a_0' = 1$ and that
\begin{eqnarray*}
a_d' &=& a_d + c^{-1} (q-q^{-1}) (f_d + f_{d-1} a_1 + \cdots + f_1 a_{d-1}) \\
&=& a_d + c^{-1} (q-q^{-1}) (-f_0 a_d) = a_d (1 - c^{-1}(c-c^{-1})) = a_d c^{-2}, \\
a_s' &=& c^{-1}(q-q^{-1}) (f_s + f_{s-1}a_1 + \cdots + f_{s-d} a_d) = 0 \ \ \ \mathrm{for}\ s >  d.
\end{eqnarray*}
This says that $Q(z) \in 1 + z\mathbb{C}[z]$ is of degree $d$ and that $\lim_{z\rightarrow \infty} \frac{Q(z)}{P(z)} = c^{-2}$. We remark that $f(z)$ is completely determined by $f_0,f_1,\cdots,f_{d-1}$, which in turn are determined by $Q(z)$. This forces
\begin{displaymath}
(q-q^{-1})f(z) = \iota_+(c \frac{Q(z)}{P(z)}) - \iota_-(c \frac{Q(z)}{P(z)}).
\end{displaymath}
If $P(z) = P_1(z)P_2(z)$ and $Q(z) = P_1(z) Q_2(z)$, then $f(z)$ should be killed by $P_2(z)$. Hence, $\deg P_1(z) = 0$ from the definition of $P(z)$. This says that $P(z),Q(z)$ are co-prime. In other words,
\begin{displaymath}
(f(z), c) = \iota_{+,-}(c, Q(z), P(z)) \ \ \ \mathrm{with}\ (c,Q(z),P(z)) \in \mathcal{T}.
\end{displaymath}

Finally, $\iota_{+,-}$ is injective as $P(z),Q(z)$ are uniquely determined by $(f,c)$.
\end{proof}
Through the isomorphism $\iota_{+,-}$, Relation \eqref{rel: hwv3} is equivalent to the following 
\begin{equation}   \label{equ: highest weight vector odd}
\sum_{n \in \BZ} \phi_{M,n}^{\pm}z^n v_{(\underline{P},f,c)} = c \frac{Q(z)}{P(z)} v_{(\underline{P},f,c)} \in \mathbb{C}v_{(\underline{P},f,c)} [[z^{\pm 1}]]\ \ \ \mathrm{with}\ \iota_{+,-}(c,Q,P) = (f,c).
\end{equation}.

In the paper \cite{GT}, Gautam-Toledano Laredo constructed an explicit algebra homomorphism  from the quantum loop algebra $U_{\hbar}(\mathcal{L}\mathfrak{g})$ to the completion of the Yangian $Y_{\hbar}(\mathfrak{g})$ with respect to some grading, where $\mathfrak{g}$ is a finite-dimensional simple Lie algebra. Also, they are able to construct a functor from a subcategory of finite-dimensional representations of $Y_{\hbar}(\mathfrak{g})$ to a subcategory of finite-dimensional representations of $U_{\hbar}(\mathcal{L}\mathfrak{g})$. It is hopeful to have some generalisations for quantum affine superalgebras and super Yangians. 
\paragraph{Other quantum affine superalgebras.} Our approach to the theory of Weyl/simple modules of $\Ud$ is quite algebraic, without using evaluation morphisms and coproduct, and is less dependent on the actions of Weyl groups. In general, for a quantum affine superalgebra, if we know its Drinfel'd realization and its PBW generators in terms of Drinfel'd currents, it will be quite hopeful that we arrive at a good theory of finite-dimensional highest weight modules (Weyl/simple modules).

We point out that in the paper \cite{AYY}, Azam-Yamane-Yousofzadeh classified finite-dimensional  simple representations of generalized quantum groups using the notion of a {\it  Weyl groupoid}. We remark that Weyl groupoids appear naturally in the study of quantum/classical Kac-Moody superalgebras due to the existence of non-isomorphic Dynkin diagrams for the same Kac-Moody superalgebra. Roughly speaking, a Weyl groupoid is generated by even reflections similar to the case of Kac-Moody algebras, together with odd reflections in order to keep track of different Dynkin diagrams.   Note that early in \cite{Yam2}, Yamane generalised Beck's argument \cite{Beck1} by  using the Weyl groupoids instead of the Weyl groups to get the two presentations of $\Ud$. Later in \cite{groupoid}, similar arguments of Weyl groupoids led to Drinfel'd realizations of the quantum affine superalgebras $U_q(D^{(1)}(2,1;x))$. Also, in the paper \cite{Ser}, Serganova studied highest weight representations of certain classes of Kac-Moody superalgebras with the help of Weyl groupoids.  We believe that Weyl groupoids should shed light on the structures of both quantum affine superalgebras themselves and their representations.

\paragraph{Affine Lie superalgebras.} Consider the affine Lie superalgebra $\mathcal{L}\mathfrak{sl}(M,N)$ with $M \neq N$. As we have clearly the triangular decomposition and the PBW basis, we obtain a highest ($l$)-weight representation theory for $U(\mathcal{L}\mathfrak{sl}(M,N))$ similar to that of Chari \cite{Chari}. Here, the set $\mathcal{W}_{M,N}$ of highest weights are couples $(\underline{P}, f)$ where
\begin{itemize}
\item[(a)] $\underline{P} \in (1 + z \BC [z])^{M+N-1}$ (corresponding to even simple roots); 
\item[(b)] $f \in \BC[[z,z^{-1}]]$ such that $Q f = 0$ for some $Q \in 1 + z \BC[z]$ (corresponding to the odd simple root).
\end{itemize} 
Finite-dimensional simple $\mathcal{L}\mathfrak{sl}(M,N)$-modules are parametrized by their highest weight. Furthermore, for $(\underline{P},f) \in \mathcal{W}_{M,N}$ and  $Q \in 1 + z \BC[z]$ such that $Qf = 0$, we have also the Weyl module $\Weyl(\underline{P},f;Q)$ defined by generators and relations. We remark the recent work \cite{Rao} of S. Eswara Rao on finite-dimensional modules over multi-loop Lie superalgebras associated to $\mathfrak{sl}(M,N)$.  In that paper, a construction of finite-dimensional highest weight modules analogous to that of Kac induction was proposed. In this way, the character formula for these modules is easily deduced once we know the character formulae \cite{CL} for Weyl modules over $\mathcal{L}\mathfrak{sl}_n$. It is an interesting problem to compare Rao's highest weight modules with our Weyl modules, as they both enjoy universal properties. 

\appendix 
\section{Oscillation relations and triangular decomposition}  
\renewcommand\theequation{B.\oldstylenums{\arabic{equation}}}
In this appendix we finish the proof of Lemma \ref{lem: oscillation relations vs triangular decomposition}: the oscillation relations of degree 4 respect the Drinfel'd type triangular decomposition for $U_q(\mathcal{L}\mathfrak{sl}(2,2))$. As indicated in the proof of the main result Theorem \ref{thm: main}, the triangular decomposition is  needed to deduce the non-triviality of the Weyl modules. In the following, we carry out the related calculations.

Recall that $V$ is the superalgebra with generators $X_{i,n}^{\pm}, h_{i,s}, K_i^{\pm}$ and Relations \eqref{rel: Cartan}-\eqref{rel: Drinfeld2}. Here, Relations \eqref{rel: Cartan}-\eqref{rel: triangular2} ensure the triangular decomposition for $V$. For $a,b,c,d \in \mathbb{Z}$, let 
\begin{eqnarray*}
\tilde{R}(a,b,c,d) &=& R(a,b,c,d) + R(a,d,c,b) \ \mathrm{with}  \\
R(a,b,c,d) &:= & X_{1,a}^+X_{2,b}^+X_{3,c}^+X_{2,d}^+ + X_{3,c}^+X_{2,b}^+X_{1,a}^+X_{2,d}^+ + X_{2,b}^+X_{1,a}^+X_{2,d}^+X_{3,c}^+ + X_{2,b}^+X_{3,c}^+X_{2,d}^+X_{1,a}^+ \\
&\ & - (q + q^{-1}) X_{2,b}^+X_{1,a}^+X_{3,c}^+X_{2,d}^+  \in V.
\end{eqnarray*}
Let $O^+ = \sum\limits_{a,b,c,d \in \BZ} \mathbb{C} \tilde{R}(a,b,c,d)$. Our aim is to show that 
\begin{displaymath} 
[O^+, X_{i,0}^-] = 0 \ \ \ \ \mathrm{for}\ i = 1,2.  
\end{displaymath}
Using  Relation \eqref{rel: triangular2}, we see that
\begin{lem}
If $[R(a,b,c,b), X_{1,0}^-] = 0$ for all $a,b,c \in \BZ$, then $[O^+, X_{1,0}^-] = 0$.
\end{lem}
\paragraph{Case  $i=1$.}
Introduce the formal series $R_{b,c}(z) := \sum\limits_{a \in \mathbb{Z}}[R(a,b,c,b)z^a, X_{1,0}^-]  \in V[[z,z^{-1}]]$. Using the relations $[X_{i,a}^+, X_{1,0}^-] = \delta_{i,1} \frac{\phi_{1,a}^+ - \phi_{1,a}^-}{q-q^{-1}}$
one gets $R_{b,c}(z) = \frac{1}{q-q^{-1}} (R_{b,c}^+(z) - R_{b,c}^-(z)) $ where
\begin{eqnarray*}
R_{b,c}^{\pm}(z) &=& \phi_1^{\pm}(z)X_{2,b}^+X_{3,c}^+X_{2,b}^+ + X_{3,c}^+X_{2,b}^+\phi_1^{\pm}(z)X_{2,b}^+ + X_{2,b}^+\phi_1^{\pm}(z)X_{2,b}^+X_{3,c}^+ + X_{2,b}^+X_{3,c}^+X_{2,b}^+ \phi_1^{\pm}(z) \\
&\ & - (q+q^{-1}) X_{2,b}^+\phi_{1}^{\pm}(z) X_{3,c}^+X_{2,b}^+ \\
\phi_1^{\pm}(z) &=& \sum_{a \geq 0} \phi_{1,\pm a}^{\pm} z^{\pm a} \in {V}[[z^{\pm 1}]].
\end{eqnarray*}
We shall prove $R^+_{b,c}(z) = 0$ ( $R^-_{b,c}(z) = 0$ being analogous). Note that we have the following relations
\begin{eqnarray*}
&& \phi_1^+(z) = K_1 \exp((q-q^{-1})\sum_{s \geq 1}h_{1,s} z^s ) =: K_1 h_1(z) \\
&& K_1 h_1(z) X_{3,c}^+ = X_{3,c}^+ K_1h_1(z) \\
&& K_1 h_1(z) X_{2,b}^+ = q^{-1}(X_{2,b}^+ + \sum_{s \geq 1} a_s X_{2,b+s}^+ z^s ) K_1 h_1(z)\ \ \ \ \mathrm{with}\ a_s = q^{-s} - q^{-s+2},
\end{eqnarray*} 
which imply that $R_{b,c}^+(z) = \lambda_{b,c}(z) K_1h_1(z)$ with
\begin{eqnarray*}
\lambda_{b,c}(z) &=& q^{-2}(X_{2,b}^+ + \sum_{s \geq 1}a_s X_{2,b+s}^+ z^s) X_{3,c}^+ (X_{2,b}^+ + \sum_{s \geq 1}a_s X_{2,b+s}^+ z^s) \\
&\ & + q^{-1} X_{3,c}^+X_{2,b}^+ (X_{2,b}^+ + \sum_{s \geq 1}a_s X_{2,b+s}^+ z^s) + q^{-1}X_{2,b}^+(X_{2,b}^+ + \sum_{s \geq 1}a_s X_{2,b+s}^+ z^s) X_{3,c}^+\\
&\ & + X_{2,b}^+X_{3,c}^+X_{2,b}^+ - (q+q^{-1}) q^{-1} X_{2,b}^+ X_{3,c}^+ (X_{2,b}^+ + \sum_{s \geq 1}a_s X_{2,b+s}^+ z^s) \\
&=: & \sum_{n \geq 0} \lambda(n,b,c) z^n \in {V}[[z]].
\end{eqnarray*}
It is clear that $\lambda(0,b,c) = 0$. To deduce that $\lambda(n,b,c) = 0$ for all $n \geq 1$, we do the triangular decomposition for $\lambda(n,b,c)$ by using the Drinfel'd relations of degree 2.
\begin{eqnarray*}
\lambda(n,b,c) &=& q^{-2} X_{2,b}^+X_{3,c}^+ a_n X_{2,b+n}^+ + q^{-2} a_n X_{2,b+n}^+ X_{3,c}^+X_{2,b}^+ + q^{-2} \sum_{s+t = n, s,t \geq 1} a_s a_t X_{2,b+s}^+X_{3,c}^+X_{2,b+t}^+ \\
&\ & + q^{-1} X_{3,c}^+X_{2,b}^+ a_n X_{2,b+n}^+ + q^{-1} X_{2,b}^+a_nX_{2,b+n}^+X_{3,c}^+ - (1+q^{-2}) X_{2,b}^+X_{3,c}^+a_n X_{2,b+n}^+ \\
&=: &  I_n(b,c) + q^{-2} I_n'(b,c) \ \ \ \ \mathrm{with}\ I_n'(b,c) = \sum_{s+t = n, s,t \geq 1}.  
\end{eqnarray*}
Here, by  relations $X_{2,b+s}^+X_{3,c}^+ = q X_{3,c}^+X_{2,b+s}^+ + qX_{2,b+s-1}^+X_{3,c+1}^+ - X_{3,c+1}^+X_{2,b+s-1}^+$, we get
\begin{eqnarray*}
I_n'(b,c) &=&  \sum_{s+t=n,s,t\geq 1} a_s a_t(q  X_{3,c}^+ X_{2,b+s}^+X_{2,b+t}^+ + qX_{2,b+s-1}^+X_{3,c+1}^+X_{2,b+t}^+ - X_{3,c+1}^+X_{2,b+s-1}^+X_{2,b+t}^+)  \\
&=& \sum_{s+t=n, s,t \geq 1} a_{s-1}a_t X_{2,b+s-1}^+X_{3,c+1}^+X_{2,b+t}^+ - a_1 a_{n-1} X_{3,c+1}^+X_{2,b}^+X_{2,b+n-1}^+ \\
&=& I_{n-1}'(b,c+1) + a_0 a_{n-1} X_{2,b}^+X_{3,c+1}^+X_{2,b+n-1}^+ - a_1a_{n-1} X_{3,c+1}^+X_{2,b}^+X_{2,b+n-1}^+,  \\
I_n(b,c) &=& a_n(q^{-1} X_{2,b}^+ X_{2,b+n}^+X_{3,c}^+ - X_{2,b}^+X_{3,c}^+ X_{2,b+n}^+ + q^{-2}X_{2,b+n}^+X_{3,c}^+X_{2,b}^+ - q^{-1}X_{3,c}^+X_{2,b+n}^+X_{2,b}^+ ) \\
&=& a_n (X_{2,b}^+X_{2,b+n-1}^+X_{3,c+1}^+ - q^{-1}X_{2,b}^+ X_{3,c+1}^+X_{2,b+n-1}^+ \\
&\ & + q^{-1 }X_{2,b+n-1}^+ X_{3,c+1}^+ X_{2,b}^+ -q^{-2}X_{3,c+1}^+ X_{2,b+n-1}^+X_{2,b}^+ )  \\
&=&  I_{n-1}(b,c+1) +  a_{n-1} ((1-q^{-2})X_{2,b}^+X_{3,c+1}^+X_{2,b+n-1}^+ + (q^{-1}-q^{-3})X_{3,c+1}^+X_{2,b+n-1}^+X_{2,b}^+ )
\end{eqnarray*}
where we used that $a_s = q^{-1} a_{s-1}$. Henceforth,
\begin{eqnarray*}
\lambda(n,b,c) =  I_{n}(b,c) + q^{-2} I'_n(b,c) =  I_{n-1}(b,c+1) + q^{-2} I_{n-1}'(b,c+1) = \lambda(n-1,b,c+1). 
\end{eqnarray*}
Since $\lambda(0,b,c) = 0$ for all $b,c \in \mathbb{Z}$, we conclude that $\lambda(n,b,c) = 0$  and $R_{b,c}^+(z) = 0$. 
\paragraph{Case $i=2$.} We need to verify that $\tilde{R}_{a,c}(z,w) := \sum\limits_{b,d \in \mathbb{Z}} [\tilde{R}(a,b,c,d), X_{2,0}^-] z^b w^d = 0 \in {V}[[z,z^{-1},w,w^{-1}]]$ for all $a,c \in \mathbb{Z}$. Similar to the case $i=1$, we can express 
\begin{displaymath}
\tilde{R}_{a,c}(z,w) = \frac{1}{q-q^{-1}}( \tilde{R}_{a,c}^{++}(z,w) + \tilde{R}_{a,c}^{+-}(z,w) + \tilde{R}_{a,c}^{-+}(z,w) + \tilde{R}_{a,c}^{--}(z,w) )
\end{displaymath}
with $\tilde{R}^{ij} \in V[[z^i,w^j]]$ for $i,j \in \{\pm \}$. We need to prove that $\tilde{R}^{ij} = 0$. For simplicity, consider only the case 
\begin{eqnarray*}
\tilde{R}^{++}_{a,c}(z,w) &:= & R_{a,c}^{++}(z,w) + R_{a,c}^{++}(w,z) \in {V}[[z,w]]\ \  \mathrm{where} \\
R_{a,c}^{++}(z,w) &:= & - X_{1,a}^+ \phi_2^+(z)X_{3,c}^+(\sum_{d \geq 0} X_{2,d}^+ w^d) + X_{1,a}^+(\sum_{b \geq 0}X_{2,b}^+ z^b) X_{3,c}^+\phi_2^+(w) \\
&\ &  - X_{3,c}^+ \phi_2^+(z)X_{1,a}^+(\sum_{d \geq 0} X_{2,d}^+ w^d) + X_{3,c}^+(\sum_{b \geq 0}X_{2,b}^+ z^b) X_{1,a}^+\phi_2^+(w)  \\
&\ &  - \phi_2^+(z) X_{1,a}^+ (\sum_{d \geq 0}X_{2,d}^+ w^d)X_{3,c}^+ + (\sum_{b \geq 0}X_{2,b}^+z^b) X_{3,c}^+\phi_2^+(w)X_{1,a}^+ \\
&\ &  - \phi_2^+(z) X_{3,c}^+ (\sum_{d \geq 0}X_{2,d}^+ w^d)X_{1,a}^+ + (\sum_{b \geq 0}X_{2,b}^+z^b) X_{1,a}^+\phi_2^+(w)X_{3,c}^+  \\
&\ & +(q+q^{-1}) \phi_2^+(z)X_{1,a}^+X_{3,c}^+(\sum_{d \geq 0}X_{2,d}^+w^d) - (q+q^{-1}) (\sum_{b \geq 0}X_{2,b}^+z^b) X_{1,a}^+X_{3,c}^+ \phi_2^+(w) \in {V}[[z,w]]
\end{eqnarray*}
with $\phi_2^+(z) = K_2 \exp((q-q^{-1})\sum\limits_{s\geq 1} h_{2,s} z^s ) = K_2 h_2(z)$. Observe that there  exists a formal series $G_{a,c}(z,w) \in {V}[[z,w]]$ with $\tilde{R}_{a,c}^{++}(z,w) = G_{a,c}(z,w) K_2 h(z) + G_{a,c}(w,z) K_2 h_2(w) \in {V}[[z,w]]$.
More precisely, we have
\begin{eqnarray*}
G_{a,c}(z,w) &=& -q X_{1,a}^+ (X_{3,c}^+ + \sum_{s\geq 1}b_s X_{3,c+s}^+z^s ) (\sum_{d \geq 0} X_{2,d}^+w^d) -(q+q^{-1}) (\sum_{b \geq 0} X_{2,b}^+ w^b)X_{1,a}^+X_{3,c}^+   \\
&\ & -q^{-1} X_{3,c}^+ (X_{1,a}^+ + \sum_{s\geq 1}a_s X_{1,a+s}^+z^s ) (\sum_{d \geq 0} X_{2,d}^+w^d) + X_{3,c}^+(\sum_{b \geq 0} X_{2,b}^+ w^b) X_{1,a}^+ \\
&\ & -(X_{1,a}^+ + \sum_{s \geq 1} a_s X_{1,a+s}^+ z^s) (\sum_{d \geq 0}X_{2,d}^+w^d)(X_{3,c}^+ + \sum_{t \geq 1}b_t X_{3,c+t}^+z^t)  \\
&\ & -(X_{3,c}^+ + \sum_{s \geq 1} b_s X_{3,c+s}^+ z^s) (\sum_{d \geq 0}X_{2,d}^+w^d)(X_{1,a}^+ + \sum_{t \geq 1}a_t X_{1,a+t}^+z^t) + X_{1,a}^+(\sum_{b \geq 0} X_{2,b}^+ w^b) X_{3,c}^+\\
&\ & + q(\sum_{b \geq 0}X_{2,b}^+w^b) X_{1,a}^+ (X_{3,c}^+ + \sum_{t \geq 1}b_t X_{3,c+t}^+ z^t) + q^{-1}(\sum_{b \geq 0}X_{2,b}^+w^b) X_{3,c}^+(X_{1,a}^+ + \sum_{s\geq 1}a_s X_{1,a+s}^+ z^s) \\
&\ & +(q+q^{-1}) (X_{1,a}^+ + \sum_{s \geq 1} a_s X_{1,a+s}^+ z^s) (X_{3,c}^+ + \sum_{t \geq 1}b_t X_{3,c+t}^+z^t)(\sum_{d \geq 0}X_{2,d}^+w^d) \\
&=: & \sum_{n,d \geq 0} \mu(a,c,n,d) z^nw^d \ \ \ \ \ \ \ \mathrm{with}\ a_s = q^{-s}-q^{-s+2},\ b_s = q^s - q^{s-2}.
\end{eqnarray*}
It is clear that $\mu(a,c,0,d) = 0$. In general, for $n \geq 1$
\begin{eqnarray*}
\mu(a,c,n,d) &=& -q b_n X_{1,a}^+ X_{3,c+n}^+X_{2,d}^+ - q^{-1}a_n X_{3,c}^+ X_{1,a+n}^+X_{2,d}^+ + q b_n X_{2,d}^+X_{1,a}^+X_{3,c+n}^+  \\
&\ & - b_n X_{1,a}^+X_{2,d}^+X_{3,c+n}^+ - a_n X_{1,a+n}^+X_{2,d}^+X_{3,c}^+ - a_n X_{3,c}^+X_{2,d}^+X_{1,a+n}^+ - b_n X_{3,c+n}^+X_{2,d}^+X_{1,a}^+ \\
&\ &+ q^{-1} a_n X_{2,d}^+X_{3,c}^+X_{1,a+n}^+ + (q+q^{-1}) \sum_{s+t=n, s,t \geq 1} a_s b_t X_{1,a+s}^+X_{3,c+t}^+X_{2,d}^+  \\
&\ & + (q+q^{-1})b_n X_{1,a}^+X_{3,c+n}^+X_{2,d}^+ + (q+q^{-1})a_n X_{1,a+n}^+X_{3,c}^+X_{2,d}^+ \\
&\ & - \sum_{s+t=n, s,t \geq 1} a_s b_t X_{1,a+s}^+X_{2,d}^+X_{3,c+t}^+ - \sum_{s+t=n,s,t \geq 1} a_sb_t X_{3,c+t}^+ X_{2,d}^+X_{1,a+s}^+ \\
&=: &  A(a,c,n,d) + B(a,c,n,d) \ \ \  \mathrm{with}\\
A(a,c,n,d) &= &q^{-1}b_n X_{1,a}^+X_{3,c+n}^+X_{2,d}^+ + q a_n X_{1,a+n}^+X_{3,c}^+X_{2,d}^+ + q b_n X_{2,d}^+ X_{1,a}^+X_{3,c+n}^+ + q^{-1} a_n X_{2,d}^+X_{3,c}^+X_{1,a+n}^+ \\
&\ &- b_n X_{1,a}^+X_{2,d}^+ X_{3,c+n}^+ - a_n X_{3,c}^+X_{2,d}^+X_{1,a+n}^+ - b_n X_{3,c+n}^+X_{2,d}^+ X_{1,a}^+ - a_n X_{1,a+n}^+X_{2,d}^+X_{3,c}^+ \\
&=& b_n X_{1,a}^+ X_{3,c+n-1}^+ X_{2,d+1}^+ + a_n  X_{3,c}^+ X_{1,a+n-1}^+X_{2,d+1}^+ + b_n  X_{2,d+1}^+X_{3,c+n-1}^+X_{1,a}^+\\
&\ & + a_n X_{2,d+1}^+X_{1,a+n-1}^+X_{3,c}^+  - q^{-1} b_n X_{1,a}^+ X_{2,d+1}^+X_{3,c+n-1}^+ - qa_n X_{3,c}^+X_{2,d+1}^+X_{1,a+n-1}^+ \\
&\ &- qb_n X_{3,c+n-1}^+X_{2,d+1}^+X_{1,a}^+ - q^{-1}a_n X_{1,a+n-1}^+X_{2,d+1}^+X_{3,c}^+ \\
&=& A(a,c,n-1,d+1) + (q-q^{-1})(b_{n-1} X_{1,a}^+X_{3,c+n-1}^+X_{2,d+1}^+ - a_{n-1}  X_{3,c}^+X_{1,a+n-1}^+ X_{2,d+1}^+)  \\
&\ &+ (1-q^2)b_{n-1} X_{3,c+n-1}^+X_{2,d+1}^+X_{1,a}^+ + (1 -q^{-2})a_{n-1}X_{1,a+n-1}^+X_{2,d+1}^+X_{3,c}^+
\end{eqnarray*}
where the second equality comes from Drinfel'd relations of degree 2 between $X_{2,a}^+$ and $X_{i,b}^+$ for $i = 1,3$.
\begin{eqnarray*}
B(a,c,n,d) &=& \sum_{s+t=n, s,t \geq 1} a_s b_t \{ q^{-1}X_{1,a+s}^+(X_{3,c+t}^+ X_{2,d}^+ - qX_{2,d}^+X_{3,c+t}^+) \\
&\ & + qX_{3,c+t}^+(X_{1,a+s}^+X_{2,d}^+ - q^{-1}X_{2,d}^+X_{1,a+s}^+) \} \\
&=& \sum_{s+t=n, s,t \geq 1} a_s b_t \{ X_{1,a+s}^+(X_{3,c+t-1}^+ X_{2,d+1}^+ - q^{-1}X_{2,d+1}^+X_{3,c+t-1}^+)\\
&\ & + X_{3,c+t}^+(X_{1,a+s-1}^+X_{2,d+1}^+ - qX_{2,d+1}^+X_{1,a+s-1}^+) \} \\
&=& B(a,c,n-1,d+1) + (q-q^{-1})(a_{n-1} X_{1,a+n-1}^+X_{3,c}^+X_{2,d+1}^+ -b_{n-1}X_{1,a}^+X_{3,c+n-1}^+X_{2,d+1}^+) \\
&\ & \ -(1-q^{-2})a_{n-1} X_{1,a+n-1}^+X_{2,d+1}^+X_{3,c}^+ - (1-q^2)b_{n-1} X_{3,c+n-1}^+X_{2,d+1}^+X_{1,a}^+, \\
\mu(a,c,n,d) &=& A(a,c,n,d) + B(a,c,n,d)\\
 &=& A(a,c,n-1,d+1) + B(a,c,n-1,d+1) = \mu(a,c,n-1,d+1).
\end{eqnarray*}
We get  $\mu(a,c,n,d) = 0$ for all $a,c,n,d$, i.e. $G_{a,c}(z,w) = 0$. Therefore $\tilde{R}^{++}_{a,c}(z,w) = 0$, as desired.

\section{Quantum brackets and coproduct formulae}
Recall that we have a morphism of superalgebras $\Phi: \Uc \longrightarrow \Ud$ (Theorem \ref{thm: two presentations}). In this appendix, we will write $h_{i,\pm 1}$ for $1 \leq i \leq M+N-1$ as products in the $\Phi(E_i^{\pm})$ and deduce their coproduct formulae. These formulae have been used in the proofs of Proposition \ref{prop: coproduct formulas Drinfel'd} and  Lemma \ref{lem: evaluation on Drinfel'd generators}.  For simplicity, let $E_0^{\pm} := \Phi(E_0^{\pm}) \in (\Ud)_{\mp (\alpha_1+\cdots+\alpha_{M+N-1})}$. Recall that $\Ud$ is $Q_{M,N}$-graded.
\begin{notation}
(1) Recall that $Q_{M,N} = \bigoplus\limits_{i=1}^{M+N-1} \BZ \alpha_i$ and the parity map $p \in \hom_{\BZ}(Q_{M,N}, \super)$ given by $p(\alpha_i) = \odd$ for $i = M$ and $\even$ for $i \neq M$. Endow $Q_{M,N}$ with a bilinear form: $(\alpha_i, \alpha_j) := (\epsilon_i - \epsilon_{i+1}, \epsilon_j - \epsilon_{j+1})$.
 
(2) Let $A$ be a $Q_{M,N}$-graded algebra. (We mainly consider $\Ud, \Ud^{\stensor 2}$.) For $u \in (A)_{\alpha}, v \in (A)_{\beta}$, define the quantum bracket
\begin{displaymath}
\lfloor u, v \rfloor := u v - (-1)^{p(\alpha)p(\beta)} q^{-(\alpha, \beta)} v u = [u,v]_{q^{-(\alpha, \beta)}} \in (A)_{\alpha + \beta}.
\end{displaymath}
For $u_1, u_2, \cdots, u_n \in A$, let $\lfloor u_1, u_2, \cdots, u_n \rfloor := \lfloor u_1, \lfloor u_2, \cdots, \lfloor u_{n-1}, u_n \rfloor \cdots \rfloor \rfloor$,  with the convention that $\lfloor u \rfloor = u$ (brackets from right to left). 
\end{notation} 

 An induction argument on $i$ shows that 
\begin{eqnarray*}
&&\lfloor X_{i,0}^+, X_{i+1,0}^+, \cdots, X_{M+N-1,0}^+, E_0^+ \rfloor = -(\prod_{j=2}^{i-1} -q_j^{-1}) \lfloor\lfloor \cdots \lfloor X_{1,1}^-, X_{2,0}^- \rfloor, \cdots \rfloor, X_{i-1,0}^- \rfloor (K_1 K_2 \cdots K_{i-1})^{-1}, \\
&&\lfloor X_{2,0}^+, X_{3,0}^+, \cdots, X_{M+N-1,0}^+, E_0^+ \rfloor = - X_{1,1}^- K_1^{-1},\ \ h_{1,1} = -\lfloor X_{1,0}^+, X_{2,0}^+, \cdots, X_{M+N-1,0}^+, E_0^+ \rfloor.
\end{eqnarray*}

Remark that $[h_{1,1}, X_{2,0}^-] = (-1)^{p(\alpha_1)} X_{2,1}^-$. Hence,
\begin{eqnarray*}
X_{2,1}^- = (-1)^{p(\alpha_1)} [\lfloor X_{2,0}^+, X_{3,0}^+, \cdots, X_{M+N-1,0}^+, E_0^+ \rfloor, X_{2,0}^- ] = (-1)^{p(\alpha_2)} \lfloor X_{1,0}^+, X_{3,0}^+, \cdots, X_{M+N-1,0}^+, E_0^+ \rfloor K_2
\end{eqnarray*}
where we have used that $[X_{i,0}^+, X_{2,0}^-] = 0 = [E_0^+, X_{2,0}^-]$ for $i \neq 2$.  From $[X_{2,0}^+, X_{2,1}^-] = K_2 h_{2,1}$ we get
\begin{displaymath}
h_{2,1} =  (-1)^{p(\alpha_1)} \lfloor X_{2,0}^+, X_{1,0}^+, X_{3,0}^+, \cdots, X_{M+N-1,0}^+, E_0^{+} \rfloor.
\end{displaymath}
Again an induction argument on $i$ shows that for $1 \leq i \leq M+N-1$
\begin{equation}  \label{equ: h_{i,1}}
h_{i,1} = \lambda_i \lfloor X_{i,0}^+, X_{i-1,0}^+, \cdots, X_{1,0}^+, X_{i+1,0}^+, \cdots, X_{M+N-1,0}^+, E_0^+ \rfloor \ \ \ \mathrm{where}\ \lambda_i = \begin{cases}
(-1)^{i} & i \leq M \\
(-1)^{i-1} & i > M
\end{cases}.
\end{equation}
Next, we compute $\Delta(h_{1,1})$ modulo $U (X^-)^2 \stensor U (X^+)^2$. By definition
\begin{displaymath}
\Delta (h_{1,1}) = - \lfloor \Delta X_{1,0}^+, \cdots, \Delta X_{M+N-1,0}^+, 1 \stensor E_0^+ \rfloor - \lfloor \Delta X_{1,0}^+, \cdots, \Delta X_{M+N-1,0}^+, E_0^+ \stensor (K_1\cdots K_{M+N-1}) \rfloor. 
\end{displaymath} 
On the other hand, remark that $-\lfloor \Delta X_{1,0}^+, \cdots, \Delta X_{M+N-1,0}^+, 1 \stensor E_0^+ \rfloor =  1 \stensor h_{1,1}$ since 
\begin{displaymath}
\lfloor X_{i,0}^+ \stensor K_i^{-1}, X_{i+1,0}^+ \stensor K_{i+1}^{-1}, \cdots, X_{M+N-1,0}^+ \stensor K_{M+N-1,0}^{-1}, 1 \stensor E_0^+   \rfloor = 0
\end{displaymath}
for $1 \leq i \leq M+N-1$, and modulo $U (X^-)^2 \stensor U (X^+)^2$, 
\begin{displaymath}
\Delta h_{1,1} = 1 \stensor h_{1,1} + h_{1,1} \stensor 1 - \sum_{i=1}^{M+N-1} \lfloor \cdots, X_{i-1,0}^+ \stensor K_{i-1}^{-1},1 \stensor X_{i,0}^+, X_{i+1,0}^+ \stensor K_{i+1}^{-1},  \cdots,  E_0^+ \stensor (K_1\cdots K_{M+N-1}) \rfloor.
\end{displaymath}
When $i \geq 3$, the corresponding term in the above summation becomes $0$. Similar argument leads to the first part of Lemma \ref{lem: coproduct Drinfel'd zero generators} for $\Delta (h_{i,1})$.

\begin{notation}
Let $A$ be a $Q_{M,N}$-graded algebra. For $u \in (A)_{\alpha}, v \in (A)_{\beta}$, define the quantum bracket
\begin{displaymath}
\lceil u, v \rceil := u v - (-1)^{p(\alpha)p(\beta)} q^{(\alpha, \beta)} v u = [u,v]_{q^{(\alpha,\beta)}} \in (A)_{\alpha + \beta}.
\end{displaymath} 
For $u_1, u_2, \cdots, u_n \in A$, let $\lceil u_1,u_2,\cdots,u_n \rceil := \lceil \cdots \lceil u_1, u_2 \rceil, \cdots, u_n \rceil$ (brackets from left to right).
\end{notation}
Using the second type of quantum brackets, we obtain 
\begin{displaymath}
\lceil E_0^-, X_{M+N-1,0}^-, \cdots, X_{i+1,0}^- \rceil = (K_1 \cdots K_i) \lfloor\lfloor \cdots\lfloor X_{1,-1}^+, X_{2,0}^+\rfloor, \cdots\rfloor, X_{i,0}^+ \rfloor.
\end{displaymath}
and $h_{1,-1} = K_1^{-1}[X_{1,-1}^+, X_{1,0}^-] = \lceil E_0^-, X_{M+N-1,0}^-, \cdots, X_{1,0}^- \rceil$. Similar to the case $h_{i,1}$, we have
\begin{equation}   \label{equ: h_{i,-1}}
h_{i,-1} = - \lambda_i \lceil E_0^-, X_{M+N-1,0}^-, \cdots, X_{i+1,0}^-, X_{1,0}^-, \cdots, X_{i,0}^- \rceil,
\end{equation}
and the second part of Lemma \ref{lem: coproduct Drinfel'd zero generators} is easily deduced.

\

{\footnotesize \noindent Huafeng ZHANG:  Universit\'{e} Paris Diderot - Paris 7, Institut de Math\'{e}matiques de Jussieu - Paris Rive Gauche CNRS UMR 7586,  B\^{a}timent Sophie Germain, Case 7012, 75205 Paris Cedex 13, France

\noindent Email: zhangh@math.jussieu.fr }

\addcontentsline{toc}{section}{References}
\begin{thebibliography}{99}  \footnotesize
\bibitem[AYY11]{AYY}{\bf S. Azam, H. Yamane, M. Yousofzadeh}, {\it Classification of finite dimensional irreducible representations of generalized quantum groups via Weyl groupoids}, {preprint, arXiv 1105.0160 (2011)}
\bibitem[BK01]{Hall}{\bf P. Baumann, C. Kassel}, {\it The Hall algebra of the category of coherent sheaves on the projective line}, {J. reine angew. Math. \textbf{533}, 207-233 (2001)}
\bibitem[Be94a]{Beck1}{\bf J. Beck}, {\it Braid group action and quantum affine algebras}, {Comm. Math. Phys. \textbf{165}, 555-568 (1994)}
\bibitem[Be94b]{Beck2}{\bf J. Beck}, {\it Convex bases of PBW type for quantum affine algebras}, {Comm. Math. Phys. \textbf{165}, 193-199 (1994)}
\bibitem[Bei06]{Bei}{\bf N. Beisert}, {\it The S-matrix of AdS/CFT and Yangian symmetry}, {PoS Solvay \textbf{002} (2006) }
\bibitem[BGM12]{Beisert2}{\bf N. Beisert, W. Galleas, T. Matsumoto}, {\it A quantum affine algebra for the deformed Hubbard chain}, {J. Phys. \textbf{A 45}, 365206 (2012)}
\bibitem[BK08]{Beisert1}{\bf N. Beisert, P. Koroteev}, {\it Quantum deformations of the one-dimensional Hubbard model}, {J. Phys. \textbf{A 41}, 255204 (2008)}
\bibitem[BT08]{BT}{\bf V. Bazhanov, Z. Tsuboi}, {\it Baxter's $\textbf{Q}$-operators for supersymmetric spin chains}, {Nucl. Phys. \textbf{B 805}, 451-516 (2008)}
\bibitem[BKK00]{BKK}{\bf G. Benkart, S-J. Kang, M. Kashiwara}, {\it Crystal bases for the quantum superalgebra $U_q(\mathfrak{gl}(m,n))$}, {J. Amer. Math. Soc. \textbf{13}, 295-331 (2000)}
\bibitem[Ch86]{Chari}{\bf V. Chari}, {\it Integrable representations of affine Lie-algebras}, {Invent. Math. \textbf{85}, 317-335 (1986) }
\bibitem[CH10]{CH}{\bf V. Chari, D. Hernandez}, {\it Beyond Kirillov-Reshetikhin modules}, {Contemp. Math. \textbf{506}, 49-81 (2010) }
\bibitem[CL06]{CL}{\bf V. Chari, S. Loktev}, {\it Weyl, Demazure and fusion modules for the current algebra of $\mathfrak{sl}_{r+1}$}, {Adv. Math. \textbf{207}, 928-960 (2006)}
\bibitem[CP91]{CP1}{\bf V. Chari, A. Pressley}, {\it Quantum affine algebras}, {Commun. Math. Phys. \textbf{142}, 261-283 (1991)}
\bibitem[CP95]{CP2}{\bf V. Chari, A. Pressley}, {\it Quantum affine algebras and their representations}, {Canadian Math. Soc. Conf. Proc. \textbf{16}, 59-78 (1995)}
\bibitem[CP01]{CP3}{\bf V. Chari, A. Pressley}, {\it Weyl modules for classical and quantum affine algebras}, {Represent. Theory \textbf{5}, 191-223 (2001, electronic)}
\bibitem[Gr07]{double}{\bf P. Gross\'{e}}, {\it On quantum shuffle and quantum affine algebras}, {J. Algebra \textbf{318}(2), 495-519 (2007)}
\bibitem[GT13]{GT}{\bf S. Gautam, V. Toledano Laredo}, {\it Yangians and quantum loop algebras}, {Sel. Math. New Ser. \textbf{19}, 271-336 (2013)}
\bibitem[He05]{He}{\bf D. Hernandez}, {\it Representations of quantum affinizations and fusion product}, {Transform. Groups \textbf{10}, no.2, 163-200 (2005)}
\bibitem[HRZ08]{Rosso}{\bf N. Hu, M. Rosso, H. Zhang}, {\it Two-parameter quantum affine algebra $U_{r,s}(\widehat{\mathfrak{sl}_n})$, Drinfel'd realization and quantum affine Lyndon basis}, {Commun. Math. Phys. \textbf{278}(2), 453-486 (2008)}
\bibitem[HSTY08]{groupoid}{\bf I. Heckenberger, F. Spill, A. Torrielli, H. Yamane}, {\it Drinfeld second realization of the quantum affine superalgebras of $D^{(1)}(2,1;x)$ via the Weyl groupoid}, {Publ. RIMS. Kyoto \textbf{B 8}, 171-206 (2008)}
\bibitem[HW12]{Wang}{\bf D. Hill, W. Wang}, {\it Categorification of quantum Kac-Moody superalgebras}, {preprint, arXiv 1202.2769 (2012)}
\bibitem[Ja96]{Jantzen}{\bf J.C. Jantzen}, {\it Lectures on quantum groups}, {Graduate Studies in Mathematics \textbf{6}, AMS (1996)}
\bibitem[JM99]{Jimbo}{\bf M. Jimbo, T. Miwa}, {\it Algebraic analysis of solvable lattice models}, {CBMS Regional Conference Series in Mathematics \textbf{85}, AMS (1999)}
\bibitem[Ka77]{Kac1}{\bf V.G. Kac}, {\it Lie superalgebras}, {Adv. Math. \textbf{26}, 8-96 (1977) }
\bibitem[Ka78]{Kac}{\bf V.G. Kac}, {\it Representations of classical Lie superalgebras}, {Differential Geometrical Methods in Mathematical Physics II, Lecture Notes in Mathematics \textbf{676}, 597-626 (1978)}
\bibitem[KKNV12]{KKNV}{\bf N.C. Kien, N.A. Ky, L.B. Nam, N.T.H. Van}, {\it Representations of quantum superalgebra $U_q(\mathfrak{gl}(2|1))$ in a coherent state basis and generalization}, {J. Math. Phys. \textbf{52}, 123512 (2012)} 
\bibitem[KKO13]{Kashiwara}{\bf S-J. Kang, M. Kashiwara, S-J. Oh}, {\it Supercategorification of quantum Kac-Moody algebras II}, {preprint, arXiv 1303.1916 (2013)} 
\bibitem[Ko13]{Kojima}{\bf T. Kojima}, {\it Diagonalization of transfer matrix of supersymmetry $U_q(\widehat{\mathfrak{sl}}(M+1|N+1))$ chain with a boundary}, {J. Math. Phys. \textbf{54}, 043507 (2013)}
\bibitem[KSU97]{boson}{\bf K. Kimura, J. Shiraishi, J. Uchiyama}, {\it A level-one representation of the quantum affine superalgebra $U_q(\widehat{\mathfrak{sl}}(M+1|N+1))$}, {Comm. Math. Phys. \textbf{188} (2), 367-378 (1997)}
\bibitem[KS95]{KS}{\bf N.A. Ky, N.I. Stoilova}, {\it Finite-dimensional representations of the quantum superalgebra $U_q(\mathfrak{gl}(2|2))$, II. Nontypical representations at generic $q$}, {J. Math. Phys. \textbf{36}, 5979-6003 (1995)}
\bibitem[KT91]{KT1}{\bf S.M. Khoroshkin, V.N. Tolstoy}, {\it Universal $R$-matrix for quantized (super)algebras}, {Commun. Math. Phys. \textbf{141}, 599-617 (1991)}
\bibitem[KT92]{KT}{\bf S.M. Khoroshkin, V.N. Tolstoy}, {\it The uniqueness theorem for the universal $R$-matrix}, {Lett. Math. Phys. \textbf{24}, 231-244 (1992)}
\bibitem[KV04]{KV}{\bf N.A. Ky, N.T.H. Van}, {\it Finite-dimensional representations of $U_q(\mathfrak{gl}(2|1))$ in a basis of $U_q(\mathfrak{gl}(2) \oplus \mathfrak{gl}(1))$}, {Advances in Natural Sciences, \textbf{2}, p. 1 (2004), arXiv 0305195} 
\bibitem[Ky94]{Ky}{\bf N.A. Ky}, {\it Finite-dimensional representations of the quantum superalgebra $U_q(\mathfrak{gl}(2|2))$. I. Typical representations at generic $q$}, {J. Math. Phys. \textbf{35}, 2583-2606 (1994)}
\bibitem[Lu93]{Lusztig}{\bf G. Lusztig}, {\it Introduction to quantum groups}, {Progress in  Mathematics \textbf{110}, Birkh\"{a}user (1993)}
\bibitem[Mu12]{PBW}{\bf I.M. Musson}, {\it Lie superalgebras and enveloping algebras}, {Graduate Studies in Mathematics \textbf{131}, AMS (2012)}
\bibitem[PT91]{PT}{\bf T.D. Palev, V.N. Tolstoy}, {\it Finite-dimensional irreducible representations of the quantum superalgebra $U_q(\mathfrak{gl}(N|1))$}, {Commun. Math. Phys. \textbf{141} (3), 549-558 (1991)}
\bibitem[PSV94]{PSV}{\bf T.D. Palev, N.I. Stoilova, J. Van der Jeugt}, {\it Finite-dimensional representations of the quantum superalgebra $U_q(\mathfrak{gl}(n|m))$ and related $q$-identities}, {Commun. Math. Phys. \textbf{166} (2), 367-378 (1994) }
\bibitem[Ra13]{Rao}{\bf S.E. Rao}, {\it Finite dimensional modules for multiloop superalgebras of types $A(m,n)$ and $C(m)$}, {Proc. Amer. Math. Soc. \textbf{141}, 3411-3419 (2013)}
\bibitem[Ros]{Rosso2}{\bf M. Rosso}, {\it Lyndon bases and the multiplicative formula for $R$-matrices}, {preprint}
\bibitem[Sc79]{Sch}{\bf M. Scheunert}, {\it The theory of Lie superalgebras, an introduction}, {Lecture Notes in Mathematics \textbf{716}, Springer, Berlin (1979) }
\bibitem[Se11]{Ser}{\bf V. Serganova}, {\it Kac-Moody superalgebras and integrability}, {Developments and Trends in Infinite-Dimensional Lie Theory, Progress in Mathematics \textbf{288}, Birkh\"{a}user Boston, Inc., Boston, MA, 169-218 (2011)}
\bibitem[Ya94]{Yam1}{\bf H. Yamane}, {\it Quantized enveloping algebras associated with simple Lie superalgebras and their universal $R$-matrices}, {Publ. RIMS. Kyoto Univ. \textbf{30}, 15-87 (1994)}
\bibitem[Ya99]{Yam2}{\bf H. Yamane}, {\it On defining relations of the affine Lie superalgebras and their quantized universal enveloping superalgebras}, {Publ. RIMS. Kyoto Univ. \textbf{35}, 321-390 (1999)}
\bibitem[Zh92]{Zh1}{\bf R.B. Zhang}, {\it Universal $L$ operator and invariants of the quantum supergroup $U_q(\mathfrak{gl}(m|n))$}, {J. Math. Phys. \textbf{33}, 1970-1979 (1992)}
\bibitem[Zh93]{Zh2}{\bf R.B. Zhang}, {\it Finite dimensional irreducible representations of the quantum supergroup $U_q(\mathfrak{gl}(m|n))$}, {J. Math. Phys. \textbf{34}, 1236-1254 (1993)}
\bibitem[Zh96]{Zh3}{\bf R.B. Zhang}, {\it The $\mathfrak{gl}(M|N)$ super Yangian and its finite dimensional representations}, {Lett. Math. Phys. \textbf{37}, 419-434 (1996)}
\bibitem[Zh97]{Zh4}{\bf R.B. Zhang}, {\it Symmetrizable quantum affine superalgebras and their representations}, {J. Math. Phys. \textbf{38}, 535-543 (1997)}
\end{thebibliography}
\end{document}